\documentclass[a4paper,12pt]{amsart}
\usepackage{amssymb}

\def\hbar{\bar{h}}

\def\iso{\buildrel \sim\over\to}

\def\GS{{\mathfrak{S}}}

\def\Gb{{\mathfrak{b}}}
\def\Gg{{\mathfrak{g}}}

\def\Gm{{\mathfrak{m}}}
\def\Gn{{\mathfrak{n}}}

\def\Gsl{{\mathfrak{sl}}}

\def\CA{{\mathcal{A}}}
\def\CB{{\mathcal{B}}}
\def\CC{{\mathcal{C}}}

\def\CF{{\mathcal{F}}}

\def\CH{{\mathcal{H}}}

\def\CL{{\mathcal{L}}}
\def\CM{{\mathcal{M}}}
\def\CN{{\mathcal{N}}}
\def\CO{{\mathcal{O}}}
\def\CP{{\mathcal{P}}}

\def\CU{{\mathcal{U}}}
\def\CV{{\mathcal{V}}}
\def\CW{{\mathcal{W}}}

\def\BC{{\mathbf{C}}}

\def\BF{{\mathbf{F}}}

\def\BN{{\mathbf{N}}}

\def\BP{{\mathbf{P}}}
\def\BQ{{\mathbf{Q}}}

\def\BZ{{\mathbf{Z}}}

\def\eps{\varepsilon}

\def\Aut{\operatorname{Aut}\nolimits}

\def\can{{\mathrm{can}}}

\def\diag{\operatorname{diag}\nolimits}

\def\End{\operatorname{End}\nolimits}

\def\Ext{\operatorname{Ext}\nolimits}

\def\GL{\operatorname{GL}\nolimits}

\def\gr{{\operatorname{gr}\nolimits}}

\def\mgr{{\operatorname{\!-gr}\nolimits}}

\def\Gr{{\operatorname{Gr}\nolimits}}

\def\Hom{\operatorname{Hom}\nolimits}

\def\Id{\operatorname{Id}\nolimits}
\def\id{\operatorname{id}\nolimits}
\def\im{\operatorname{im}\nolimits}

\def\Ind{\operatorname{Ind}\nolimits}

\def\mMod{\operatorname{\!-mod}\nolimits}

\def\mMOD{\operatorname{\!-Mod}\nolimits}

\def\opp{{\operatorname{opp}\nolimits}}

\def\mproj{\operatorname{\!-proj}\nolimits}

\def\rad{\operatorname{rad}\nolimits}

\def\rk{\operatorname{rk}\nolimits}

\def\Rep{\operatorname{Rep}\nolimits}

\def\Res{\operatorname{Res}\nolimits}

\def\soc{\operatorname{soc}\nolimits}

\def\Stab{\operatorname{Stab}\nolimits}

\def\ie{{\em i.e.}}

\def\tS{{\tilde{S}}}

\def\tgamma{{\tilde{\gamma}}}

\def\CHom{{{\mathcal H}om}}
\def\CEnd{{{\mathcal E}nd}}

\def\idun{\mathbf{1}}

\newtheorem{thm}{Theorem}[section]
\newtheorem{lemma}[thm]{Lemma}
\newtheorem{cor}[thm]{Corollary}
\newtheorem{prop}[thm]{Proposition}

\theoremstyle{definition}
\newtheorem{rem}[thm]{Remark}
\newtheorem{example}[thm]{Example}

\usepackage{yfonts}
\usepackage{youngtab}
\input xypic
\xyoption{all}
\xyoption{poly}

\def\FA{{\textgoth{A}}}
\def\FB{{\textgoth{B}}}
\def\FC{{\textgoth{C}}}

\title{Quiver Hecke algebras and $2$-Lie algebras}
\author{Rapha\"el Rouquier}
\address{Mathematical Institute,
University of Oxford, 24-29 St Giles', Oxford, OX1 3LB, UK
and Department of Mathematics, UCLA, Box 951555,
Los Angeles, CA 90095-1555, USA}
\email{rouquier@maths.ox.ac.uk}
\date{9 December 2011}

\begin{document}
\maketitle
\tableofcontents

\section{Introduction}
This text provides an introduction and complements to some basic constructions
and results in $2$-representation theory of Kac-Moody algebras.
We discuss quiver Hecke algebras \cite{Rou3}, which have
been introduced independently by Khovanov and
Lauda \cite{KhoLau1} and \cite{KhoLau2}, and their cyclotomic versions,
which have been considered independently in the case of level $2$ weights for
type $A$, by Brundan and Stroppel \cite{BrStr}.
We discuss the $2$-categories associated with Kac-Moody algebras
and their $2$-representations: this has been introduced in joint work with
Joe Chuang \cite{ChRou} for $\Gsl_2$ and implicitly for type $A$ (finite
or affine). While the general philosophy of categorifications was older
(cf for example \cite{BeFrKho}), the new idea in \cite{ChRou} was to
introduce some structure at the level of natural transformations: an
endomorphism
of $E$ and an endomorphism of $E^2$ satisfying Hecke-type relations.
The generalization to other types is based on 
quiver Hecke algebras, which account for a half Kac-Moody algebra. We
discuss the geometrical construction of the quiver
Hecke algebras via quiver varieties, which was our starting point for the
definition of quiver Hecke algebras, and that of cyclotomic quiver Hecke
algebras.

\smallskip
The first chapter gives a gentle introduction to nil (affine)
Hecke algebras of type $A$. We
recall basic properties of Hecke algebras of symmetric groups and provide
the construction via BGG-Demazure operators of the nil Hecke algebras.
We also construct symmetrizing forms.

The second chapter is devoted to quiver Hecke algebras. We explain
that the more complicated relation involved in the definition is
actually a consequence of the other ones, up to polynomial torsion: this
leads to a new, simpler, definition of quiver Hecke algebras.
We construct next the faithful polynomial representation.
This generalizes the constructions of
the first chapter, that correspond to a quiver with one vertex. We
explain the relation, for type $A$ quivers, with affine Hecke algebras.
Finally, we explain how to put together all quiver Hecke algebras associated
with a quiver to obtain a monoidal category that categorifies a
half Kac-Moody algebra (and its quantum version).

The third chapter introduces $2$-categories associated with
Kac-Moody algebras and discusses their integrable representations. We
provide various results that reduce the amount of conditions to check
that a category is endowed with a structure of an integrable
$2$-representation, once
the quiver Hecke relations hold: for example, the $\Gsl_2$-relations imply
all other relations, and it can be enough to check them on $K_0$.
We explain the universal construction of ``simple'' $2$-representations, 
and give a detailed description for $\Gsl_2$. We present a
Jordan-H\"older type result. We move next to cyclotomic quiver Hecke algebras,
and present Kang-Kashiwara and Webster's construction of $2$-representations on
cyclotomic quiver Hecke algebras. We prove that the $2$-representation
is equivalent to the universal simple $2$-representation.
Finally, we explain the construction of Fock spaces from representations
of symmetric groups in this framework.

The last chapter brings in geometrical methods available in the
case of symmetric Kac-Moody algebras. We start with a brief
recollection of Ringel's construction of quantum groups via Hall algebras and
Lusztig's construction of enveloping algebras via constructible functions.
We move next to the construction of nil affine Hecke algebras in the
cohomology of flag varieties. We introduce Lusztig's category
of perverse sheaves on the moduli space of representations of a quiver
and show that it is equivalent to the monoidal category of quiver Hecke
algebras (a result obtained independently by Varagnolo and Vasserot). As
a consequence, the indecomposable projective modules for quiver
Hecke algebras over a field of characteristic $0$, and for ``geometric''
parameters, correspond to the canonical basis.
Finally, we show that Zheng's microlocalized categories of sheaves
can be endowed with a structure of $2$-representation isomorphic to the
universal simple $2$-representation. As a consequence, 
the indecomposable projective modules for cyclotomic quiver
Hecke algebras over a field of characteristic $0$, and for ``geometric''
parameters, correspond to the canonical basis of simple representations.

\smallskip
This article is based on a series of lectures at the National Taiwan
University, Taipei, in December 2008 and
a series of lectures at BICMR, Peking University, in March--April 2010. I wish
to thank Professors Shun-Jen Cheng and Weiqiang Wang,
and Professor Jiping Zhang for their invitations to give
these lecture series.

\section{One vertex quiver Hecke algebras}
The results of this section are classical (cf for example
\cite[\S 3]{Rou3}).
\subsection{Nil Hecke algebras}
\subsubsection{The symmetric group as a Weyl group}

Let $n\ge 1$. Given $i\in\{1,\ldots,n-1\}$, we put $s_i=(i,i+1)\in\GS_n$.

We define a function $r:\GS_n\to\BZ_{\ge 0}$.
Given $w\in\GS_n$, let $R_w=\{(i,j) | i<j \text{ and } w(i)>w(j)\}$ and
let $r(w)=|R(w)|$ be the number of inversions.

The {\em length} $l(w)$ of $w\in\GS_n$ is the minimal integer $r$
such that there exists $i_1,\ldots,i_r$ with
$w=s_{i_1}\cdots s_{i_r}$. Such an expression is called a {\em reduced 
decomposition} of $w$. Proposition \ref{pr:lengthequalsr}
says that since $s_1,\ldots,
s_{n-1}$ generate $\GS_n$, these notions make sense.

\smallskip
Note that reduced decompositions are not unique: we have for example
$(13)=s_1s_2s_1=s_2s_1s_2$. Simpler is $s_1s_3=s_3s_1$.

\begin{prop}
\label{pr:lengthequalsr}
The set $\{s_1,\ldots,s_{n-1}\}$ generates $\GS_n$. Given $w\in \GS_n$, we
have $r(w)=l(w)$.
\end{prop}

\begin{proof}
Let $w\in \GS_n$, $w\not=1$. Note that $R_w{\not=}\emptyset$.
Consider $(i,j)\in R_w$ such that $j-i$ is minimal. Assume
$j\not=i+1$. By the minimality assumption, $(i,i+1){\not\in}R_w$ and
$(i+1,j){\not\in}R_w$, so $w(j)>w(i+1)>w(i)$, a contradiction.
So, $j=i+1$. Let $w'=ws_i$. We have
$R_{w'}=R_w-\{(i,i+1)\}$, hence $r(w')=r(w)-1$. We deduce by induction
that there exist $i_1,\ldots,i_{r(w)}\in\{1,\ldots,n-1\}$ such that
$w=s_{i_{r(w)}}\cdots s_{i_1}$. In particular, the set
$\{s_1,\ldots,s_{n-1}\}$ generates $\GS_n$ and $l(g)\le r(g)$ for
all $g\in\GS_n$.

Let $j\in\{1,\ldots,n\}$ and $v=ws_j$. Assume $(j,j+1){\not\in}R_w$.
Then, $R_v=R_w\cup\{(j,j+1)\}$. It follows that $r(v)=r(w)+1$. If
$(j,j+1)\in R_w$, then $r(v)=r(w)-1$.
We deduce by induction that $l(g)\ge r(g)$ for all $g\in\GS_n$.
\end{proof}

\begin{prop}
The element $w[1,n]=(1,n)(2,n-1)(3,n-2)\cdots$ is the unique element of
$\GS_n$ with maximal length. We have
$l(w[1,n])=\frac{n(n-1)}{2}$.
\end{prop}

\begin{proof}
Note that $R_{w[1,n]}=\{(i,j) | i<j\}$ and this contains any set
$R_w$ for $w\in\GS_n$, with equality if and only if $w=w[1,n]$.
The result follows from Proposition \ref{pr:lengthequalsr}.
\end{proof}

The set 
$$C_n=\{1,s_{n-1}=(n-1,n),s_{n-2}s_{n-1}=(n-2,n-1,n),\ldots,
s_1\cdots s_{n-2}s_{n-1}=(1,2,\ldots,n)\}$$
is a complete set of representatives for left cosets $\GS_n/\GS_{n-1}$.
Let $w\in\GS_{n-1}$ and $g\in C_n$. We
have $R(gw)=R(w)\coprod R(g)$, so $l(gw)=l(g)+l(w)$. 
Consider now $w\in\GS_n$. There is a unique decomposition
$w=c_n c_{n-1}\cdots c_2$ where $c_i\in C_i$ and we have
$l(w)=l(c_n)+\cdots+l(c_2)$. Each $c_i$ has a unique
reduced decomposition and that provides us with a canonical reduced
decomposition of $w$:
$$w=(s_{j_1} s_{j_1+1}\cdots s_{i_1})
(s_{j_2} s_{j_2+1}\cdots s_{i_2})
\cdots (s_{j_r} s_{j_r+1}\cdots s_{i_r})$$
where $i_1>i_2>\cdots >i_r$ and $1\le j_r<i_r$.

In the case of the longest element, we obtain
$$w[1,n]=(s_1\cdots s_{n-1})(s_1\cdots s_{n-2})\cdots s_1
=(s_1\cdots s_{n-1})w[1,n-1].$$

\smallskip
Using canonical reduced decompositions, we can count the number of elements
with a given length and deduce the following result.

\begin{prop}
\label{pr:Poincare}
We have $\sum_{w\in\GS_n}q^{l(w)}=\frac{(1-q)(1-q^2)\cdots(1-q^n)}{(1-q)^n}$.
\end{prop}

\begin{lemma}
Let $w\in\GS_n$. Then, $l(w^{-1})=l(w)$ and
$l(w[1,n]w^{-1})=l(w[1,n])-l(w)$.
\end{lemma}

\begin{proof}
The first statement is clear, since $w=s_{i_1}\cdots s_{i_r}$ is a reduced
expression if and only if $w^{-1}=s_{i_r}\cdots s_{i_1}$ is a reduced
expression.

We have $R_{w[1,n]w}=\{(i,j) | (i<j) \text{ and } w(i)<w(j)\}$. The second
statement follows.
\end{proof}

We recall the following classical result.

\begin{prop}
\label{pr:presentationSn}
The group $\GS_n$ has a presentation with generators $s_1,\ldots,s_{n-1}$
and relations
$$s_is_j=s_js_i \text{ if }|i-j|>1 \text{ and }s_is_{i+1}s_i=s_{i+1}s_i
s_{i+1}.$$
\end{prop}

\subsubsection{Finite Hecke algebras}
Let us recall some classical results about Hecke algebras of symmetric
groups.

Let $R=\BZ[q_1,q_2]$.
Let $H_n^f$ be the {\em Hecke algebra} of $\GS_n$: this
is the $R$-algebra generated by $T_1,\ldots,T_{n-1}$,
with relations
$$T_iT_{i+1}T_i=T_{i+1}T_iT_{i+1},\ \ T_iT_j=T_jT_i \text { if }
|i-j|>1 \text { and } (T_i-q_1)(T_i-q_2)=0.$$

There is an isomorphism of algebras
$$H_n^f\otimes_R R/(q_1-1,q_2+1)\iso \BZ[\GS_n],\ T_i\mapsto s_i.$$

\smallskip
Let $w\in\GS_n$ with a
reduced decomposition $w=s_{i_1}\cdots s_{i_r}$.
We put $T_w=T_{i_1}\cdots T_{i_r}\in H_n^f$.
One shows that $T_w$ is independent of the choice of a reduced
decomposition of $w$ and that $\{T_w\}_{w\in\GS_n}$ is an $R$-basis of
the free $R$-module $H_n^f$.

Given $w,w'\in\GS_n$ with $l(ww')=l(w)+l(w')$, we have
$T_w T_{w'}=T_{ww'}$.

\smallskip
 The algebra $H_n^f$ is a {\em deformation}
of $\BZ[\GS_n]$. At the specialization $q_1=1$, $q_2=-1$, the element
$T_w$ becomes the group element $w$.

\subsubsection{Nil Hecke algebras of type $A$}
Let ${^0H}_n^f=H_n^f\otimes_R R/(q_1,q_2)$. 
Given $w,w'\in\GS_n$, we have 
$$T_wT_{w'}=
\begin{cases}
T_{ww'} &\text{ if }l(ww')=l(w)+l(w')\\
0 & \text{ otherwise.}
\end{cases}$$
So, the algebra ${^0H}_n^f$ is graded with $\deg T_w=-2l(w)$. The choice of
a negative sign will become clear soon. The factor $2$ comes from
the cohomological interpretation.

\smallskip
Given $M=\bigoplus_{i\in\BZ}M_i$ a graded $\BZ$-module and $r\in\BZ$, we
denote by
$M\langle r\rangle$ the graded module given by
$(M\langle r\rangle)_i=M_{i+r}$.

\smallskip
We have $({^0H}_n^f)_i=\bigoplus_{w\in\GS_n,l(w)=-i/2}
\BZ T_w$. So, $({^0H}_n^f)_i=0$ unless $i\in\{0,-2,\ldots,-n(n-1)\}$.

\smallskip
Let $k$ be a field and $k{^0H}_n^f={^0H}_n^f\otimes_\BZ k$.

\begin{prop}
\label{pr:radsocnilHecke}
The Jacobson radical of $k{^0H}_n^f$ is
$\rad(k{^0H}_n^f)=\bigoplus_{w\not=1}kT_w$ and $k{^0H}_n^f$ has a unique
minimal non-zero two-sided ideal $\soc(k{^0H}_n^f)=kT_{w[1,n]}$.
The trivial module $k$, with $0$-action of the $T_i$'s, is
the unique simple $k{^0H}_n^f$-module.
\end{prop}

\begin{proof}
Let $A=k{^0H}_n^f$.
Let $J=A_{<0}=\bigoplus_{w\not=1}kT_w$.
We have $J^{n(n-1)+1}=0$. So,
$J$ is a nilpotent two-sided ideal of
$A$ and $A/J\simeq k$. It follows that
$J=\mathrm{rad}(A)$: the algebra $A$ is local and $k$ is the unique
simple module.

Let $M$ be a non-zero left ideal of $A$. Let $m=\sum_w \alpha_w T_w\in M$
be a non-zero element. Consider $w\in \GS_n$ of minimal length
such that $\alpha_w\not=0$. We have 
$T_{w[1,n]w^{-1}}m=\alpha_w T_{w[1,n]}$ because
$T_{w[1,n]w^{-1}}T_{w'}=0$ if $l(w')\ge l(w)$ and $w'\not=w$. It follows
that $kT_{w[1,n]}\subset M$. That shows that $kT_{w[1,n]}$ is the unique
minimal non-zero left ideal of $A$. A similar proof shows it is also the
unique minimal non-zero right ideal.
\end{proof}

\begin{rem}
Let $k$ be a field and $A$ be a finite dimensional graded $k$-algebra.
Assume $A_0=k$ and $A_i=0$ for $i>0$. Then $\rad(A)=A_{<0}$.
\end{rem}

\subsubsection{BGG-Demazure operators}
We refer to \cite[Chapter IV]{Hi} for a general discussion of the results
below.

Let $P_n=\BZ[X_1,\ldots,X_n]$. We let $\GS_n$ act on $P_n$ by permutation
of the $X_i$'s.
We
define an endomorphism of abelian groups
$\partial_i\in\End_\BZ(P_n)$ by
$$\partial_i(P)=\frac{P-s_i(P)}{X_{i+1}-X_i}.$$
Note that the operators $\partial_w$ are $P_n^{\GS_n}$-linear.
Note also that $\im\partial_i\subset P_n^{s_i}=\ker\partial_i$.

\smallskip
The following lemma follows from easy calculations.

\begin{lemma}
\label{le:demazurebraid}
We have
$\partial_i^2=0$, $\partial_i\partial_j=\partial_j\partial_i$ for
$|i-j|>1$ and
$\partial_i\partial_{i+1}\partial_i=
\partial_{i+1}\partial_i\partial_{i+1}$.
\end{lemma}

We deduce that we have obtained a representation of the nil Hecke algebra.
\begin{prop}
The assignment $T_i\mapsto \partial_i$ defines a representation of
${^0H}_n^f$ on $P_n$.
\end{prop}

Define a grading of the algebra $P_n$ by $\deg X_i=2$. Then,
the representation above is compatible with the gradings.

Given $w\in\GS_n$, we denote by $\partial_w$ the image of $T_w$.

\smallskip
The following result is clear.

\begin{lemma}
Let $P\in P_n$. 
We have $\partial_i(P)=0$ for all $i$ if and only if
$P\in P_n^{\GS_n}$.
\end{lemma}

 If $M$ is a free graded module over a commutative ring $k$
with $\dim_k M_i<\infty$ for all $i\in\BZ$,
we put $\mathrm{grdim}(M)=\sum_{i\in\BZ}q^{i/2} \dim(M_i)$.

\begin{thm}
\label{pr:basispol}
The set $\{\partial_w(X_2X_3^2\cdots X_n^{n-1})\}_{w\in\GS_n}$
is a basis of $P_n$ over $P_n^{\GS_n}$.
\end{thm}

\begin{proof}
Let us show by induction on $n$ that
$$\partial_{w[1,n]}(X_2X_3^2\cdots X_n^{n-1})=1.$$
We have
$w[1,n]=s_{n-1}\cdots s_1 w[2,n]$ and
$l(w[1,n])=l(w[2,n])+n-1$. By induction,
$$\partial_{w[2,n]}(X_2X_3^2\cdots X_n^{n-1})=
X_2\cdots X_n\cdot \partial_{w[2,n]}(X_3\cdots X_n^{n-2})=
X_2\cdots X_n.$$
On the other hand,
$\partial_{n-1}\cdots\partial_1(X_2\cdots X_n)=1$ and we deduce that
$\partial_{w[1,n]}(X_2X_3^2\cdots X_n^{n-1})=1$.

\smallskip
Let $M$ be a free $P_n^{\GS_n}$-module with basis
$\{b_w\}_{w\in\GS_n}$, with
$\deg b_w=2l(w[1,n]w^{-1})=n(n-1)-2l(w)$. 
Define a morphism of $P_n^{\GS_n}$-modules
$$\phi:M\to
P_n,\ b_w\mapsto \partial_w(X_2X_3^2\cdots X_n^{n-1}).$$
This is a graded morphism.

Let $k$ be a field.
Let $a=\sum_w Q_w b_w\in\ker (\phi\otimes k)$, where 
$Q_w\in k[X_1,\ldots,X_n]^{\GS_n}$.
Assume $a\not=0$. Consider $v\in\GS_n$ with
$Q_v\not=0$ and such that $l(v)$ is minimal with this property.
We have $\partial_{w[1,n]v^{-1}}(\phi(a))=Q_v$, hence we have
a contradiction. It follows that $\phi\otimes k$ is injective.

We have $\mathrm{grdim}P_n=(1-q)^{-n}$.
On the other hand, we have
$P_n^{\GS_n}=\BZ[e_1,\ldots,e_n]$, where
$e_r=e_r(X_1,\ldots,X_n)=\sum_{1\le i_1<\cdots<i_r\le n}X_{i_1}\cdots X_{i_r}$.
So, $\mathrm{grdim} P_n^{\GS_n}=(1-q)^{-1}(1-q^2)^{-1}\cdots
(1-q^n)^{-1}$. We deduce that 
$$\mathrm{grdim}M=(1-q)^{-1}\cdots (1-q^n)^{-1}\sum_{w\in\GS_n}q^{l(w)}.$$
The formula of Proposition \ref{pr:Poincare} 
shows that $\mathrm{grdim}M=\mathrm{grdim} P_n$.
Lemma \ref{le:Nakgraded} shows that $\phi_i\otimes k$ is an isomorphism and
then Lemma \ref{le:Nakp} shows that $\phi_i$ is an isomorphism for all $i$.
\end{proof}

The following two lemmas are clear.

\begin{lemma}
\label{le:Nakp}
Let $f:M\to N$ be a morphism between free finitely generated
$\BZ$-modules. If $f\otimes_\BZ (\BZ/p)$ is surjective for all prime $p$, then
$f$ is surjective.
\end{lemma}

\begin{lemma}
\label{le:Nakgraded}
Let $k$ be a field and $M,N$ be two graded $k$-modules with
$\dim M_i=\dim N_i$ finite for all $i$. If $f:M\to N$ is an injective
morphism of graded $k$-modules, then $f$ is an isomorphism.
\end{lemma}

\begin{rem}
Note that $\{X_2^{a_2}\cdots X_n^{a_n}\}_{0\le a_i\le i-1}$ is the more
classical basis of $P_n$ over $P_n^{\GS_n}$.
\end{rem}

\subsection{Nil affine Hecke algebras}
\subsubsection{Definition}
Let ${^0H}_n$ be the {\em nil affine Hecke algebra} of $\GL_n$:
this is the $\BZ$-algebra with generators $X_1,\ldots,X_n,
T_1,\ldots,T_{n-1}$ and relations
$$X_iX_j=X_jX_i,\ T_i^2=0,\ T_iT_{i+1}T_i=T_{i+1}T_iT_{i+1},\
T_iT_j=T_jT_i \text{ if } |i-j|>1,$$
$$T_i X_j=X_j T_i \text{ if } j-i\not=0,1,\ 
T_iX_{i+1}-X_iT_i=1 \text{ and }
T_iX_i-X_{i+1}T_i=-1.$$
It is a graded algebra, with $\deg X_i=2$ and $\deg T_i=-2$.

The following lemma is easy.

\begin{lemma}
\label{le:Leibniz}
Given $P,Q\in P_n$, we have
$\partial_i(PQ)=\partial_i(P)Q+s_i(P)\partial_i(Q)$.
\end{lemma}

Lemma \ref{le:Leibniz} is the key ingredient to prove the following
lemma.

\begin{lemma}
\label{le:polynomialrep}
We have a representation $\rho$ of ${^0H}_n$ on $P_n$ given by
$$\rho(T_i)(P)=\partial_i(P)\text{ and } \rho(X_i)(P)=X_iP.$$
\end{lemma}

\begin{prop}
\label{pr:faithfulnilaffineHecke}
We have a decomposition
${^0H}_n= P_n\otimes {^0H}_n^f$ as a $\BZ$-module and
the representation of ${^0H}_n$ on $P_n$ is faithful.
\end{prop}

\begin{proof}
Let $\{P_w\}_{w\in \GS_n}$ be a family of elements of 
$P_n$. Let $a=\sum_w P_wT_w$.
If $a\not=0$, there is $w\in\GS_n$ of minimal
length such that $P_w\not=0$. We have $aT_{w^{-1}w[1,n]}=
P_wT_{w[1,n]}$ and
$$\rho(a)(\partial_{w^{-1}w[1,n]}(X_2\cdots X_n^{n-1}))
=P_w \partial_{w[1,n]}(X_2\cdots X_n^{n-1})=P_w$$
(cf Proof of Proposition \ref{pr:basispol}). We deduce that $\rho(a)\not=0$.
Consequently, the multiplication map
$P_n\otimes {^0H}_n^f\to {^0H}_n$ is injective and
the representation is faithful. On the other hand, the multiplication map
is easily seen to be surjective.
\end{proof}

Note that $P_n$ and ${^0H}_n^f$ are subalgebras of 
${^0H}_n$.

The module $P_n$ is an induced module:
we have an isomorphism of ${^0H}_n$-modules
$$P_n\iso {^0H}_n\otimes_{{^0H}_n^f}\BZ,\ 
P\mapsto P\otimes 1.$$

\begin{rem}
\label{re:generalcommutation}
Given $P\in P_n$, one shows that
$T_i P-s_i(P)T_i=PT_i -T_is_i(P)=\partial_i(P)$.
\end{rem}

\subsubsection{Description as a matrix ring}

Let $b_n=T_{w[1,n]}X_2X_3^2\cdots X_n^{n-1}$. 

\begin{lemma}
\label{le:bnidempotent}
We have $b_n^2=b_n$ and ${^0H}_n={^0H}_n b_n {^0H}_n$.
\end{lemma}

\begin{proof}
Note that $T_{w[1,n]}$ is the unique element of $\End_{P_n^{\GS_n}}(P_n)$
that sends $X_2X_3^2\cdots X_n^{n-1}$ to $1$ and
$\partial_w(X_2X_3^2\cdots X_n^{n-1})$ to $0$ for $w\not=1$
(cf Proof of Theorem \ref{pr:basispol} for the first fact).
We have
$$\rho(T_{w[1,n]}X_2X_3^2\cdots X_n^{n-1}T_{w[1,n]})(
\partial_w(X_2X_3^2\cdots X_n^{n-1}))=0$$
for $w\not=1$ and
$$\rho(T_{w[1,n]}X_2X_3^2\cdots X_n^{n-1}T_{w[1,n]})(
X_2X_3^2\cdots X_n^{n-1})=\partial_{w[1,n]}(X_2X_3^2\cdots X_n^{n-1})=1.$$
It follows that
$b_nT_{w[1,n]}=T_{w[1,n]}$.

\medskip
We show now by induction on $n$ that $1\in {^0H}_n T_{w[1,n]} {^0H}_n$.
Given $1\le r\le n-1$, we have
$$T_r\cdots T_{n-1} T_{w[1,n-1]}X_n-X_rT_r\cdots T_{n-1} T_{w[1,n-1]}=
T_{r+1}\cdots T_{n-1} T_{w[1,n-1]},$$
where we use the convention that
$T_{r+1}\cdots T_{n-1}=\prod_{r+1\le j\le n-1}T_j=1$ if $r=n-1$.
By induction on $r$, we deduce that
$T_{w[1,n-1]}\in {^0H}_n T_{w[1,n]} {^0H}_n$, since
$T_{w[1,n]}=T_1\cdots T_{n-1} T_{w[1,n-1]}$. By induction on
$n$, it follows that $1\in {^0H}_n T_{w[1,n]} {^0H}_n={^0H}_n b_n {^0H}_n$.
\end{proof}

\begin{rem}
Given $w\in\GS_n$ and $P\in P_n$, one shows that
$T_w P T_{w[1,n]}=\partial_w(P) T_{w[1,n]}$ (a particular case
was obtained in the proof of Lemma \ref{le:bnidempotent}).
\end{rem}

We have an isomorphism of ${^0H}_n$-modules
$$P_n\iso {^0H}_n b_n,\ P\mapsto Pb_n$$

This shows that $P_n$ is a progenerator as a ${^0H}_n$-module: it is
a finitely generated projective ${^0H}_n$-module and ${^0H}_n$ is a
direct summand of a multiple of $P_n$, as a ${^0H}_n$-module.

Given $A$ a ring, we denote by $A^\opp$ the opposite ring: it is $A$ as an
abelian group, but the multiplication of $a$ and $b$ in $A^\opp$ is
the product $ba$ computed in $A$.

\begin{prop}
\label{pr:MoritanilHecke}
The action of ${^0H}_n$ on $P_n$ induces an isomorphism of 
$P_n^{\GS_n}$-algebras
$${^0H}_n\iso \End_{P_n^{\GS_n}}(P_n)^{\opp}.$$
 Since $P_n$ is a free
$P_n^{\GS_n}$-module
of rank $n!$, the algebra ${^0}H_n$ is isomorphic to a
$(n!\times n!)$-matrix algebra over $P_n^{\GS_n}$.
\end{prop}

\begin{proof}
Since $P_n$ is a progenerator for ${^0H}_n$, we deduce that the canonical map
${^0H}_n\to \End_{P_n^{\GS_n}}(P_n)$
is a split injection of
$P_n^{\GS_n}$-modules (Lemma \ref{le:criterionsplitting}).
The proposition follows from the fact that
${^0H}_n$ is a free $P_n^{\GS_n}$-module of rank $(n!)^2$.
\end{proof}

\begin{lemma}
\label{le:criterionsplitting}
Let $R$ be a commutative ring and $A$ an $R$-algebra, projective and
finitely generated as an $R$-module.
Let $M$ be a progenerator for $A$.
Then, the canonical map
$A\to \End_R(M)^{\opp}$ is a split injection of $R$-modules.
\end{lemma}

\begin{proof}
Let $f:A\to \End_R(M)$ be the canonical map and $L$ its cokernel.
The composition  of morphisms of $R$-modules
$$M\xrightarrow{m\mapsto \id\otimes m} \End_R(M)\otimes_AM
\xrightarrow{\alpha\otimes m\mapsto \alpha(m)} M$$
is the identity. So, $f\otimes_A 1_M$ is a split injection of
$R$-modules, hence
$L\otimes_A M$ is a projective $R$-module, since
$\End_R(M)$ is a projective $R$-module and $M$ is a projective 
$A$-module.

By Morita
theory, there is $N$ an $(\End_A(M),A)$-bimodule that is projective
as an $\End_A(M)$-module and such that $M\otimes_{\End_A(M)}N\simeq A$
as $(A,A)$-bimodules. The $R$-module
$L\simeq (L\otimes_A M)\otimes_{\End_A(M)}N$ is projective, since
$N$ is a projective $\End_A(M)$-module. Since $L$ is a projective
$R$-module, we deduce that $f$ is a split injection of $R$-modules.
\end{proof}

\smallskip
We give now a second proof of Proposition \ref{pr:MoritanilHecke}.
The proof of the faithfulness of the representation $P_n$ of
${^0H}_n$ works also to show that $P_n\otimes_{P_n^{\GS_n}}(P_n^{\GS_n}/\Gm)$
is a faithful representation of
${^0H}_n\otimes_{P_n^{\GS_n}}(P_n^{\GS_n}/\Gm)$, for
any maximal ideal $\Gm$ of $P_n^{\GS_n}$. Proposition
\ref{pr:MoritanilHecke} follows now from Lemma \ref{le:isoatclosedpoints}.

\begin{lemma}
\label{le:isoatclosedpoints}
Let $R$ be a commutative ring, $f:M\to N$ a morphism between free
$R$-modules of the same finite rank. If $f\otimes_R 1_{R/\Gm}$ is injective for
every maximal ideal $\Gm$ of $R$, then $f$ is an isomorphism.
\end{lemma}

\begin{proof}
Fix bases of $M$ and $N$ and let $d$ be the determinant of $f$ with 
respect to those bases. Let $I$ be the ideal of $R$ generated by $d$. Assume
$d{\not\in} R^\times$. There
is a maximal ideal $\Gm$ of $R$ containing $I$. Since
$f\otimes_R 1_{R/\Gm}$ is an injective map between vector spaces of
the same finite rank, it is an isomorphism, so we have
$d\cdot 1_{R/\Gm}{\not=}0$, a contradiction.
\end{proof}

\subsection{Symmetrizing forms}
\subsubsection{Definition and basic properties}
\label{se:defsymmetrizing}
Let $R$ be a commutative ring and $A$ an $R$-algebra, finitely generated
and projective as an $R$-module. A {\em symmetrizing form} for $A$ is
an $R$-linear map $t\in\Hom_R(A,R)$ such that
\begin{itemize}
\item $t$ is a trace, \ie,
$t(ab)=t(ba)$ for all $a,b\in A$
\item the morphism
of $(A,A)$-bimodules
$$\hat{t}:A\to \Hom_R(A,R),\ a\mapsto (b\mapsto t(ab))$$
is an isomorphism.
\end{itemize}

\smallskip
Consider now a commutative ring $R'$ such that $R$ is an $R'$-algebra,
finitely generated and projective as an $R'$-module. 
Consider $t\in\Hom_R(A,R)$ a trace and $t'\in\Hom_{R'}(R,R')$. We
have a commutative diagram
$$\xymatrix{
A\ar[rr]^-{\hat{t}}\ar[drr]_-{\widehat{t't}} && \Hom_R(A,R)
\ar[rr]^-{\Hom_R(A,\hat{t}')} &&
\Hom_{R}(A,\Hom_{R'}(R,R')) \\
&& \Hom_{R'}(A,R')\ar[urr]_{\text{adjunction}}^-\sim
}$$
We deduce the following result.

\begin{lemma}
\label{le:2outof3}
If two of the forms $t$, $t'$ and $t't$ are symmetrizing, then so is
the third one.
\end{lemma}

Let now $B$ be another symmetric $R$-algebra and $M$ an $(A,B)$-bimodule,
finitely generated and projective as an $A$-module and as a right
$B$-module.
We have isomorphisms of functors 
$$\Hom_A(M,-)\xleftarrow[\sim]{\can}\Hom_A(M,A)\otimes_A -
\xrightarrow[\sim]{f\mapsto tf}\Hom_R(M,R)\otimes_A -$$
and similarly $\Hom_B(\Hom_R(M,R),-)\iso M\otimes_B-$.
We deduce that $M\otimes_B -$ is left and right adjoint to
$\Hom_A(M,-)$.

\subsubsection{Polynomials}
\begin{prop}
\label{pr:symmpolynomials}
The linear form $\partial_{w[1,n]}$ is
a symmetrizing form for the $P_n^{\GS_n}$-algebra
$P_n$.
\end{prop}

We will prove this proposition in \S \ref{se:symmnilaffine}:
it will be deduced from a
corresponding statement for the nil affine Hecke algebra, that is easier
to prove.

\smallskip
Together with Lemma \ref{le:2outof3}, Proposition \ref{pr:symmpolynomials}
provides more general symmetrizing forms.
\begin{cor}
\label{co:symmparab}
Given $1\le i\le n$, then the linear
form $\partial_{w[1,n]}\partial_{w[1,i]}\partial_{w[i+1,n]}$
is a symmetrizing form for the $P_n^{\GS_n}$-algebra
$P_n^{\GS\{1,\ldots,i\}\times\GS\{i+1,\ldots,n\}}$.
\end{cor}

\subsubsection{Nil Hecke algebras}
Define the $\BZ$-linear form $t_0:{^0H}_n^f\to\BZ$ by 
$t_0(T_w)=\delta_{w,w[1,n]}$. 

Define an algebra automorphism $\sigma$
(the Nakayama automorphism) 
of ${^0H}_n^f$ by $\sigma(T_i)=T_{n-i}$.
Note that $\sigma(T_w)=T_{w[1,n]ww[1,n]}$.

The form $t_0$ is not symmetric, it nevertheless gives rise to a 
Frobenius algebra structure.

\begin{prop}
Given $a,b\in {^0H}_n^f$, we have
$t_0(ab)=t_0(\sigma(b)a)$.
The form $t_0$ induces an isomorphism of right
${^0H}_n^f$-modules
$$\hat{t}_0:{^0H}_n^f\iso\Hom_{\BZ}({^0H}_n^f,\BZ), a\mapsto (b\mapsto
t_0(ab)).$$
\end{prop}

\begin{proof}
Let $w,w'\in\GS_n$. We have $t_0(T_wT_{w'})=0$ unless $w'=w^{-1}w[1,n]$, in
which case we have $t_0(T_w T_{w^{-1}w[1,n]})=1$. We have
$t_0(\sigma(T_{w'})T_w)=t_0(T_{w[1,n]w'w[1,n]}T_w)$. This is $0$, unless
$w=(w[1,n]w'w[1,n])^{-1}w[1,n]$, or equivalently, unless
$w=w[1,n]w^{\prime -1}$. In that case, we get
$t_0(\sigma(T_{w^{-1}w[1,n]},T_w)=1$. This shows that 
$t_0(T_w T_{w'})=t_0(\sigma(T_{w'})T_w)$ for all $w,w'\in\GS_n$.

Let $p$ be a prime number. The kernel $I$ of
$\hat{t}_0\otimes_\BZ \BF_p$ is a two-sided ideal of $\BF_p {^0H}_n^f$.
On the other hand, $\hat{t}_0(T_{w[1,n]})(1)=t_0(T_{w[1,n]})=1$, hence
$T_{w[1,n]}{\not\in}I$. It follows from Proposition \ref{pr:radsocnilHecke}
that $I=0$. Lemma
\ref{le:isoatclosedpoints} shows now that $\hat{t}_0$ is an isomorphism.
\end{proof}

\subsubsection{Nil affine Hecke algebras}
\label{se:symmnilaffine}
We define a $P_n^{\GS_n}$-linear form $t$ on ${^0H}_n$
$$t:{^0H}_n\to P_n^{\GS_n},\
t(PT_w)=\delta_{w,w[1,n]}\partial_{w[1,n]}(P) \text{ for }
P\in P_n \text{ and } w\in\GS_n.$$

Let $\gamma$ be the $\BZ$-algebra automorphism of ${^0H}_n$ defined by
$$\gamma(X_i)=X_{n-i+1} \text{ and }\gamma(T_i)=-T_{n-i}.$$

\begin{lemma}
\label{le:tracenilaffine}
We have $t(ab)=t(\gamma(b)a)$ for $a,b\in {^0H}_n$.
\end{lemma}

\begin{proof}
Let $i\in\{1,\ldots,n\}$.
By induction on $l(w)$, one shows that
$$T_wX_i-X_{w(i)}T_w\in\bigoplus_{w'\in\GS_n,l(w')<l(w)} P_nT_{w'}.$$
It follows that
$$T_{w[1,n]}X_i-X_{n-i+1}T_{w[1,n]}\in\bigoplus_{w{\not=}w[1,n]} P_nT_w.$$
We deduce that 
$t(PT_w X_i)=t(PX_{n-i}T_w)$ for $w\in\GS_n$ and $P\in P_n$.

Let $i\in\{1,\ldots,n-1\}$ and $P\in P_n$.
Remark \ref{re:generalcommutation} shows that
$$t(T_{n-i}PT_w)=t(s_{n-i}(P)T_{n-i}T_w)+t(\partial_{n-i}(P)T_w).$$
We have $\partial_{w[1,n]}(\partial_{n-i}(P))=0$, so
$t(\partial_{n-i}(P)T_w)=0$.
We have
$\partial_{w[1,n]}(P+s_{n-i}(P))=0$, hence
$\partial_{w[1,n]}(s_{n-i}(P))=-\partial_{w[1,n]}(P)$.
Since $s_{n-i}w=w[1,n]$ if and only if $ws_i=w[1,n]$, we deduce
that 
$t(s_{n-i}(P)T_{n-i}T_w)=-t(PT_wT_i)$.
This shows that $t(PT_wT_i)=t(-T_{n-i}PT_w)$ for $w\in\GS_n$ and $P\in P_n$.
The lemma follows.
\end{proof}

When viewed as a subalgebra of $\End_\BZ(P_n)$, then
${^0H}_n$ contains $\GS_n$, since the action of $\GS_n$ is trivial
on $P_n^{\GS_n}$ (Proposition \ref{pr:MoritanilHecke}). The injection of
$\GS_n$ in
${^0H}_n$ is given explicitly by $s_i\mapsto (X_i-X_{i+1})T_i+1$.

\smallskip
The following lemma is an immediate calculation involving endomorphisms
of $P_n$.
\begin{lemma}
\label{le:Nakayamainner}
We have
$w[1,n]\cdot a\cdot w[1,n]=\gamma(a) \text{ for all }a\in {^0H}_n$.
\end{lemma}

Let $t'$ be the linear form on ${^0H}_n$ defined by
$t'(a)=t(a w[1,n])$. 

\begin{prop}
\label{pr:symmnilaffine}
The form $t'$ is symmetrizing for the $P_n^{\GS_n}$-algebra ${^0H}_n$.
\end{prop}

\begin{proof}
Lemmas \ref{le:tracenilaffine} and \ref{le:Nakayamainner} show that
$t'(ab)=t'(ba)$ for all $a,b\in {^0H}_n$.

Let $\Gm$ be a maximal ideal of $P_n^{\GS_n}$ and $k=P_n^{\GS_n}/\Gm$.
We have $k{^0H}_n\simeq M_{n!}(k)$ by Proposition \ref{pr:MoritanilHecke}.
We have $t'(X_2\cdots X_n^{n-1}T_{w[1,n]} w[1,n])=1$ (cf the proof of
Theorem \ref{pr:basispol}), hence the form
$t'\otimes_{P_n^{\GS_n}}1_k$ is not zero. As a consequence, it
is a symmetrizing form, since
$k{^0H}_n\simeq M_{n!}(k)$ by Proposition \ref{pr:MoritanilHecke}.
We deduce that $\hat{t}'\otimes_{P_n^{\GS_n}}k$ is an isomorphism,
so $\hat{t}'$ is an isomorphism by Lemma \ref{le:isoatclosedpoints}.
\end{proof}

\begin{proof}[Proof of Proposition \ref{pr:symmpolynomials}]
Let $\Gm$ be a maximal ideal of $P_n^{\GS_n}$ and $k=P_n^{\GS_n}/\Gm$.
Let $P$ be a non-zero element of $P_n\otimes_{P_n^{\GS_n}}k$.
By Proposition \ref{pr:symmnilaffine}, there
is $a\in k{^0H}_n$ such that $t'(PT_{w[1,n]}a)\not=0$.
So, $t(P T_{w[1,n]}aw[1,n])\not=0$. There are
elements $Q_w\in P_n\otimes_{P_n^{\GS_n}}k$
such that $aw[1,n]=\sum_w T_w Q_w$. Then
$$t(PT_{w[1,n]}aw[1,n])=t(PT_{w[1,n]}Q_1)=t(\gamma(Q_1)PT_{w[1,n]})=
\partial_{w[1,n]}(\gamma(Q_1)P)\not=0.$$
We deduce that $\hat{\partial}_{w[1,n]}\otimes_{P_n^{\GS_n}}
(P_n^{\GS_n}/\Gm)$ is injective for any maximal ideal $\Gm$ of
$P_n^{\GS_n}$. It follows from Lemma \ref{le:isoatclosedpoints} that
$\hat{\partial}_{w[1,n]}$ is an isomorphism.
\end{proof}

\begin{rem}
One can show that the automorphism $\sigma$ of ${^0H}_n^f$ is not inner for
$n\ge 3$.
\end{rem}

\section{Quiver Hecke algebras}
\subsection{Representations of quivers}
We refer to \cite{GaRoi} for a general discussion of quivers and their
representations.

\subsubsection{Quivers}
\label{sectiongalcarquois}
Let $Q$ be a {\em quiver} (= an oriented graph), \ie,
\begin{itemize}
\item a finite set $Q_0$ (the vertices) 
\item a finite set $Q_1$ (the arrows)
\item maps $p,q:Q_1\to Q_0$ (tail=source and head=target of an arrow).
\end{itemize}

Let $k$ be a commutative ring. A {\em representation}
of $Q$ over $k$ is the data of
$(V_s,\phi_a)_{s\in Q_0, a\in Q_1}$ where $V_s$ is a $k$-module and
$\phi_a\in\Hom_k(V_{p(a)},V_{q(a)})$.

A morphism from
$(V_s,\phi_a)_{s,a}$ to $(V'_s,\phi'_a)_{s,a}$ is the data of a family
$(f_s)_{s\in Q_0}$, where $f_s\in \Hom_k(V_s,V'_s)$,
such that for all $a\in Q_1$, the following diagram commutes:
$$\xymatrix{
V_{p(a)} \ar[r]^{\phi_a} \ar[d]_{f_{p(a)}} & V_{q(a)} \ar[d]^{f_{q(a)}} \\
V'_{p(a)} \ar[r]_{\phi'_a} & V'_{q(a)}
}$$

\smallskip

The {\em quiver algebra} $k(Q)$ associated to $Q$ is
the $k$-algebra generated by the set $Q_0\cup Q_1$ with relations
$$sa=\delta_{q(a),s}a,\ \ as=\delta_{p(a),s}a,\ \ ss'=\delta_{s,s'}s\
\text{ for }s,s'\in Q_0 \text{ and }a\in Q_1\
\text{ and }\ 1=\sum_{t\in Q_0}t$$

Let $\gamma=(s_1,a_1,s_2,a_2,\ldots,s_n)$ be a {\em path}
in $Q$, \ie,
a sequence of vertices $s_i\in Q_0$ and arrows $a_i\in Q_1$ such that
$p(a_i)=s_i$ and $q(a_i)=s_{i+1}$.
We put $\tgamma=s_n a_{n-1}\cdots a_2 s_2 a_1 s_1\in k(Q)$.

\smallskip
The following proposition is clear.

\begin{prop}
The set of $\tgamma$, where $\gamma$ runs over the set of paths of $Q$,
is a basis of $k(Q)$.
\end{prop}

Note that $k(Q)$ is a graded algebra, with $Q_0$ in degree $0$ and $Q_1$ in
degree $1$. In general, a path of length $n$ is homogeneous of degree $n$.

\medskip
There is an equivalence from the category of representations of $Q$ to
the category of (left) $k(Q)$-modules: given $(V_s,\phi_a)$ a representation
of $Q$, let
$M=\bigoplus_s V_s$. We define a structure of $k(Q)$-module as
follows: $s\in Q_0$ is the projection onto $V_s$. An element
$a\in Q_1$ acts by zero on $\bigoplus_{s\not=p(a)}V_s$ and sends
$V_{p(a)}$ to $V_{q(a)}$ via $\phi_a$.

\smallskip
Assume $k$ is a field.
Given $s\in Q_0$, there is a simple representation $S=S(s)$ of $Q$ given by
$S_t=0$ for $t\not=s$, $S_s=k$ and $\phi_a=0$ for all $a\in Q_1$.
When $k(Q)$ is finite dimensional, we obtain all simple representations of
$Q$, up to isomorphism.

\begin{example}
\label{ex:examplesquivers}
For each of the following quivers, we give the list of finite dimensional
indecomposable
representations (up to isomorphism) and we indicate the isomorphism type
of the quiver algebra. We assume $k$ is a field.
\begin{itemize}
\item[(i)]
$\bullet$~: $(k)$. The quiver algebra is $k$.
\item[(ii)]
$\xymatrix{1\ar[r]&2}$~: $S(1)=(S_1=k,S_2=0,\phi=0)$, $S(2)=
(S_1=0,S_2=k,\phi=0)$
and $M=(M_1=k,M_2=k,\phi=1)$. The quiver algebra is isomorphic to the
algebra of $2\times 2$ upper triangular matrices.
\item[(iii)]
$\ \ \ \ \xymatrix{\bullet\ar@(ul,dl)}$~:
$(k^n,\phi(n,\lambda))_{n\ge 1,\lambda\in k}$
with $\phi(n,\lambda)=
\small{\left(\begin{matrix}
 \lambda & 1 \\
         & \ddots & \ddots \\
         &        & \ddots & 1\\ 
         &        &        & \lambda
\end{matrix}\right)}$,
assuming $k$ is an algebraically closed. The quiver algebra is isomorphic to
$k[x]$.
\end{itemize}
\end{example}

\subsubsection{Quivers with relations}
Let $Q$ be a quiver and $k$ a commutative ring. A set $R$ of relations for $Q$
over $k$ is a finite set of elements of $k(Q)_{\ge 2}$. We denote by
$I=(R)$ the two-sided ideal of $k(Q)$ generated by $R$ and we put $A=k(Q)/I$.

\begin{rem}
Assume $k$ is an algebraically closed field. Let $A$ be a basic finite
dimensional $k$-algebra (\ie, all simple $A$-modules have dimension $1$).
One shows that there is a quiver $Q$ with relations $R$ such that
$A\simeq k(Q)/(R)$. The vertices of $Q$ are in bijection with the set
of simple $A$-modules, up to isomorphism, while the set of arrows is
in bijection with a basis of $\rad(A)/\rad(A)^2$.
\end{rem}

\subsection{Quiver Hecke algebras}
We review some constructions and results of \cite[\S 3.2]{Rou3} and
provide some complements.
We will give three definitions of quiver Hecke algebras:
\begin{itemize}
\item By generators and relations, modulo polynomial torsion
\item By (more complicated) generators and relations
\item As a subalgebra of a ring of endomorphisms of a polynomial ring (over
a quiver)
\end{itemize}

\subsubsection{Wreath and nil-wreath products}
\label{se:wreath}
Let $k$ be a commutative ring.
Let $I$ be a finite set and $n\ge 0$. We define
a quiver $\Psi_{I,n}$ with vertex set $I^n$.
We will use the action of the symmetric group $\GS_n$ on $I^n$.
The arrows are
\begin{itemize}
\item $s_i=s_{i,v}:v\to s_i(v)$ for $1\le i\le n-1$ and $v\in I^n$
\item $x_i=x_{i,v}:v\to v$ for $1\le i\le n$ and $v\in I^n$.
\end{itemize}
We define the quiver algebra
$A=A(\Psi_{I,n},R_1)$ over $k$ with the quiver above and relations $R_1$:
$$s_i^2=1,\ s_is_j=s_js_i \text{ if }|i-j|>1,\
s_is_{i+1}s_i=s_{i+1}s_is_{i+1}$$
$$x_ix_j=x_jx_i,\ s_ix_j=x_js_i \text{ if }j\not=i,i+1 \text{ and }
s_i x_i=x_{i+1}s_i.$$

When $|I|=1$, we have $A=k[x_1,\ldots,x_n]\rtimes\GS_n=k[x]\wr\GS_n$. In
general, $A=k[x]^I\wr \GS_n$.

\medskip
We can construct a similar algebra, based on the nil Hecke algebra instead
of $k\GS_n$. We use $T_i$ to denote the arrow called $s_i$ earlier.
We define a quiver algebra
$A'=A(\Psi_{I,n},R_0)$ with relations $R_0$:
$$T_i^2=0,\ T_iT_j=T_jT_i \text{ if }|i-j|>1,\
T_iT_{i+1}T_i=T_{i+1}T_iT_{i+1}$$
$$x_ix_j=x_jx_i,\ T_ix_j=x_jT_i \text{ if }j\not=i,i+1,
T_i x_i=x_{i+1}T_i \text{ and } T_i x_{i+1}=x_iT_i.$$

\begin{rem}
Let $B$ be a $k$-algebra and $n\ge 0$. There is a (unique) $k$-algebra
structure on $(B^{\otimes_k n})\otimes ({^0H}_n^f)$ such that 
$B^{\otimes_k n}=\underbrace{B\otimes_k B\otimes_k\cdots\otimes_k B}_{n
\text{ factors}}$
and ${^0H}_n^f$ are subalgebras and
$(1\otimes T_w)\cdot\bigl((a_1\otimes\cdots\otimes a_n)\otimes 1\bigr)=
\bigl((a_{w(1)}\otimes\cdots\otimes a_{w(n)})\otimes T_w\bigr)$.
We denote by $B\wr ({^0H}_n^f)$ the corresponding algebra.

We have $A'=k[x]^I\wr ({^0H}_n^f)$.
\end{rem}

\subsubsection{Definition}
\label{se:defquiverHecke}
We come now to the definition of the quiver Hecke algebras \cite[\S 3.2]{Rou3}.
Fix a matrix $Q=(Q_{ij})_{i,j\in I}$ in $k[u,u']$. Assume
\begin{itemize}
\item $Q_{ii}=0$
\item $Q_{ij}$ is not a zero-divisor in $k[u,u']$ for $i\not=j$ and
\item $Q_{ij}(u,u')=Q_{ji}(u',u)$.
\end{itemize}

We define the algebra $H_n'(Q)=A(\Psi_{I,n},R'_Q)$ with relations $R'_Q$
(we use $\tau_i$ to denote the arrow called $s_i$ earlier):
$$\tau_{i,s_i(v)}\tau_{i,v}=Q_{v_i,v_{i+1}}(x_{i,v},x_{i+1,v})
,\ \tau_i\tau_j=\tau_j\tau_i \text{ if }|i-j|>1$$
$$\tau_{i+1,s_is_{i+1}(v)}\tau_{i,s_{i+1}(v)}\tau_{i+1,v}=
\tau_{i,s_{i+1}s_i(v)}\tau_{i+1,s_i(v)}\tau_{i,v}
\text{ if }v_i\not=v_{i+2} \text{ or }v_i=v_{i+1}=v_{i+2}$$
$$x_ix_j=x_jx_i,\ 
\tau_{i,v}x_{a,v}-x_{s_i(a),s_i(v)}\tau_{i,v}=\begin{cases}
-1_v & \text{ if }a=i \text{ and } v_i=v_{i+1} \\
1_v & \text{ if }a=i+1 \text{ and } v_i=v_{i+1} \\
0 & \text{ otherwise.}
\end{cases}$$

We define $H_n(Q)=A(\Psi_{I,n},R_Q)$ where $R_Q$ consist of the
relations $R'_Q$ together with
\begin{multline*}
\tau_{i+1,s_is_{i+1}(v)}\tau_{i,s_{i+1}(v)}\tau_{i+1,v}-
\tau_{i,s_{i+1}s_i(v)}\tau_{i+1,s_i(v)}\tau_{i,v}=\\
(x_{i+2,v}-x_{i,v})^{-1}\left(Q_{v_i,v_{i+1}}(x_{i+2,v},x_{i+1,v})-
Q_{v_i,v_{i+1}}(x_{i,v},x_{i+1,v})\right)
\end{multline*}
when $v_i=v_{i+2}\not=v_{i+1}$

Note that the relation is written using
$(x_{i+2,v}-x_{i,v})^{-1}$ to simplify the expression: the fraction is
actually a polynomial in $x_{i,v}$, $x_{i+1,v}$ and $x_{i+2,v}$.

\smallskip
The following lemma shows that, up to multiplication by a polynomial, these
extra relations follow from the ones in $R'_Q$. Since
$H_n(Q)$ has no polynomial torsion, it follows that $H_n(Q)$ is the quotient
of $H'_n(Q)$ by the polynomial torsion.

\begin{lemma}
The kernel of the canonical surjective morphism of quiver algebras
$H'_n(Q)\twoheadrightarrow H_n(Q)$ is the subspace of elements $a$
such that there is $P\in k[x_1,\ldots,x_n]$ that is not a
zero-divisor and such that $Pa=0$.
\end{lemma}

\begin{proof}
The property of $H_n(Q)$ to have no polynomial torsion is a
consequence of Proposition \ref{pr:PBW} below. We show here that the kernel of
the canonical map $H'_n(Q)\to H_n(Q)$ is made of torsion elements, \ie,
that suitable polynomial multiples of the relations in $R_Q$ but no in
$R'_Q$ come from relations in $R'_Q$.

Consider $v\in I^n$ with $v_i=v_{i+2}\not=v_{i+1}$. We have the
following equalities in $H'_n(Q)$:
$$(\tau_{i,v}\tau_{i+1,s_{i+1}(v)}\tau_{i,s_{i+1}(v)})\tau_{i+1,v}=
(\tau_{i+1,s_i(v)}\tau_{i,v}\tau_{i+1,s_{i+1}(v)})\tau_{i+1,v}=
\tau_{i+1,s_i(v)}\tau_{i,v}Q_{v_{i+1},v_i}(x_{i+1,v},x_{i+2,v})$$
\begin{multline*}
(\tau_{i,v}\tau_{i,s_i(v)})\tau_{i+1,s_i(v)}\tau_{i,v}=
Q_{v_{i+1},v_i}(x_{i,s_i(v)},x_{i+1,s_i(v)})
\tau_{i+1,s_i(v)}\tau_{i,v}=\\
\tau_{i+1,s_i(v)}\tau_{i,v}Q_{v_{i+1},v_i}(x_{i+1,v},x_{i+2,v})
+\tau_{i,v}(x_{i+2,v}-x_{i+1,v})^{-1}(Q_{v_i,v_{i+1}}(x_{i+2,v},x_{i+1,v})-
Q_{v_i,v_{i+1}}(x_{i,v},x_{i+1,v}))
\end{multline*}
Let 
\begin{multline*}
a=\tau_{i+1,s_is_{i+1}(v)}\tau_{i,s_{i+1}(v)}\tau_{i+1,v}-
\tau_{i,s_{i+1}s_i(v)}\tau_{i+1,s_i(v)}\tau_{i,v}-\\
-(x_{i+2,v}-x_{i,v})^{-1}(Q_{v_i,v_{i+1}}(x_{i+2,v},x_{i+1,v})-
Q_{v_i,v_{i+1}}(x_{i,v},x_{i+1,v}))
\end{multline*}
We have shown that $\tau_{i,v}a=0$, hence
$$0=\tau_{i,s_i(v)}\tau_{i,v}a=Q_{v_i,v_{i+1}}(x_{i,v},x_{i+1,v})a$$
and the lemma follows.
\end{proof}

\smallskip
When $|I|=1$, we have $A=k{^0H}_n$.

\medskip
Let $J$ be a set of finite sequences of elements of
$\{1,\ldots,n-1\}$ such that
$\{s_{i_r}\cdots s_{i_1}\}_{(i_1,\ldots,i_r)\in J}$ is a set of minimal
length representatives of elements of $\GS_n$.

The following lemma is straightforward.
\begin{lemma}
\label{le:generatingsetquiverHecke}
The set
$$S=\{\tau_{i_r,s_{i_{r-1}}\cdots s_{i_1}(v)}\cdots
 \tau_{i_1,v}x_{1,v}^{a_1}\cdots x_{n,v}^{a_n}
\}_{(i_1,\ldots,i_r)\in J,(a_1,\ldots,a_n)\in\BZ_{\ge 0}^n,v\in I^n}$$
generates $H_n(Q)$ as a $k$-module.
\end{lemma}

\begin{proof}
Given a product $a$ of
generators $\tau_i$ and $x_i$, one shows by induction on the number of
$\tau_i$'s in the product, then on the number of pairs of an $x_i$ to the left
of a $\tau_j$, that $a$ is a linear combination of elements in $S$.
\end{proof}

Note that the generating set is compatible with the quiver algebra structure,
as it is made of paths. Given $v,v'\in I^n$, 
the set
$$\{\tau_{i_r,s_{i_{r-1}}\cdots s_{i_1}(v)}\cdots
 \tau_{i_1,v}x_{1,v}^{a_1}\cdots x_{n,v}^{a_n}
\}_{(i_1,\ldots,i_r)\in J,(a_1,\ldots,a_n)\in\BZ_{\ge 0}^n,
s_{i_r}\cdots s_{i_1}(v)=v'}$$
generates $1_{v'}H_n(Q)1_v$ as a $k$-module.

\medskip
The algebra $H_n(Q)$ is filtered with $1_v$ and $x_{i,v}$ in degree $0$
and $\tau_{i,v}$ in degree $1$. The relations $R_Q$ become
the relations $R_0$ after neglecting terms of lower degree, \ie,
the morphism $A(\Psi_{I,n})\to H_n(Q)$
gives rise to a surjective algebra morphism
$f:A'=k[x]^I\wr {^0H}_n^f\to \gr H_n(Q)$, where
$\gr H_n(Q)=\bigoplus_{i\ge 0} F^iH_n(Q)/F^{i-1}H_n(Q)$ is the graded
algebra associated with the filtration.

\begin{prop}
\label{pr:PBW}
The algebra $H_n(Q)$ satisfies the Poincar\'e-Birkhoff-Witt property, \ie,
the morphism $f$ is an isomorphism.
Furthermore, $H_n(Q)$ is a free $k$-module with basis $S$
\end{prop}

We will prove this proposition by constructing a faithful
polynomial representation. This is similar to the case of the nil affine
Hecke algebra (case $|I|=1$).

\smallskip
Assume there is $d\in\BZ^I$ such that $Q_{ij}(u^{d_i},v^{d_j})$
is a homogeneous polynomial for all $i\not=j$. We denote by
$p_{ij}$ the degree of $Q_{ij}(u^{d_i},v^{d_j})$.
Then, the algebra $H_n(Q)$ is a graded $k$-algebra with
$\deg x_i=2d_{\nu_i}$ and $\deg\tau_{ij}=p_{\nu_i\nu_j}$.

\smallskip
The quiver Hecke algebras have been introduced and studied
independently by Khovanov and Lauda \cite{KhoLau1,KhoLau2} for particular $Q$'s.

\subsubsection{Polynomial realization}
\label{se:isographs}
Let $P=(P_{ij})_{i,j\in I}$ be a matrix in
$k[u,u']$ with $P_{ii}=0$ for all $i\in I$ and
such that $P_{ij}$ is not a zero-divisor for $i\not=j$.
Let $Q_{i,j}(u,u')=P_{i,j}(u,u')P_{j,i}(u',u)$.

We consider the following representation $M=(M_v)_{v\in I^n}$ of
our quiver algebra. We put $M_v=k[x_1,\ldots,x_n]$. We let
$x_i$ act by multiplication by $x_i$ and
$$\tau_{i,v}:M_v\to M_{s_i(v)} \textrm{ acts by }
\begin{cases}
(x_i-x_{i+1})^{-1}(s_i-1) & \text{ if } s_i(v)=v \\
P_{v_i,v_{i+1}}(x_{i+1},x_i)s_i & \text{ otherwise.}
\end{cases}$$

\begin{prop}
\label{pr:polrep}
The construction above defines a faithful representation of
$H_n(Q)$ on $M$.
\end{prop}

\begin{proof}[Proof of Propositions \ref{pr:PBW} and \ref{pr:polrep}]
Let $\tau'_{i,v}=\begin{cases}
(x_i-x_{i+1})^{-1}(s_i  -1) & \text{ if }v_i=v_{i+1} \\
P_{v_i,v_{i+1}}(x_{i+1},x_i)s_i &
\text{ otherwise}.
\end{cases}$

We have $\tau_{i,s_{i+1}(v)}'\tau_{i+1,v}'=$
$$\begin{cases}
(x_i-x_{i+1})^{-1}\left((x_i-x_{i+2})^{-1}(s_is_{i+1}-s_i)-
(x_{i+1}-x_{i+2})^{-1}(s_{i+1}-1)\right) &
 \text{ if }v_i=v_{i+1}=v_{i+2} \\
P_{v_,v_{i+1}}(x_{i+1},x_i)(x_i-x_{i+2})^{-1}(s_is_{i+1}-s_i) &
 \text{ if }v_{i+1}=v_{i+2}\not=v_i \\
(x_i-x_{i+1})^{-1}\left(P_{v_{i+1},v_{i+2}}(x_{i+2},x_i)s_is_{i+1}-
P_{v_{i+1},v_{i+2}}(x_{i+2},x_{i+1})s_{i+1} \right)  &
 \text{ if }v_i=v_{i+2}\not=v_{i+1} \\
P_{v_i,v_{i+2}}(x_{i+1},x_i)P_{v_{i+1},v_{i+2}}(x_{i+2},x_i)s_is_{i+1}&
 \text{ if }v_{i+2}\not\in\{v_i,v_{i+1}\}.
\end{cases}$$

One checks then easily that the defining relations of $H_n(Q)$ hold.

\smallskip
It is easy to check that the image of $S$ in $\End_{k^{I^n}}(M)$ is
linearly independent over $k$: this can be done by extending
scalars to $k(x_1,\ldots,x_n)$ and relating by a triangular base
change the bases $\{\partial_w\}_{w\in\GS_n}$ and $\{w\}_{w\in\GS_n}$ of 
$\End_{k(x_1,\ldots,x_n)^{\GS_n}}(k(x_1,\ldots,x_n))$ as a left
$k(x_1,\ldots,x_n)$-module. It follows that 
the canonical map $H_n(Q)\to \End_{k^{I^n}}(M)$ is injective and that
$S$ is a basis of $H_n(Q)$ over $k$. Also, the image of $S$ in
$\gr H_n(Q)$ lifts to a basis of $A'$.

\smallskip
Note finally that given $(Q_{ij})$, we can construct a matrix $(P_{ij})$ as
follows: for $i\not=j$, choose an order. When $i<j$, we define $P_{ij}=Q_{ij}$ 
and $P_{ji}=1$.
\end{proof}

Let $\nu\in I^n$. We put $1_{|\nu|}=\sum_{\sigma\in\GS_n/\Stab(\nu)}
1_{\sigma(\nu)}$ and $H(|\nu|)=1_{|\nu|}H_n(Q)1_{|\nu|}$.
The next proposition shows that $H(|\nu|)$ doesn't have ``non-obvious''
quotients that remain torsion-free over polynomials.

Let $n_i=\#\{r|\nu_r=i\}$ and let
$\gamma_i:\{1,\ldots,n_i\}\to\{1,\ldots,n\}$ be the
increasing map such that $\nu_{\gamma_i(r)}=i$ 
for all $r$.

 For every $\sigma\in\GS_n$, we have a morphism
of algebras
$$\bigotimes_i k[X_{i,1},\ldots,X_{i,n_i}]\to 
1_{\sigma(\nu)}H_n(\Gamma)1_{\sigma(\nu)},\
X_{i,r}\mapsto x_{i,\sigma(\gamma_i(r))}.$$

The diagonal map restricts to an algebra
morphism $\bigotimes_i k[X_{i,1},\ldots,X_{i,n_i}]^{\GS_{n_i}}\to
Z(H(|\nu|))$
(this is actually an isomorphism by \cite[Proposition 3.9]{Rou3}).

\begin{prop}
\label{pr:ideals}
Let $J$ be a non-zero two-sided ideal of $H(|\nu|)$. Then,
there is a non-zero
$P\in \bigotimes_i k[X_{i,1},\ldots,X_{i,n_i}]^{\GS_{n_i}}$
 such that $P\cdot \id_M\in J$.
\end{prop}

\begin{proof}
Consider the algebra $A=k^{I}[x]\wr \GS_n=\left(\bigoplus_{\mu\in I^n}
k[x_1,\ldots,x_n]1_\mu\right)\rtimes\GS_n$. Let 
$\CO=\bigoplus_{\mu} k[x_1,\ldots,x_n][\{(x_i-x_j)^{-1}\}]_{i\not=j,
\mu_i=\mu_j}1_\mu$ and $A'=\CO\otimes_{(k^{I}[x])^{\otimes n}}A$.
Let $B=\bigoplus_{\sigma,\sigma'\in\GS_n}1_{\sigma(\nu)}A'1_{\sigma'(\nu)}$.
The algebra $B$ is Morita-equivalent
to its center which is isomorphic to
$$\bigotimes_{i\in I} \left(k[x_1,\ldots,x_{n_i}]
[\{(x_a-x_b)^{-1}\}]_{a\not=b}\right)^{\GS_{n_i}}.$$
It
follows that any non-zero ideal of $B$ intersects non-trivially $Z(B)$.
The proposition follows now from the embedding of
$\End_{\CB}(M)$ in $B$ and the properties of that embedding
 \cite[Proposition 3.12]{Rou3}.
\end{proof}

\subsubsection{Cartan matrices and quivers}
\label{se:Cartan}
A generalized Cartan matrix is a matrix
$C=(a_{ij})_{i,j\in I}$ such that $a_{ii}=2$, $a_{ij}\le 0$ for $i\not=j$ and
$a_{ij}=0$ if and only if $a_{ji}=0$.
The matrix $C$ is symmetrizable if in addition
there is a diagonal matrix $D$ with diagonal coefficients in $\BZ_{>0}$
such that $DC$ is symmetric.

Consider now a graph with vertex set $I$ and with no loop. We
define a symmetric Cartan matrix by putting $a_{ij}=-m_{ij}$ for $i\not=j$,
where $m_{ij}$ is the number of edges between $i$ and $j$.
This correspondence gives a bijection between graphs with no loops and
symmetric Cartan matrices.

\smallskip
Let $\Gamma$ be a quiver (with no loops) and $I$ its vertex set. This
defines a graph by forgetting the orientation, hence a symmetric
Cartan matrix.
Let $d_{ij}$ be the number of arrows $i\to j$, so that
$m_{ij}=d_{ij}+d_{ji}$.
Let $Q_{ij}=(-1)^{d_{ij}}(u-v)^{m_{ij}}$ for $i\not=j$.
We put $H_n(\Gamma)=H_n(Q)$, where $k=\BZ$.

This is a graded algebra, with $\deg x_i=2$ and $\deg \tau_{v,i}=
-a_{v_i,v_{i+1}}$.

\smallskip
 Let $v,v'\in I^n$. Let
$n_i=\#\{r|v_r=i\}$ and
$n'_i=\#\{r|v'_r=i\}$. We have $1_{v'}H_n(\Gamma)1_v=0$ unless
$n_i=n'_i$ for all $i$. Assume this holds.
Define $\gamma_i,\gamma'_i:\{1,\ldots,n_i\}\to\{1,\ldots,n\}$ to be the
increasing maps such that $v_{\gamma_i(r)}=v'_{\gamma'_i(r)}=i$ 
for all $r$. Let $W=\prod_i \GS_{n_i}$.

\begin{lemma}
\label{le:grdim}
The left (resp. right) $\BZ[x_1,\ldots,x_n]$-module
$1_{v'}H_n(\Gamma)1_v$ is free: there is a graded $\BZ$-module $L$ such that
\begin{itemize}
\item
$1_{v'}H_n(\Gamma)1_v\simeq \BZ[x_1,\ldots,x_n]\otimes_\BZ L$ as
graded left $\BZ[x_1,\ldots,x_n]$-modules and
\item
$1_{v'}H_n(\Gamma)1_v\simeq L\otimes_\BZ\BZ[x_1,\ldots,x_n]$ as graded
right $\BZ[x_1,\ldots,x_n]$-modules.
\end{itemize}
We have 
$$\mathrm{grdim} L=\sum_{w\in W} q^{\frac{1}{2}\sum_{s,t\in I}
a_{st}\cdot \#\{a,b| \gamma_s(a)<\gamma_t(b)\text{ and }
\gamma_s'(w_s(a))> \gamma_t'(w_t(b))\}}.$$
\end{lemma}

\begin{proof}
The first part of the lemma follows from Proposition \ref{pr:PBW} (and its
right counterpart).
We have 
$$\mathrm{grdim}L=\sum_{\substack{(i_1,\ldots,i_r)\in J\\
s_{i_r}\cdots s_{i_1}(v)=v'}}
q^{\frac{1}{2}
\deg(\tau_{i_r,s_{i_{r-1}}\cdots s_{i_1}(v)}\cdots \tau_{i_1,v})}.$$
The set $E$ of elements $h\in\GS_n$ such that $h(v)=v'$ is a left (and right)
principal homogeneous set under the action of $W$, via the maps $\gamma_i$
and $\gamma'_i$. Denote by $g\in\GS_n$
the unique element of minimal length such that $g(v)=v'$. Then,
we obtain a bijection $W\iso E,\ w\mapsto w\circ g$. The formula
follows from a variant of Proposition \ref{pr:lengthequalsr}.
\end{proof}

\subsubsection{Relation with (degenerate) affine Hecke algebras}
We show in this section that quiver Hecke algebras associated with
quivers of type $A$ (finite or affine) are connected with (degenerate) affine
Hecke algebras for $\GL_n$. 

\smallskip

Let $\bar{R}=\BZ[q^{\pm 1}]=R[q^{\pm 1}]/(q_1-q,q_2+1)$.
Let $H_n$ be the {\em affine Hecke algebra}: it
is the $\bar{R}$-algebra
generated by elements $X_1,\ldots,X_n$, $T_1,\ldots,T_{n-1}$ where
the $X_i$ are invertible and the relations are
$$(T_i-q)(T_i+1)=0,\ T_iT_j=T_jT_i \text{ if }|i-j|>1,
T_iT_{i+1}T_i=T_{i+1}T_iT_{i+1}$$
$$X_iX_j=X_jX_i,\ T_i X_j=X_j T_i \text{ if } j-i\not=0,1 \text{ and }
T_iX_{i+1}-X_iT_i=(q-1)X_{i+1}.$$

As in the nil affine Hecke case (Proposition \ref{pr:faithfulnilaffineHecke}),
we have a decomposition as $\bar{R}$-modules
$$H_n=\bar{R}[X_1^{\pm 1},\ldots,X_n^{\pm 1}]\otimes_{\bar{R}} \bar{R}H_n^f$$
and $\bar{R}[X_1^{\pm 1},\ldots,X_n^{\pm 1}]$ and $\bar{R}H_n^f$ are
subalgebras.

\medskip
Let $k$ be field endowed with an $\bar{R}$-algebra structure and assume
$q 1_k\not=1_k$.

Given $M$ a $k$-vector space and $x$ an endomorphism of $M$, we say that
$x$ is {\em locally nilpotent} on $M$ if $M$ is the union of subspaces on
which $x$ is nilpotent or, equivalently, if for every $m\in M$, there
is $N>0$ such that $x^N m=0$.

\smallskip
Let $M$ be a $kH_n$-module. Given $v\in (k^\times)^n$, we denote by $M_v$ the
subspace of $M$
on which $X_i-v_i$ acts locally nilpotently for $1\le i\le n$.

\smallskip
Let $I$ be a subset of $k^\times$ and let $\CC_I$ be the category of
$kH_n$-modules $M$ such that $M=\bigoplus_{v\in I^n}M_v$. Note that
a finite dimensional $kH_n$-module is in $\CC_I$ if and only if the
eigenvalues of the $X_i$ acting on $M$ are in $I$.

\smallskip
We define a quiver $\Gamma$ with vertex set $I$ and arrows $i\to qi$.
Assume $q\not=1$. Denote by $e$ the multiplicative order of $q$. 
When $\Gamma$ is connected, the possible types of the underlying graph are
\begin{itemize}
\item $A_n$ if $|I|=n<e$.
\item $\tilde{A}_{e-1}$ if $|I|=e$.
\item $A_{\infty}$ if $I$ is bounded in one direction but not finite.
\item $A_{\infty,\infty}$ if $I$ is unbounded in both directions.
\end{itemize}

\smallskip
We denote by $\CC_\Gamma^0$ the category of $kH_n(\Gamma)$-modules
$M$ such that for every $v\in I^n$, then $x_{i,v}$ acts locally nilpotently
on $M_v$.

The proof of the following Theorem (and the next one)
relies on checking relations and writing
formulas for an inverse functor.
\begin{thm}[{\protect \cite[Theorem 3.20]{Rou3}}]
\label{th:equivaffineHeckeandquiver}
There is an equivalence of categories $\CC_\Gamma^0\iso \CC_I$ given by
$(M_v)_v\mapsto \bigoplus_v M_v$ and where
$X_i$ acts on $M_v$ by $(x_i+v_i)$ and
$T_i$ acts on $M_v$ by
\begin{itemize}
\item $(qx_i-x_{i+1})\tau_i+q$ if $v_i=v_{i+1}$ \\
\item $(q^{-1}x_i-x_{i+1})^{-1}(\tau_i+(1-q)x_{i+1})$ if $v_{i+1}=qv_i$ \\
\item $(v_ix_i-v_{i+1}x_{i+1})^{-1}\left(
(qv_ix_i-v_{i+1}x_{i+1})\tau_i+
(1-q)v_{i+1}x_{i+1}\right)$
otherwise.
\end{itemize}
\end{thm}

\medskip
There is yet another version of affine Hecke algebras: the {\em degenerate 
affine Hecke algebra} $\bar{H}_n$, a $\BZ$-algebra generated by
$X_1,\ldots,X_n$ and $s_1,\ldots,s_{n-1}$ with relations
$$s_i^2=1,s_is_j=s_js_i \text{ if }|i-j|>1,\
s_is_{i+1}s_i=s_{i+1}s_is_{i+1}$$
$$X_iX_j=X_jX_i,\ s_i X_j=X_j s_i \text{ if } j-i\not=0,1 \text{ and }
s_iX_{i+1}-X_is_i=1.$$
We have a $\BZ$-module decomposition
$$\bar{H}_n=\BZ[X_1,\ldots,X_n]\otimes \BZ\GS_n$$
and $\BZ[X_1,\ldots,X_n]$ and $\BZ\GS_n$ are subalgebras.

Let $k$ be a field.

Let $I$ a subset of $k$. We denote by
$\Gamma$ the quiver with set of vertices $I$ and with arrows $i\to i+1$.

When $\Gamma$ is connected, the possible types of the underlying graph are
\begin{itemize}
\item $A_n$ if $|I|=n$ and $k$ has characteristic $0$ or $p>n$.
\item $\tilde{A}_{p-1}$ if $|I|=p$ is the characteristic of $k$.
\item $A_{\infty}$ if $I$ is bounded in one direction but not finite.
\item $A_{\infty,\infty}$ if $I$ is unbounded in both directions.
\end{itemize}

Given $M$ a $k\bar{H}_n$-module and $v\in k^n$, we denote by $M_v$ the
subspace of $M$ where $X_i-v_i$ acts locally nilpotently for all $i$.
Let $I$ be a subset of $k^\times$ and let $\bar{\CC}_I$ be the category of
$k\bar{H}_n$-modules $M$ such that $M=\bigoplus_{v\in I^n}M_v$.

\begin{thm}[{\protect \cite[Theorem 3.17]{Rou3}}]
\label{th:equivdegenaffineHeckeandquiver}
There is an equivalence of categories $\CC_\Gamma^0\iso \bar{\CC}_I$ given by
$(M_v)_v\mapsto \bigoplus_v M_v$ and where
 $X_i$ acts on $M_v$ by $(x_i+v_i)$ and
$s_i$ acts on $M_v$ by
\begin{itemize}
\item $(x_i-x_{i+1}+1)\tau_i+1$ if $v_i=v_{i+1}$ \\
\item $(x_i-x_{i+1}-1)^{-1}(\tau_i-1)$
if $v_{i+1}=v_i+1$ \\
\item $(x_i-x_{i+1}+v_{i+1}-v_i+1)
(x_i-x_{i+1}+v_{i+1}-v_i)^{-1}(\tau_i-1)+1$
otherwise.
\end{itemize}
\end{thm}

\subsection{Half $2$-Kac-Moody algebras}
\subsubsection{Monoidal categories}
Recall that a {\em strict monoidal category} is a category equipped with a 
tensor product with a unit and
satisfying $(V\otimes W)\otimes X=V\otimes (W\otimes X)$.
We will write $XY$ for $X\otimes Y$. We will also denote by $X$ the
identity endomorphism of an object $X$.

We have a canonical map $\Hom(V_1,V_2)\times \Hom(W_1,W_2)\to
\Hom(V_1\otimes W_1,V_2\otimes W_2)$. Given $f:V_1\to V_2$ and
$g:W_1\to W_2$, there is a commutative diagram
$$\xymatrix{
V_1\otimes W_1\ar[rr]^{f\otimes W_1} \ar[d]_{V_1\otimes g}
 \ar[drr]^{f\otimes g} &&
V_2\otimes W_1 \ar[d]^{V_2\otimes g} \\
V_1\otimes W_2\ar[rr]_{f\otimes W_2} &&
V_2\otimes W_2
}$$

\smallskip
A typical example of a monoidal category is the category of vector
spaces over a field (or more generally, modules over a commutative
algebra).

A more interesting example is the following. Let $\CA$ be a category
and $\CC$ be the category of functors $\CA\to\CA$. Then, $\CC$ is
a strict monoidal category where the product is the composition
of functors. The $\Hom$-spaces are given by natural transformations
of functors.

\smallskip
Let $\CC$ be an additive category. We define the {\em idempotent completion}
$\CC^i$ of $\CC$ as the additive category obtained from $\CC$ by adding
images of idempotents: its objects are pairs $(M,e)$ where $M$ is
an object of $\CC$ and $e$ is an idempotent of $\End_{\CC}(M)$.
We put $\Hom_{\CC^i}((M,e),(N,f))=f\Hom_\CC(M,N)e$.
We have a fully faithful functor $\CC\to\CC^i$ given by
$M\mapsto (M,\id_M)$.
If $A$ is an algebra, the idempotent completion of the category of
free $A$-modules is equivalent to the category of projective $A$-modules.

We say that $\CC$ is {\em idempotent complete} if the canonical functor 
$\CC\to\CC^i$ is an equivalence, \ie, if every idempotent has an image.

\subsubsection{Symmetric groups}
Let us start with an example of monoidal category based on symmetric
groups. We define $\CC$ to be the strict monoidal $\BZ$-linear category
generated by an object $E$ and by an arrow $s:E^2\to E^2$ subject to the
relations
$$s^2=E^2,\ (Es)\circ(sE)\circ(Es)=(sE)\circ(Es)\circ(sE).$$

This category is easy to describe: its objects are direct sums of copies of
$E^n$ for various $n$'s. We have
$\Hom(E^n,E^m)=0$ if $m\not=n$ and $\End(E^n)=\BZ[\GS_n]$:
this is given by $s_i\mapsto E^{n-i-1}sE^{i-1}$.

Note that, as a monoidal category, $\CC$ is equipped with maps
$\End(E^m)\times\End(E^n)\to \End(E^{m+n})$. They correspond to
the embedding $\GS_m\times \GS_n\hookrightarrow \GS_{m+n}$,
where $\GS_n$ goes to $\GS\{1,\ldots,n\}$ and
$\GS_m$ goes to $\GS\{n+1,\ldots,n+m\}$.

\begin{rem}
The category $\CC$ can also be defined as the free symmetric
monoidal $\BZ$-linear category on one object $E$.
\end{rem}

\subsubsection{Half}
\label{se:defhalf}
Let us follow now \cite[\S 4.1.1]{Rou3}.
Let $C=(a_{ij})_{i,j\in I}$ be a generalized Cartan matrix.
We construct a matrix $Q$ satisfying the conditions of
\S \ref{se:defquiverHecke}.

Let $\{t_{i,j,r,s}\}$ be a family of indeterminates with
$i\not=j\in I$, $0\le r<-a_{ij}$ and $0\le s<-a_{ji}$ and such that
$t_{j,i,s,r}=t_{i,j,r,s}$. Let $\{t_{ij}\}_{i\not=j}$ be a family of
indeterminates with $t_{ij}=t_{ji}$ if $a_{ij}=0$.

Let $k=k^C=\BZ[\{t_{i,j,r,s}\}\cup\{t_{ij}^{\pm 1}\}]$. Given $i,j\in I$,
we put
$$Q_{ij}=\begin{cases}
0 & \text{ if } i=j \\
t_{ij} &\text{ if }i\not=j \text{ and }a_{ij}=0\\
t_{ij}u^{-a_{ij}}+\sum_{\substack{0\le r<-a_{ij}\\ 0\le s<-a_{ji}}}
t_{i,j,r,s} u^r v^s + t_{ji}v^{-a_{ji}} & \text{ if }i\not=j
\text{ and }a_{ij}\not=0.
\end{cases}$$

We define $\CB=\CB(C)$ as
the strict monoidal $k$-linear category generated by objects
$F_s \text{ for } s\in I$
and by arrows
$$x_s:F_s\to F_s \text{ and }
\tau_{st}:F_sF_t\to F_tF_s \text{ for }s,t\in I$$
with relations

\begin{enumerate}
\item 
\label{en:half1}
$\tau_{st}\circ \tau_{ts}=Q_{st}(F_tx_s,x_tF_s)$
\item
\label{en:half2}
$\tau_{tu}F_s\circ F_t \tau_{su}\circ \tau_{st}F_u-
F_u \tau_{st}\circ \tau_{su}F_t\circ F_s \tau_{tu}=
\begin{cases}
\frac{Q_{st}(x_sF_t,F_sx_t)F_s-F_sQ_{st}(F_tx_s,x_tF_s)}{x_sF_tF_s-F_sF_tx_s}F_s
 & \text{ if } s=u\vspace{0.2cm}\\
0  & \text{ otherwise.}
\end{cases}$
\item
\label{en:half3}
$\tau_{st}\circ x_s F_t-F_t x_s\circ \tau_{st}=\delta_{st}$
\item
\label{en:half4}
$\tau_{st}\circ F_sx_t-x_tF_s\circ \tau_{st}=-\delta_{st}$
\end{enumerate}

Given $n\ge 0$, we denote by $\CB_n$ the full subcategory of $\CB$ whose
objects are direct sums of objects of the form
$F_{v_n}\cdots F_{v_1}$ for $v_1,\ldots,v_n\in I$. We have
$\CB=\bigoplus_n \CB_n$ (as $k$-linear categories).

\smallskip
The relations state that the maps $x_s$ and $\tau_{st}$ give an action of
the quiver Hecke algebra associated with $Q$ on sum of
products of $F_s$.
More precisely, we have an isomorphism of algebras
\begin{align*}
H_n(Q)&\iso \bigoplus_{v,v'\in I^n} \Hom_{\CB}(F_{v_n}\cdots
F_{v_1},F_{v'_n}\cdots F_{v'_1})\\
1_v&\mapsto \id_{F_{v_n}\cdots F_{v_1}} \\
x_{i,v}&\mapsto F_{v_n}\cdots F_{v_{i+1}}x_{v_i}F_{v_{i-1}}\cdots F_{v_1}\\
\tau_{i,v}&\mapsto F_{v_n}\cdots F_{v_{i+2}}\tau_{v_{i+1},v_i}
F_{v_{i-1}}\cdots F_{v_1}
\end{align*}

Note that divided powers can be defined in $\CB^i$, following
\cite[\S 5.2.1]{ChRou}. We have an isomorphism $k\otimes_\BZ{^0H}_n\iso
\End(E_i^n)$. The endomorphism of $E_i^n$ induced by $T_{w[1,n]}$ has an
image $E_i^{(n)}\in\CB^i$. We have
$E_i^n\simeq E_i^{(n)}\otimes_k k^{n!}$.

\medskip
Assume now $C$ is symmetrizable, with $D=\diag(d_i)_{i\in I}$.
We define $k^{\mathrm{gr}}=k^C/(\{t_{i,j,r,s}\}_{d_ir+d_js\not=-d_ia_{ij}})$
and
$\CB^{\mathrm{gr}}=\CB\otimes_k k^{\mathrm{gr}}$.
This category can be graded by setting
$\deg x_s=2d_s$,  $\deg\tau_{st}=d_sa_{st}$ and $\deg \eps_{s,\lambda}=
d_s(1-\langle\lambda,\alpha_s^\vee\rangle)$.

\smallskip
Let $\Gamma$ be quiver with no loops whose underlying graph corresponds
to $C$ (which is then symmetric).
Let $J$ be the ideal of $k$ generated by the coefficients of the
polynomials
$Q_{ij}-(-1)^{d_{ij}}(u-v)^{-a_{ij}}$ (cf \S \ref{se:Cartan}) and
$k^\Gamma=k^C/J=\BZ$.
We put $\CB(\Gamma)=\CB\otimes_k k^\Gamma$. This is again graded, as above.

\smallskip
Similar monoidal categories have been constructed independently by Khovanov
and Lauda \cite{KhoLau1,KhoLau2}, who have shown that, over a field, they
provide a categorification of $U_\BZ(\Gn^-)$ and $U_{\BZ[q^{\pm 1/2}]}(\Gn^-)$,
where $\Gn^-$ is the half Kac-Moody algebra associated to $C$
(cf \S \ref{se:defKM}).

\medskip
Given an additive category $\CC$, we denote by $K_0(\CC)$ the
Grothendieck group of $\CC$. Assume $\CC$ is enriched in graded
$\BZ$-modules. We can define a new additive category
$\CC\,\mgr$ with objects families $\{M_i\}_{i\in\BZ}$ of objects of
$\CC$ with $M_i=0$ for almost all $i$. We put
$\Hom_{\CC\,\mgr}(\{M_i\},\{N_i\})=\bigoplus_{m,n}\Hom_{\CC}(M_m,N_n)_{n-m}$.
The category $\CC\,\mgr$ is $\BZ$-graded, \ie, it is equipped with an
automorphism $T$ given by $T(\{M_i\})_n=M_{n+1}$. The action of $T$
on $K_0(\CC\,\mgr)$ endows it with a structure of
$\BZ[q^{\pm 1/2}]$-module, where $q^{1/2}$ acts by $[T]$.

\begin{thm}[{\protect\cite[Corollary 7 and Theorem 8]{KhoLau2}}]
\label{th:categB}
Given $s\not=t\in I$, there are isomorphisms in $\CB^i$
$$\bigoplus_{i\text{ even}}F_s^{(-a_{st}-i+1)}F_tF_s^{(i)}\simeq
\bigoplus_{i\text{ odd}}F_s^{(-a_{st}-i+1)}F_tF_s^{(i)}.$$
Let $K$ be a field that is a $k$-algebra. The relations above provide
isomorphisms of rings
$$U_{\BZ}(\Gn^-)\iso K_0(\CB^i\otimes_k K).$$
When $C$ is symmetrizable and $K$ is in addition a $k^{\mathrm{gr}}$-algebra,
this gives an isomorphism of $\BZ[q^{\pm 1/2}]$-algebras
$$U_{\BZ[q^{\pm 1/2}]}(\Gn^-)\iso K_0((\CB(\Gamma)^i\otimes_k K)\,\mgr).$$
\end{thm}

\section{$2$-Kac-Moody algebras}
\subsection{Kac-Moody algebras}
\label{se:KacMoody}
We recall some basic facts on Kac-Moody algebras and their representations
\cite{Kac} and quantum counterparts \cite{Lu1}.

\smallskip
Given an algebra $A$, we denote by $A\mMOD$ the category of $A$-modules,
by $A\mMod$ the category of finitely generated $A$-modules and by
$A\mproj$ the category of finitely generated projective $A$-modules.

\subsubsection{Data}
\label{se:defKM}
Let $C=(a_{ij})_{i,j\in I}$ be a generalized Cartan matrix.
Let 
$$(X,Y,\langle-,-\rangle,\{\alpha_i\}_{i\in I},\{\alpha^\vee_i\}_{i\in I})$$
be a root datum of type $C$, \ie,
\begin{itemize}
\item $X$ and $Y$ are finitely generated free abelian groups and
$\langle-,-\rangle:X\times Y\to\BZ$ is a perfect pairing
\item $\{\alpha_i\}$ is a linearly independent set in $X$ and
$\{\alpha_i^\vee\}$ is a linearly independent set in $Y$
\item $\langle\alpha_j,\alpha_i^\vee\rangle=a_{ij}$.
\end{itemize}

Associated with this data, there is a Kac-Moody algebra $\Gg$ (over $\BC$)
generated
by elements $e_i$, $f_i$ and $h_\zeta$ for $i\in I$ and $\zeta\in Y$.
When $C$ is symmetrizable, there
is also a quantum enveloping algebra $U_q(\Gg)$.
When $C$ corresponds to a Dynkin
diagram, we recover complex reductive Lie algebras and their
corresponding quantum groups.

\smallskip
We denote by $\Gn^-$ the Lie subalgebra of $\Gg$ generated by the $f_i$,
$i\in I$. We denote by $U_{\BZ}(\Gn^-)$ the subring of $U(\Gn^-)$
generated by the elements $\frac{f_i^n}{n!}$ for $i\in I$ and $n\ge 1$.
Assume $C$ is symmetrizable. The quantum algebra $U_q(\Gn^-)$
is the $\BC(q^{1/2})$-algebra with generators $f_i$, $i\in I$, and
relations
$$\sum_{r=0}^{1-a_{ij}}(-1)^r\left[\begin{matrix}1-a_{ij}\\r\end{matrix}
\right]_q
f_i^rf_jf_i^{1-a_{ij}-r}=0$$
for $i\not=j$, 
where
$$[n]_q=\frac{q^{n/2}-q^{-n/2}}{q^{1/2}-q^{-1/2}},\
 \left[\begin{matrix} n\\ r\end{matrix}\right]_q=
\frac{[n]_q!}{[r]_q![n-r]_q!} \text{ and }[n]_q!=[2]_q\cdots [n]_q.$$

In the symmetrizable case, we denote by
$U_{\BZ[q^{\pm 1/2}]}(\Gn^-)$ the $\BZ[q^{\pm 1/2}]$-subalgebra
 of $U_q(\Gn^-)$ generated by the
$\frac{f_i^n}{[n]_q}$.

\begin{example}
Let $C=(2)$, $X=Y=\BZ$, $\alpha=2$ and $\alpha^\vee=1$. Then
$\Gg=\Gsl_2(\BC)$.
\end{example}

\subsubsection{Integrable representations}
Let $M$ be a representation of $\Gg$. We say that it is {\em integrable}
if 
\begin{itemize}
\item $M=\bigoplus_{\lambda\in X}M_\lambda$, where $M_\lambda=
\{m\in M | h_\zeta\cdot m=\langle\lambda,\zeta\rangle m,\ \forall\zeta\in Y\}$
\item
$e_i$ and $f_i$ are locally nilpotent on $M$ for every $i$, \ie,
given $m\in M$, there is an integer $n$ such that $f^l(m)=e^l(m)=0$ for
$l\ge n$.
\end{itemize}

\smallskip
Define a quiver with vertex set $X$ and arrows
$e_i=e_{i,\lambda}:\lambda\to\lambda+\alpha_i$ and
$f_i=f_{i,\lambda}:\lambda\to\lambda-\alpha_i$. Let $A(\Gg)$ be its quiver
algebra over $\BC$, subject to the relations
$$e_{i,\lambda-\alpha_j}f_{j,\lambda}-f_{j,\lambda+\alpha_i}
e_{i,\lambda}=\delta_{ij}\langle\lambda,\alpha_i^\vee\rangle 1_{\lambda}.$$

\smallskip
The following proposition has a proof based on the representation theory of
$\Gsl_2$.
It says that the Serre relations are automatically satisfied
under integrability conditions.

\begin{prop}
The functor $M\mapsto (M_\lambda)_{\lambda\in X}$ is an equivalence from the
category of integrable representations
of $\Gg$ to the category of representations of the quiver algebra
$A(\Gg)$ on which the $e_i$'s and $f_i$'s are locally nilpotent.
\end{prop}

We define a category $\CU(\Gg)$ with $\BC$-linear $\Hom$-spaces.
Its set of objects
is $X$. The morphisms are generated by $e_i:\lambda\to\lambda+\alpha_i$,
$f_i:\lambda\to\lambda-\alpha_i$, subject to the relations
$$[e_i,f_j]_{|\lambda}=\delta_{ij}\langle\lambda,\alpha_i^\vee\rangle.$$

Then, a $\BC$-linear functor $\CU(\Gg)\to\BC\mMOD$
is the same as a representation of the quiver algebra $A(\Gg)$.

\begin{rem}
A representation of a $k$-algebra $A$ is the same as a functor
compatible with the $k$-linear structure
$\CC\to k\mMOD$, where $\CC$ is the category with one object
$\ast$ and $\End(\ast)=A$.
\end{rem}

\subsubsection{Quantum counterpart}
Assume $C$ is symmetrizable.
One can proceed similarly and show that the category of
integrable representations of $U_q(\Gg)$ is equivalent to the category
of representations of the quiver algebra $A_q(\Gg)$, defined as the
$\BC(\sqrt{q})$-algebra with the same quiver as above and relations
\begin{equation}
\label{re:quantumKM}
e_{i,\lambda-\alpha_j}f_{j,\lambda}-f_{j,\lambda+\alpha_i}
e_{i,\lambda}=\delta_{ij}[\langle\lambda,\alpha_i^\vee\rangle]_q 1_{\lambda}.
\end{equation}

\subsubsection{Category $\CO^{\mathrm{int}}$}
\label{se:Oint}
Define the set of integral dominant weights
$X^+=\{\lambda\in X | \langle\lambda,\alpha_i^\vee\rangle\ge 0 \ \forall i\in I\}$.
We denote by $\CO_\Gg^{\mathrm{int}}$ the category of
integrable highest weight modules $M$ of $\Gg$, \ie, $\Gg$-modules such that
\begin{itemize}
\item $M=\bigoplus_{\lambda\in X}M_\lambda$ and $\dim M_\lambda<\infty$ 
for all $\lambda$
\item $e_i$ and $f_i$ are locally nilpotent on $M$ for $i\in I$
\item there is a finite set $K\subset X$ such that 
$\{\lambda\in X | M_\lambda{\not=}0\}\subset\bigcup_{\mu\in K}
(\mu+\sum_{i\in I}\BZ_{\le 0}\alpha_i)$.
\end{itemize}

\smallskip
Let $\Gb$ be the Lie subalgebra of $\Gg$ generated by the elements
$e_i$, $i\in I$ and $h_\zeta$, $\zeta\in Y$.
Let $\lambda\in X^+$. We denote by $\BC v_\lambda$ the
one-dimensional representation of $\Gb$ where $e_i$ acts by $0$ and
$h_i$ acts by $\langle\lambda,\alpha_i^\vee\rangle$. We define the
Verma module
$$\Delta(\lambda)=\Ind_{U(\Gb)}^{U(\Gg)}\BC v_\lambda\simeq U(\Gn^-)$$
and 
$$L^{\max}(\lambda)=\Delta(\lambda)/\Bigl(\sum_{i\in I}U(\Gn^-)
f_i^{\langle\lambda,\alpha_i^\vee\rangle+1}v_\lambda\Bigr).$$
This is the largest quotient of $\Delta(\lambda)$ in $\CO_\Gg^{\mathrm{int}}$.
The Verma module
$\Delta(\lambda)$ has a unique simple quotient $L(\lambda)$ and there is
a surjection $L^{\max}(\lambda)\twoheadrightarrow L(\lambda)$. When $C$ is
symmetrizable, $L^{\max}(\lambda)=L(\lambda)$.
The set $\{L(\lambda)\}_{\lambda\in X^+}$ is a complete set of representatives
of isomorphism classes of simple integrable highest weight modules. These
are the finite-dimensional simple $U(\Gg)$-modules when $\Gg$ is 
finite-dimensional.
When $C$ is symmetrizable, 
integrable highest weight modules are semi-simple.

\smallskip
Note that $L^{\max}(\lambda)$ is characterized by the fact that it represents
the functor 
$$\CO_\Gg^{\mathrm{int}}\to \BC\mMOD,\  
M\mapsto M_\lambda^{\mathrm{hw}}:=\{m\in M_\lambda\ |\ e_i(m)=0\ \forall i\in I\},$$
\ie,
$$\Hom_\Gg(L^{\max}(\lambda),M)\iso M_\lambda^{\mathrm{hw}},\ f\mapsto f(v_\lambda).$$

\subsubsection{$\hat{\Gsl}_n$}
Consider the complex simple Lie algebra $\Gsl_n(\BC)$. This is a Kac-Moody
algebra associated with the following graph:
$$A_{n-1}=
\xy
(-10,0) *\cir<3pt>{}="0" *+!U(1.5){\scriptstyle 1},
(0,0)   *\cir<3pt>{}="1" *+!U(1.5){\scriptstyle 2},
(10,0)  *\cir<3pt>{}="2" *+!U(1.5){\scriptstyle 3},
(30,0)  *\cir<3pt>{}="3" *+!U(1.5){\scriptstyle n-2},
(40,0)  *\cir<3pt>{}="4" *+!U(1.5){\scriptstyle n-1},
"0";"1" **\dir{-},
"1";"2" **\dir{-},
"2";"3" **\dir{.},
"3";"4" **\dir{-},
\endxy$$
The Lie algebra $\Gsl_n(\BC)$ is generated
by the elements $e_i=e_{i,i+1}$, $f_i=e_{i+1,i}$ and $h_i=e_{i,i}-e_{i+1,i+1}$
for $1\le i\le n-1$. We have
$Y_{\Gsl_n}=\BZ h_1\oplus\cdots\oplus\BZ h_{n-1}$,
$X_{\Gsl_n}=\BZ \Lambda_1\oplus\cdots\oplus\BZ\Lambda_{n-1}$ where
$\langle\Lambda_i,h_j\rangle=\delta_{ij}$. We have
$\alpha_i=e_{i,i}^*-e_{i+1,i+1}^*$ and $\alpha_i^\vee=h_i$ for $1\le i\le n-1$.

\medskip
Consider now the Lie algebra $\Gg_l=\Gsl_n(\BC)\otimes_\BC \BC[t,t^{-1}]$.
We define
$\Gg'=\Gg_l\oplus\BC c$, a central extension of $\Gg_l$ given by
$$[a\otimes t^m,b\otimes t^n]=[a,b]\otimes t^{m+n}+
m\delta_{m,-n}\mathrm{tr}(ab)c \text{ for }a,b\in\Gsl_n(\BC).$$
Finally, we define 
$\Gg=\Gg'\oplus\BC d$. We endow it with a Lie
algebra structure where 
\begin{itemize}
\item $\Gg'$ is a Lie subalgebra
\item $[d,at^n]=nat^n$ for $a\in \Gsl_n(\BC)$
\item $[d,c]=0$.
\end{itemize}
Note that $\Gg'=[\Gg,\Gg]$.

Let $e_0=e_{n,1}\otimes t$, $f_0=e_{1,n}\otimes t^{-1}$
and $h_0=(e_{n,n}-e_{1,1})+c$.

Let $Y=\BZ h_0\oplus\cdots\oplus\BZ h_{n-1}\oplus\BZ d=Y_{\Gsl_n}\oplus
\BZ c\oplus\BZ d$ and
$X=\Hom_\BZ(Y,\BZ)$, with
$(\Lambda_0,\ldots,\Lambda_{n-1},\partial)$ the basis dual to
$(h_0,\ldots,h_{n-1},d)$.

Let $\alpha_0=\partial-(\alpha_1+\cdots+\alpha_{n-1})$
and $\alpha_0^\vee=h_0$.
This provides an identification of $\Gg$ with a Kac-Moody algebra of graph
$$\hat{A}_{n-1}=
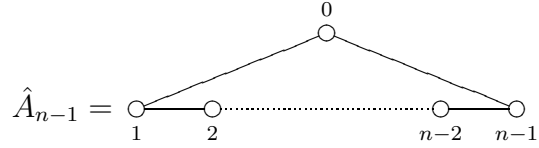
\xy 
(-10,0) *\cir<3pt>{}="0"*+!U(1.5){\scriptstyle 1},
(0,0)   *\cir<3pt>{}="1"*+!U(1.5){\scriptstyle 2},
(15,10)  *\cir<3pt>{}="2"*+!D(1.5){\scriptstyle 0},
(30,0)  *\cir<3pt>{}="3"*+!U(1.5){\scriptstyle n-2},
(40,0)  *\cir<3pt>{}="4"*+!U(1.5){\scriptstyle n-1},
"0";"1" **\dir{-},
"1";"3" **\dir{.},
"3";"4" **\dir{-},
"0";"2" **\dir{-},
"4";"2" **\dir{-},
\endxy
$$

\subsubsection{Fock spaces}
\label{se:Fock}
We recall in this section some classical results on representations of
symmetric groups and related Hecke algebras, and the relation with Fock spaces
\cite{Ari,Gr,Kl,Ma}.

\smallskip
Let $\CF$ be the complex vector space with basis all partitions. Let
$p\ge 2$ be an integer.

Let us construct an action of $\hat{\Gsl}_p$ on $\CF$. Let
$\lambda$ be a partition. We consider the associated Young diagram, whose
boxes we number modulo $p$. We define
$e_i(\lambda)$ (resp. $f_i(\lambda)$)
as the sum of the partitions obtained by removing (resp.
adding) an $i$-node to $\lambda$.
We put $d(\lambda)=N_0(\lambda)\lambda$, where $N_0(\lambda)$ is the
number of $0$-nodes of $\lambda$.

\begin{example}
Let us consider for example $p=3$ and $\lambda=(3,1)$.
$$
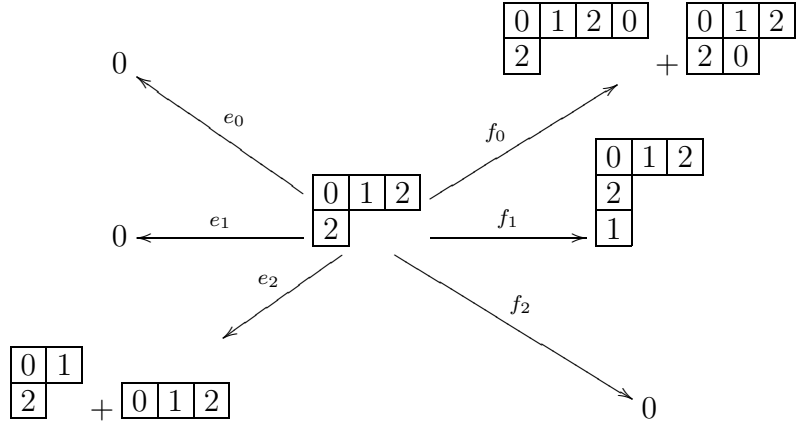
\xymatrix{
0 && {\protect \young(0120,2)}+{\protect\young(012,20)}\\
0 & {\protect \young(012,2)} \ar[ul]_-{e_0} \ar[l]_-{e_1} \ar[dl]_-{e_2}
 \ar[ur]^{f_0} \ar[r]^{f_1} \ar[dr]^{f_2} & {\protect \young(012,2,1)} \\
{\protect\young(01,2)}+{\protect\young(012)} && 0
}$$
\end{example}

The construction above defines an action of $\hat{\Gsl}_p$ on $\CF$, where
$c$ acts by $1$ (\ie, the level is $1$). This defines an object
of $\CO^{\mathrm{int}}$.

\medskip
Let $K_0(\BQ\GS_n)$ be the Grothendieck group of the category 
$\BQ\GS_n\mMod$. It is a
free abelian group with basis the isomorphism classes of irreducible
representations of $\GS_n$ over $\BQ$.

There is an isomorphism 
$$\CF\iso \bigoplus_{n\ge 0}\BC\otimes K_0(\BQ\GS_n)$$
It sends a partition $\lambda$ to the class of the corresponding simple
module $S_\lambda$ of $\BQ\GS_n$. We identify $\CF$ with the sum of
Grothendieck groups via this isomorphism.

We consider now a prime number $p$ and $\bar{\CF}=
\bigoplus_{n\ge 0}\BC\otimes K_0(\BF_p\GS_n)$.
The decomposition map defines a surjective morphism of abelian groups
$\mathrm{dec}:\CF\to\bar{\CF}$. There is a $\BZ\GS_n$-module  $\tS_\lambda$, free
over $\BZ$, such that $S_\lambda\simeq \BQ\otimes_\BZ \tS_\lambda$. We have
$\mathrm{dec}([S_\lambda])=[\BF_p\otimes_\BZ \tilde{S}_\lambda]$.

\smallskip
The action of $\hat{\Gsl}_p$ on $\CF$ stabilizes the kernel
of the decomposition map: this provides us with an action on
$\bar{\CF}$. The action of $\hat{\Gsl}_p$ on $\bar{\CF}$ is irreducible and
$\bar{\CF}\simeq L(\Lambda_0)$.

\medskip
Let $\CV=\bigoplus_{n\ge 0}\BF_p\GS_n\mMod$. Define
$E=\bigoplus_{n\ge 0}\Res_{\GS_n}^{\GS_{n+1}}$ and
$F=\bigoplus_{n\ge 0}\Ind_{\GS_n}^{\GS_{n+1}}$, two exact endofunctors
of $\CV$. They are left and right adjoint. We
have $\Ind_{\GS_n}^{\GS_{n+1}}=\BF_p\GS_{n+1}\otimes_{\BF_p\GS_n}-$.

Left multiplication by $(1,n+1)+\cdots+(n,n+1)$ defines an
endomorphism of the $(\BF_p\GS_{n+1},\BF_p\GS_n)$-bimodule 
$\BF_p\GS_{n+1}$, hence an endomorphism of the functor
$\Ind_{\GS_n}^{\GS_{n+1}}$. We denote by $X$ the corresponding
endomorphism of $F$.

\smallskip
Given $M$ an $\BF_p\GS_n$-module, all eigenvalues of $X$ acting
on $F(M)=\Ind_{\GS_n}^{\GS_{n+1}}M$ are in $\BF_p$. We denote by
$F_i(M)$ the generalized $i$-eigenspace of $X$, for $i\in\BF_p$. This
gives us a decomposition $F=\bigoplus_{i\in\BF_p}F_i$. Similarly,
we have a decomposition
$E=\bigoplus_{i\in\BF_p}E_i$, where $E_i$ is left and right adjoint to
$F_i$.

The following proposition shows that the action of $\hat{\Gsl}_n'$ on
$\bar{\CF}$ comes from the $i$-induction and $i$-restriction functors.

\begin{prop}
Given $M\in\BF_p\GS_n\mMod$, we
have $[E_i(M)]=e_i([M])$ and $[F_i(M)]=f_i([M])$.
\end{prop}

We denote by $\hat{\GS}_p$ the affine symmetric group, a Coxeter group
of type $\hat{A}_{p-1}$.

\begin{prop}
The decomposition of $\bar{\CF}$ into weight spaces corresponds to
the decomposition into blocks. Two blocks are in the same orbit under
the adjoint action of $\hat{\GS}_p$ if and only if they have
the same defect.
\end{prop}

\medskip
Let $k$ be a field and $q\in k^\times$ be an element with finite order
$p\ge 2$ ($p$ needs not be a prime). The construction above extends with
$\BQ\GS_n$ replaced by 
$H_n^f\otimes_{k[q_1,q_2]}\bigl(k[q_2]/(q_2+1)\bigr)(q_1)$ and 
$\BF_p\GS_n$ replaced by $H_n^f\otimes_{k[q_1,q_2]}k[q_1,q_2]/(q_2+1,q_1-q)$.
This provides a realization of $L(\lambda_0)$ as 
$\bigoplus_{n\ge 0}\BC\otimes
K_0(H_n^f\otimes_{k[q_1,q_2]}k[q_1,q_2]/(q_2+1,q_1-q))$, via $i$-induction
and $i$-restriction functors.

\subsection{$2$-categories}

\subsubsection{Duals}
Let $\CC$ be a strict monoidal category and $V\in\CC$. A {\em right
dual} to $V$ is an object $V^*\in\CC$ together with maps
$\eps_V:V\otimes V^*\to 1$ and $\eta_V:1\to V^*\otimes V$ such that
the following two compositions are the identity maps:
$$V\xrightarrow{V\otimes\eta_V}V\otimes V^*\otimes V
\xrightarrow{\eps_V\otimes V}V \text{ and }
V^*\xrightarrow{\eta_V\otimes V^*}V^*\otimes V\otimes V^*
\xrightarrow{V^*\otimes\eps_V}V^*.$$

When $\CC$ is the category of finite dimensional vector spaces,
we obtain the usual dual.

When $\CC$ is the category of endofunctors of a category, the notion
of right dual coincides with that of right adjoint.

\subsubsection{$2$-categories}
A strict $2$-category $\FC$ is a category enriched in categories:
\ie, it is the data of a set of objects, and given
$M,N$ two objects, the data of a category $\CHom(M,N)$ together
with composition
functors $\CHom(L,M)\times\CHom(M,N)\to \CHom(L,N)$ satisfying
associativity conditions. We also require that $\CEnd(M)$ comes with an
identity object for the composition, which makes it into a strict monoidal
category.

We can think of this as the data of objects, $1$-arrows (the objects
of the categories $\CHom(M,N)$) and $2$-arrows (the arrows of the
categories $\CHom(M,N)$).

\smallskip
A strict monoidal category $\CC$ is the same data as a strict $2$-category
with one object $\ast$ and $\CEnd(\ast)=\CC$.

\smallskip
While the typical example of a category is the category of sets,
the typical example of a strict $2$-category is the $2$-category 
$\FC at$ of categories: its objects are categories, and
$\CHom(\CC,\CC')$ is the category of functors $\CC\to\CC'$.

A related example of a $2$-category is that $\FB imod$
of bimodules: its objects are
algebras over a fixed commutative ring $k$ and $\CHom(A,B)$
is the category of $(B,A)$-bimodules. Composition is given by
tensor product.

\subsubsection{$2$-Kac Moody algebras}
\label{se:def2km}
We come now to the definition of $2$-Kac-Moody algebras \cite[\S 4.1.3]{Rou3}.
Our aim now is to add $E_s$'s to $\CB$ and construct maps
which we will make formally invertible to enforce the relations
$[e_i,f_i]_{|\lambda}=\langle\lambda,\alpha_i^\vee\rangle$. In order to make
sense of this, we will need to add formally ``idempotents'' corresponding
to weights $\lambda\in X$: this requires moving from a monoidal
category to a $2$-category.

\smallskip
Let $C$ be a generalized Cartan matrix and let
$\CB'$ be the strict monoidal $k$-linear category obtained from
$\CB=\CB(C)$ by adding $E_s$
right dual to $F_s$ for every $s\in I$.
Define
$$\eps_s=\eps_{F_s}:F_s E_s\to \idun\text{ and } \eta_s=\eta_{F_s}:
\idun\to E_s F_s.$$

\smallskip
Consider now a root datum $(X,Y,\langle-,-\rangle,\{\alpha_i\}_{i\in I},
\{\alpha_i^\vee\}_{i\in I})$ of type $C$.

Consider the strict $2$-category $\FA'$ with set of objects $X$ and
where $\CHom(\lambda,\lambda')$ is
the full $k$-linear subcategory of $\CB'$
with objects direct sums of objects of the form
$E_{s_n}^{a_n}F_{t_n}^{b_n}\cdots E_{s_1}^{a_1}F_{t_1}^{b_1}$
where $a_l,b_l\ge 0$, $s_l,t_l\in I$ and
$\lambda'-\lambda=\sum_l (a_l\alpha_{s_l}-b_l\alpha_{t_l})$.

\smallskip
Let
$\FA=\FA(\Gg)$ be the $k$-linear strict
$2$-category deduced from $\FA'$ by inverting the following
$2$-arrows:
\begin{itemize}
\item  when $\langle\lambda,\alpha_s^\vee\rangle\le 0$,
$$\rho_{s,\lambda}=
\sigma_{ss}+\sum_{i=0}^{-\langle\lambda,\alpha_s^\vee\rangle+1}\eps_s\circ
(x_s^i E_s):
F_s E_s\idun_\lambda\to E_s F_s\idun_\lambda \oplus
\idun_\lambda^{-\langle\lambda,\alpha_s^\vee\rangle}$$
\item  when $\langle\lambda,\alpha_s^\vee\rangle\ge 0$,
$$\rho_{s,\lambda}=
\sigma_{ss}+\sum_{i=0}^{-1+\langle\lambda,\alpha_s^\vee\rangle}
(E_s x_s^i)\circ \eta_s:
F_s E_s\idun_\lambda\oplus
 \idun_\lambda^{\langle\lambda,\alpha_s^\vee\rangle}
\to E_s F_s\idun_\lambda$$
\item $\sigma_{st}:F_sE_t\idun_\lambda\to E_tF_s\idun_\lambda$
 for all $s\not=t$ and all $\lambda$
\end{itemize}

where we define
$$\sigma_{st}=(E_t F_s\eps_t)\circ (E_t \tau_{ts}E_s)\circ
(\eta_t F_sE_t):F_s E_t\to E_t F_s.$$

\smallskip
We put $\FA^{\mathrm{gr}}=\FA\otimes_k k^{\mathrm{gr}}$ as in
\S \ref{se:defhalf}.
The grading defined on $\CB^{\mathrm{gr}}$ extends to a grading on 
$\FA^{\mathrm{gr}}$ with
$\deg \eps_{s,\lambda}=
d_s(1+\langle\lambda,\alpha_s^\vee\rangle)$ and
$\deg \eta_{s,\lambda}=d_s(1-\langle\lambda,\alpha_s^\vee\rangle)$.

Given a quiver $\Gamma$ with no loops corresponding to $C$, we put
$\FA^\Gamma=\FA\otimes_k k^\Gamma$, a $2$-category with graded spaces
of $2$-arrows.

\begin{rem}
One shows easily by passing to the Grothendieck groups, the graded
category
$\FA^{\mathrm{gr}}\,\mgr$ gives rise to the relations (\ref{re:quantumKM}) of
\S \ref{se:KacMoody}. Khovanov and Lauda have constructed a related
$2$-category and shown that the canonical morphism from the integral form
of $U_q(\Gg)$ to the $K_0$ is an isomorphism in type $A_n$
\cite{Lau1,KhoLau3}.
\end{rem}

It can be shown that the isomorphisms 
$\rho_{i,\lambda}$ give rise to commutation isomorphisms between
$E_i^m$ and $F_i^n$ (cf \cite[Lemma 4.12]{Rou3} for a version with explicit
isomorphisms).

\begin{lemma}[{\protect \cite[[Lemma 4.12]{Rou3}}]
\label{le:commutationEF}
Let $m,n\ge 0$, $\lambda\in X$ and $i\in I$. Let $r=m-n+\langle \lambda,
\alpha_i^\vee\rangle$. There are isomorphisms
$$F_i^nE_i^m  \idun_{\lambda}\iso \bigoplus_{l=0}^{\min(m,n)}
E_i^{m-l}F_i^{n-l} \idun_{\lambda}\otimes_k
k^{\frac{m!n!}{(m-l)!(n-l)!}{-r\choose l}}
 \text{ if }r\le 0$$
$$\bigoplus_{l=0}^{\min(m,n)}
F_i^{n-l}E_i^{m-l} \idun_{\lambda}\otimes_k
k^{\frac{m!n!}{(m-l)!(n-l)!}{r\choose l}}
\iso E_i^mF_i^n\idun_{\lambda}
\text{ if }r\ge 0.$$
\end{lemma}

\begin{rem}
We have chosen to switch the roles of $E$ and $F$, compared to
\cite{Rou3}, as we will deal here with highest weight representations,
while in \cite{Rou3} we dealt with lowest weight representations. The
two definitions are equivalent, as there is a strict equivalence
of $2$-categories
$$I:\FA^{\opp}\iso\FA,\ \idun_\lambda\mapsto\idun_{-\lambda},\
E_s\mapsto F_s,\ F_s\mapsto E_s,\ \tau_{st}\mapsto -\tau_{ts},\
x_s\mapsto x_s.$$
Given a $2$-category $\FC$, we have denoted by $\FC^{\opp}$ the
$2$-category with the same set of objects as $\FC$ and with
$\CHom_{\FC^\opp}(c,c')=\CHom_{\FC}(c,c')^\opp$.
\end{rem}

\subsection{$2$-representation theory}
We are now reaching our main object of study. We review 
\cite[\S 5]{Rou3} (based in part on \cite{ChRou}) and provide some complements.

\subsubsection{Integrable $2$-representations}
A representation of $\FA$ on $k$-linear categories
is defined to be a strict $2$-functor from $\FA$ to the
strict $2$-category of $k$-linear categories.
This is the same thing as the data of
\begin{itemize}
\item a $k$-linear category $\CV_\lambda$ for $\lambda\in X$
\item a $k$-linear functor $F_i:\CV_\lambda\to\CV_{\lambda-\alpha_i}$ for
$\lambda\in X$ and $i\in I$ admitting a right adjoint $E_i$
\item $x_i\in\End(F_i)$ and $\tau_{ij}\in\Hom(F_iF_j,F_jF_i)$
\end{itemize}
such that
\begin{itemize}
\item the quiver Hecke algebra relations for $x_i$ and $\tau_{ij}$
\item the maps $\rho_{i,\lambda}$ and $\sigma_{ij}$ ($i\not=j$)
are isomorphisms.
\end{itemize}
A representation of $\FA$ such that 
$E_i$ and $F_i$ are locally nilpotent for all $i$ will be called 
an {\em integrable $2$-representation} of $\FA$ (or of $\Gg$).

\smallskip
In the definition on an
integrable $2$-representations, the condition that the maps $\sigma_{ij}$
are isomorphisms for $i\not=j$ is a consequence of the other conditions
\cite[Theorem 5.25]{Rou3}.

\begin{rem}
One can equivalently start with the functors $E_i$'s, and natural
transformations $x_i$ and $\tau_{ij}$ between products of $E$'s.
\end{rem}

The definition provides $E_i$ as a right adjoint of $F_i$, but the next
result shows that $E_i$ will actually also be a left adjoint
(cf \cite[\S 4.1.4]{Rou2} for the explicit units and counits).

\begin{thm}[{\protect \cite[Theorem 5.16]{Rou2}}]
Let $\CV$ be an integrable $2$-representation of $\FA$.
Then $E_i$ is a left adjoint of $F_i$, for all $i$.
\end{thm}

\smallskip
It is often unreasonable to check directly that the maps
$\rho_{i,\lambda}$ and $\sigma_{ij}$ are isomorphisms in examples. It turns
out that, under finiteness assumptions, it is enough to check that the
$\Gsl_2$-relations hold on the Grothendieck group
(the crucial part is \cite[Theorem 5.27]{ChRou}).

\begin{thm}[{\protect  \cite[Theorem 5.27]{Rou3}}]
\label{th:enoughK0}
Let $K$ be a field that is a $K$-algebra
and let $\CV=\bigoplus_{\lambda\in X}\CV_\lambda$
be a $K$-linear abelian category such that
all objects have finite composition series and simple objects have
endomorphism ring $K$.
Assume given
\begin{itemize}
\item a $K$-linear exact
functor $F_i:\CV_\lambda\to\CV_{\lambda-\alpha_i}$ for
$\lambda\in X$ and $i\in I$ with an exact right adjoint $E_i$
\item $x_i\in\End(F_i)$ and $\tau_{ij}\in\Hom(F_iF_j,F_jF_i)$
\end{itemize}
such that
\begin{itemize}
\item $E_i$ and $F_i$ are locally nilpotent
\item $E_i$ is left adjoint to $F_i$
\item the quiver Hecke algebra relations for $x_i$ and $\tau_{ij}$
\item the endomorphisms $[E_i]$ and $[F_i]$ define an integrable
representation of $\Gsl_2$ on $\BC\otimes K_0(\CV)$, and
$[E_i][F_i]-[F_i][E_i]$ acts by $\langle\lambda,\alpha_i^\vee\rangle$ on
$K_0(\CV_\lambda)$, for all $i$ and $\lambda$.
\end{itemize}
Then, the data above defines an integrable $2$-representation of $\Gg$ on $\CV$.
\end{thm}

Let us give a variant, based on ``abstract'' $\Gsl_2$-relations between
functors.

\begin{cor}
\label{cor:additiveK0}
Let $k'$ be a commutative $k$-algebra.
Let $\{\CV_\lambda\}_{\lambda\in X}$ be
a family of $k'$-linear categories whose $\Hom$'s are finitely generated
$k'$-modules.

Assume given
\begin{itemize}
\item
$F_s:\CV_\lambda\to\CV_{\lambda-\alpha_s}$ with a right adjoint $E_s$
for $s\in I$
\item
$x_s\in\End(F_s)$ and $\tau_{st}\in\Hom(F_sF_t,F_tF_s)$ for every $s,t\in I$.
\end{itemize}

We assume that
\begin{itemize}
\item $E_s$ is a left adjoint of $F_s$
\item $E_s$ and $F_s$ are locally nilpotent
\item given $\lambda\in X$, there are isomorphisms of functors
$$(E_sF_s)_{|\CV_\lambda}\simeq (F_sE_s)_{|\CV_\lambda}\oplus
\Id_{\CV_\lambda}^{\langle\lambda,\alpha_s^\vee\rangle} \text{ if }
\langle\lambda,\alpha_s^\vee\rangle\ge 0$$
$$(F_sE_s)_{|\CV_\lambda}\simeq (E_sF_s)_{|\CV_\lambda}\oplus
\Id_{\CV_\lambda}^{-\langle\lambda,\alpha_s^\vee\rangle} \text{ if }
\langle\lambda,\alpha_s^\vee\rangle\le 0$$
\item the quiver Hecke algebra relations hold.
\end{itemize}

Then, the data above
defines an integrable action of $\FA(\Gg)$ on
$\CV=\bigoplus_\lambda\CV_\lambda$.
\end{cor}

\begin{proof}
Let $K$ be an algebraically closed field that is a $k'$-algebra.
Let $\CW=\CV\mMod_K$ be the category of $k'$-linear
functors $\CV^\opp\to K\mMod$. The functors $F_s$ and
$E_s$ induce adjoint exact functors on $\CW$ satisfying the conditions
of Theorem \ref{th:enoughK0}. Consequently, the maps
$\rho_{s,\lambda}$ and $\sigma_{st}$ (for $s\not=t$), taken in $\CV$,
are isomorphisms after applying $-\otimes_{k'}K$. Since this holds
for all $K$, a variant of Lemma \ref{le:isoatclosedpoints} shows that those
maps are isomorphisms.
\end{proof}

\subsubsection{Some $2$-representations of $\Gsl_2$}
\label{se:repsl2}
We assume here $|I|=1$, $X=\BZ$ and $\alpha=2$. We have $k=\BZ$.
Fix a field $K$.

The most obvious example of a $2$-representation is
$\bar{\CL}(0)$ defined by $\bar{\CL}(0)_\lambda=0$ for $\lambda\not=0$ and
$\bar{\CL}(0)_0=K\mMod$. All the extra data vanishes.

\smallskip
Consider now $\bar{\CL}(1)$, a categorification of the simple $2$-dimensional
representation of $\Gsl_2$. We put
$$\bar{\CL}(1)_\lambda=0 \text{ for } \lambda{\not=}\pm 1,\ 
\bar{\CL}(1)_1=K\mMod \text{ and }
\bar{\CL}(1)_{-1}=K\mMod.$$
We define $E$ and $F$ to be the identity functors between $\bar{\CL}(1)_1$ and
$\bar{\CL}(1)_{-1}$ and we set $x=\tau=0$.

\smallskip
A categorification of the simple $3$-dimensional representation is given
by
$$\bar{\CL}(2)_\lambda=0 \text{ for } \lambda{\not=}-2,0,2,\ 
\bar{\CL}(2)_{-2}=\bar{\CL}(2)_2=K\mMod \text{ and }
\bar{\CL}(2)_0=(K[y]/y^2)\mMod.$$
The functors $E$ and $F$ are induction and restriction functors.
We define $x$ as multiplication by $-y$ on $F=\Ind:\bar{\CL}(2)_{2}\to\bar{\CL}(2)_0$ and
as multiplication by $y$ on $F=\Res:\bar{\CL}(2)_0\to\bar{\CL}(2)_{-2}$.
We define $\tau\in\End_K(K[y]/y^2)$ by $\tau(1)=0$ and $\tau(y)=1$.

\medskip
Let us construct more generally $\bar{\CL}(n)$.
Let $H_{i,n}$ be the subalgebra
of ${^0H}_n$ generated by ${^0H}_i$ and 
$P_n^{\GS_n}$. We have
$H_{i,n}={^0H}_i^f\otimes_\BZ P_n^{\GS\{i+1,\ldots,n\}}$ as 
$\BZ$-modules and
$H_{i,n}={^0H}_i\otimes_\BZ \BZ[X_{i+1},\ldots,X_n]^{\GS\{i+1,\ldots,n\}}$ as 
algebras. By Proposition \ref{pr:MoritanilHecke}, we have a Morita equivalence
between
the $P_n^{\GS_n}$-algebras $H_{i,n}$ and $P_n^{\GS\{1,\ldots,i\}\times
\GS\{i+1,\ldots,n\}}$. Since ${^0H}_i$ is a symmetric algebra over
$P_i^{\GS_i}$ (Proposition \ref{pr:symmnilaffine}) and 
$P_n^{\GS\{1,\ldots,i\}\times \GS\{i+1,\ldots,n\}}$ is symmetric over
$P_n^{\GS_n}$ (Corollary \ref{co:symmparab}), we deduce from
Lemma \ref{le:2outof3} that $H_{i,n}$ is a symmetric algebra over $P_n^{\GS_n}$.

 We have a chain of algebras
$$H_{0,n}=P_n^{\GS_n}\subset H_{1,n}\subset\cdots\subset
H_{n,n}={^0H}_n$$
and $H_{i+1,n}$ is a free left (and right) $H_{i,n}$-module of
rank $(i+1)(n-i)$.

\smallskip
Let $\bar{H}_{i,n}=H_{i,n}\otimes_{P_n^{\GS_n}}K$, where
the morphism of rings $P_n^{\GS_n}\to K$ is given by sending
homogeneous polynomials of positive degree to $0$. This is a finite-dimensional
$K$-algebra Morita-equivalent to its center
$P_n^{\GS\{1,\ldots,i\}\times \GS\{i+1,\ldots,n\}}\otimes_{P_n^{\GS_n}}K$.
That center is $\BZ_{\ge 0}$-graded, with degree $0$ part of dimension $1$,
hence it is local. It follows that $\bar{H}_{i,n}$ has a unique simple
module, of dimension $i!$.

We put $\bar{\CL}(n)_\lambda=\bar{H}_{(n-\lambda)/2,n}\mMod$
for $\lambda\in\{n,n-2,\ldots,2-n,-n\}$ and $\bar{\CL}(n)_\lambda=0$
otherwise.

We denote by $E$ the restriction functor and
$F$ the induction functor: they are both exact functors. Since the
algebras $\bar{H}_{i,n}$ are symmetric over $K$, we deduce that
$E$ is both right and left adjoint to $F$.
It is immediate to check that $[E]$ and $[F]$ induce an action
of $\Gsl_2(\BC)$ on $\BC\otimes K_0(\bar{\CL}(n))=\BC^{n+1}$.

We denote by $x$ the endomorphism
of the $(\bar{H}_{i+1,n},\bar{H}_{i,n})$-bimodule
$\bar{H}_{i+1,n}$ given by right multiplication by $X_{i+1}$: this
provides a corresponding endomorphism of the functor $F$. Similarly,
we define an endomorphism $\tau$ of $F^2$ corresponding to
the right multiplication by $T_i$ on
the $(\bar{H}_{i+2,n},\bar{H}_{i,n})$-bimodule
$\bar{H}_{i+2,n}$. 

\smallskip
Theorem \ref{th:enoughK0} provides the following result. These
are the ``minimal categorifications'' of \cite[\S 5.3]{ChRou}.

\begin{prop}
\label{pr:minsl2}
The data above defines an action of $\FA$ on $\bar{\CL}(n)$.
\end{prop}

Let us now consider a deformed additive version $\CL(n)$. They are
necessary to have Jordan-H\"older type decompositions in the
additive setting. We
put $\CL(n)_\lambda=H_{(n-\lambda)/2,n}\mproj$, and we define
$E$, $F$, $X$ and $T$ as above. Proposition \ref{pr:minsl2} shows
that the morphisms of bimodules corresponding to the
maps $\rho_{\lambda}$ become isomorphisms after applying
$\otimes_{P_n^{\GS_n}}K$, for any field $K$. A graded version of
Lemma \ref{le:isoatclosedpoints} enables us to deduce that the maps
$\rho_{\lambda}$ are
isomorphisms. We obtain consequently the following proposition.

\begin{prop}
The data above defines an action of $\FA$ on $\CL(n)$.
\end{prop}

The categories $\CL(n)$ and $\bar{\CL}(n)$ are enriched in graded
vector spaces, and the actions are compatible with the gradings.

\subsubsection{Simple $2$-representations $\CL(\lambda)$}
\label{se:simplerep}
Let $\lambda\in X$. Given
$\CV$ a $2$-representation, we put
$$\CV_\lambda^{\mathrm{hw}}=\{M\in\CV_\lambda | E_i(M)=0
\ \forall i\}.$$
Note that $\CV_{\lambda}^{\mathrm{hw}}=0$ if $\lambda{\not\in}X^+$.

A consequence of the relations in Lemma \ref{le:commutationEF} is the
following description of highest weight objects.
\begin{lemma}
\label{le:criterionhw}
Let $\lambda\in X^+$ and $i\in I$. Let 
$d=\langle \lambda,\alpha_i^\vee\rangle+1$. Then
\begin{itemize}
\item $E_i \idun_\lambda$ is a direct summand of $E_i^{d+1}F_i^d \idun_\lambda$
\item $F_i^d \idun_\lambda$ is a direct summand of $F_i^{d+1}E_i \idun_\lambda$.
\end{itemize}
As a consequence, given $\CV$ a $2$-representation and $M\in\CV_\lambda$,
we have $M\in\CV^{\textrm{hw}}_\lambda$ if and only if
$F_i^{\langle \lambda,\alpha_i^\vee\rangle+1}(M)=0$ for all $i\in I$.
\end{lemma}

Assume
$\lambda\in X^+$.
There is a $2$-representation $\CL(\lambda)$ with an object 
$v_\lambda\in\CL(\lambda)_\lambda$ that has the following
property: given $\CV$ a $2$-representation, there is an equivalence
$$\CHom_{\FA}(\CL(\lambda),\CV)\iso \CV_\lambda^{\mathrm{hw}},
\ \Phi\mapsto \Phi(v_\lambda).$$
So, $\CL(\lambda)$ represents the $2$-functor
$\CV\mapsto\CV_\lambda^{\mathrm{hw}}$ and it thus unique up to an
equivalence unique up to a unique isomorphism.

\smallskip
Let us provide a construction of $\CL(\lambda)$.
We define a $2$-representation $\CM(\lambda)$ by setting
$\CM(\lambda)_\mu=\CHom_{\FA}(\lambda,\mu)$. The composition in
$\FA$ provides a natural action of $\FA$ on $\CM(\lambda)$.
Define now $\CN(\lambda)$, a sub-$2$-representation, by setting
$\CN(\lambda)_\mu$ as the full additive subcategory of
$\CM(\lambda)_\mu$ generated by objects of the form
$RE_i$, where $R$ is a $1$-arrow in $\FA$ from $\lambda+\alpha_i$ to
$\mu$ and $i\in I$.

We put now $\CL(\lambda)=\CM(\lambda)/\CN(\lambda)$ (quotient as
additive categories) and we denote by $v_\lambda$ the image of
$\idun_\lambda$. We put $Z_\lambda=\End_{\CL(\lambda)}(v_\lambda)$.

\begin{rem}
A important fact is that this construction provides a higher version of
$L^{\max}(\lambda)$ ($=L(\lambda)$ in the symmetrizable case),
 not of $\Delta(\lambda)$: this is a consequence of Lemma
\ref{le:criterionhw}.
\end{rem}

\subsubsection{Isotypic $2$-representations}
\label{se:isotypic}
An isotypic representation is a multiple of a simple representation, or
equivalently, the tensor product of a simple representation by a 
multiplicity vector space. The $2$-categorical version of that requires
to take a tensor product by a multiplicity category.

Let $A$ be a commutative ring and
$\CC$ be an $A$-linear category. Let $B$ a commutative $A$-algebra.
We denote by $\CC\otimes_AB$ the $B$-linear category with same
objects as $\CC$ and with $\Hom_{\CC\otimes_AB}(M,N)=
\Hom_{\CC}(M,N)\otimes_AB$.
Let now $\CC'$ be another $A$-linear category. We denote by
$\CC\otimes_A\CC'$ the $A$-linear category with objects finite
families $((M_1,M'_1),\ldots,(M_m,M'_m))$ where $M_i\in\CC$ and
$M'_i\in\CC'$. We put
$$\Hom(((M_1,M'_1),\ldots,(M_m,M'_m)),
((N_1,N'_1),\ldots,(N_n,N'_n)))=\bigoplus_{i,j}
\Hom_{\CC}(M_i,N_j)\otimes_A \Hom_{\CC'}(M'_i,N'_j).$$

Let $\CV$ be a $k$-linear $2$-representation of $\FA$ and let
$\lambda\in X^+$. There is a canonical fully faithful functor compatible
with the $\FA$-action
$$R_\lambda:\CL(\lambda)\otimes_{Z_\lambda}\CV_\lambda^{\mathrm{hw}}\to \CV,\
M\otimes N\mapsto M(N).$$
If $\CV$ is idempotent-closed and every object of $\CV$ is a direct
summand of an object of $\FA(\CV_\lambda^{\mathrm{hw}})$, then
$R_\lambda$ induces an equivalence
$\bigl(\CL(\lambda)\otimes_{Z_\lambda}\CV_\lambda^{\mathrm{hw}}\bigr)^i\iso
\CV$.

\smallskip
The only full sub-$2$-representations of isotypic $2$-representations are the
obvious ones.

\begin{prop}
Let $\lambda\in X^+$, let $\CM$ be a $Z_\lambda$-linear category, and
let $\CW$ be an idempotent-complete full $k$-linear
sub-$2$-representation of
$\bigl(\CL(\lambda)\otimes_{Z_\lambda}\CM\bigr)^i$.
Let $\CN$ be the subcategory of $\CM^i$ image of $\CW_\lambda$
under the canonical equivalence 
$\bigl(\CL(\lambda)\otimes_{Z_\lambda}\CM\bigr)^i_\lambda\iso\CM^i$

Then
$\CW=\bigl(\CL(\lambda)\otimes_{Z_{\lambda}} \CN\bigr)^i$.
\end{prop}

\begin{proof}
Let $\CV=\bigl(\CL(\lambda)\otimes_{Z_\lambda}\CM\bigr)^i$.
Every object $M$
of $\CV$ is a direct summand of a direct sum of objects of the form
$F_{i_r}\cdots F_{i_1}(N)$, where $N\in\CV_\lambda$. Since
$E_{i_1}\cdots E_{i_r}$ is right adjoint to
$F_{i_r}\cdots F_{i_1}$, we deduce that 
$F_{i_r}\cdots F_{i_1}$ is a direct summand of 
$F_{i_r}\cdots F_{i_1}E_{i_1}\cdots E_{i_r}F_{i_r}\cdots F_{i_1}$ (in $\FA$).
As a consequence, any $M\in \CV$ is a direct summand 
of a direct sum of objects of the form
$F_{i_r}\cdots F_{i_1}E_{i_1}\cdots E_{i_r}(M)$, where
$E_{i_1}\cdots E_{i_r}(M)\in\CV_\lambda$.

This shows that every object of $\CW$ is a direct summand of an object of
$\FA(\CW_\lambda)$, hence the canonical functor
$\bigl(\CL(\lambda)\otimes_{Z_\lambda}\CW_\lambda\bigr)^i\to\CW$ is an
equivalence.
\end{proof}

\subsubsection{Structure}
We explain here a counterpart of Jordan-H\"older series.
This provides a powerful tool to reduce
statements to the case of $\CL(\lambda)$'s and this is one the key ideas 
of \cite{ChRou}.

\smallskip
Let $\CV$ be a $2$-representation. Given
$\xi\in X/(\bigoplus_i \BZ\alpha_i)$, let
$\CV_\xi=\bigoplus_{\lambda\in\xi}\CV_\lambda$. Then
$\CV=\bigoplus_{\xi}\CV_\xi$ is a decomposition as a direct sum
of $2$-representations. This gives a direct sum decomposition of 
the $2$-category of $2$-representations.

\begin{thm}[{\protect \cite[Theorem 5.8]{Rou3}}]
\label{th:JordanHolder}
Let $\CV$ be a $k$-linear category acted on by $\FA$.
Assume that given $\lambda\in X$ and $M\in\CV_\lambda$, there is $r>0$ such
that $E_{i_1}\cdots E_{i_r}(M)=0$ for all $i_1,\ldots,i_r\in I$.

Then $\CV$ is integrable and
there are $Z_\lambda$-linear categories $\CM_{\lambda,r}$ for
$\lambda\in X^+$ and $r\in\BZ_{\ge 1}$, there is a filtration
by full $k$-linear sub-$2$-representations closed under
taking direct summands
$$0=\CV\{0\}\subset \CV\{1\}\subset\cdots\subset\CV$$
with $\bigcup_r \CV\{r\}=\CV$, 
and there are isomorphisms of $2$-representations
$$\bigl(\CV\{r+1\}/\CV\{r\}\bigr)^i\iso
 \bigoplus_{\lambda\in X^+}
\left(\CL(\lambda)\otimes_{Z_\lambda}
\CM_{\lambda,r}\right)^i.$$
\end{thm}

Note that the assumption of the theorem is automatically satisfied if
$\Gg$ is a finite-dimensional Lie algebra.

\begin{prop}
Assume $C$ is a finite Cartan matrix, \ie, $\Gg$ is a finite-dimensional
Lie algebra. Let $\CV$ be an integrable $2$-representation of $\FA$.
Then  given $\lambda\in X$ and $M\in\CV_\lambda$, there is $r>0$ such
that $E_{i_1}\cdots E_{i_r}(M)=0$ for all $i_1,\ldots,i_r\in I$.
\end{prop}

\begin{proof}
We switch the roles of $E$'s and $F$'s in the proof.
There are positive integers $n_i$, such that $F_i^{n_i}(M)=0$.
It follows that the canonical $\CB$-functor $\CB\to\CV,\ L\mapsto
L(M)$, factors through the additive quotient $\bar{\CB}$ of $\CB$ by
its full additive subcategory generated by the $\CB F_i^{n_i}$ for
$i\in I$. We will be done by showing that
there is $r>0$ such that 
$\bar{\CB}_r=0$

Note that the corresponding property holds in
$\BC\otimes K_0(\bar{\CB})$, which is a quotient of the vector space
$U(\Gn^-)/(\sum_i U(\Gn^-)f_i^{n_i})$ by Theorem \ref{th:categB}:
 this is an $L(\mu)$ for some
$\mu\in X^+$, hence it is finite-dimensional (cf \S \ref{se:Oint}).
We deduce that $K_0(\bar{\CB}_r)=0$ for some $r$. The $\Hom$-spaces in
$\CB_r$ are modules of finite rank modules over, which is a noetherian
ring, and the same holds for $\bar{\CB}_r$. So, the vanishing of
$K_0$ forces $\bar{\CB}_r=0$.
\end{proof}

Extensions between $\CL(\lambda)$'s can occur only in one direction.

\begin{lemma}
Let $\CV$ be a $k$-linear $2$-representation of $\FA$ and
$\CW$ a full $k$-linear sub-$2$-representation closed under taking
direct summands. Let $\lambda\in X^+$ and let $\CM$ be a $Z_\lambda$-linear
category.

Assume there is a morphism 
$\Phi:\CV\to \CL(\lambda)\otimes_{Z_\lambda}\CM$
of $2$-representations of $\FA$ with $\Phi(\CW)=0$ and inducing an
isomorphism $\CV/\CW\iso \CL(\lambda)\otimes_{Z_\lambda}\CM$.

If $\CW_{\lambda+\alpha}=0$ for all $\alpha\in\bigoplus_{i\in I}\BZ_{\ge 0}
\alpha_i$, then there is a morphism of $2$-representations
$\Psi:\CL(\lambda)\otimes_{Z_\lambda}\CM\to\CV$ that is a right inverse
to $\Phi$. As a consequence, 
$\CV\simeq\CW\oplus \CL(\lambda)\otimes_{Z_\lambda}\CM$ as
$2$-representations of $\FA$.
\end{lemma}

\begin{proof}
By assumption, $\CV_{\lambda+\alpha_i}=0$ for all $i$. It follows that
the restriction of $\Phi$ to $\CV_\lambda$ is an equivalence
$\Phi':\CV_\lambda^{\mathrm{hw}}\iso \CM$. The functor
$\Phi^{\prime -1}$ induces a fully faithful functor
$\Psi:\CL(\lambda)\otimes_{Z_\lambda}\CM\to\CV$ that is a right inverse to
$\Phi$.
\end{proof}

As a consequence, we can order terms and obtain a Jordan-H\"older
filtration under stronger finiteness assumptions from Theorem
\ref{th:JordanHolder}.

\begin{thm}
\label{th:JordanHolderfinite}
Let $\CV$ be an integrable $k$-linear $2$-representation of $\FA$.
Assume there is a finite set $K\subset X$ such that 
$\{\lambda\in X | \CV_\lambda{\not=}0\}\subset\bigcup_{\mu\in K}
(\mu+\sum_{i\in I}\BZ_{\le 0}\alpha_i)$.

Then there are
\begin{itemize}
\item $\lambda_1,\lambda_2,\ldots\in X^+$ such that
$\lambda_a-\lambda_b\in\sum_i \BZ_{\ge 0}\alpha_i$ implies
$a<b$
\item $Z_{\lambda_r}$-linear categories $\CM_r$
\item and a filtration by full $k$-linear sub-$2$-representations closed
under taking direct summands
$$0=\CV\{0\}\subset \CV\{1\}\subset\cdots\subset\CV$$
\end{itemize}
such that $\bigcup_r \CV\{r\}=\CV$
and 
$\bigl(\CV\{r+1\}/\CV\{r\}\bigr)^i\iso \bigl(\CL(\lambda_r)
\otimes_{Z_{\lambda_r}}\CM_r\bigr)^i$ as $2$-representations of $\FA$.
\end{thm}

\subsection{Cyclotomic quiver Hecke algebras}
\subsubsection{Construction of $\CB(\lambda)$}
Let $\lambda\in X^+$. Let
$n_i=\langle \lambda,\alpha_i^\vee\rangle$,
let $A_\lambda=\BZ[\{z_{i,r}\}_{i\in I,1\le r\le n_i}]$ and let
$k_\lambda=k\otimes_\BZ A_\lambda$.
We define the
additive category quotient
$$\CB(\lambda)=(\CB\otimes_\BZ A_\lambda)/
(\sum_{r=0}^{n_i}x_i^{n_i-r}\otimes z_{i,r})_{i\in I}$$
where we put $z_{i,0}=1$ for $i\in I$.

We define now the {\em cyclotomic quiver Hecke algebras}
$$H_n(\lambda)=\End_{\CB(\lambda)}(
\bigoplus_{(i_1,\ldots,i_n)\in I^n}F_{i_1}\cdots F_{i_n}).$$

We have
$$H_n(\lambda)=(H_n(Q)\otimes_\BZ A_\lambda)/
(\sum_{r=0}^{n_i}x_{1,i}^{n_i-r}\otimes z_{i,r})_{i\in I}$$
and in particular $H_0(\lambda)=k_\lambda$.

One can also consider
the {\em reduced cyclotomic quiver Hecke algebras}
$\bar{H}_n(\lambda)=H_n(\lambda)\otimes_{A_\lambda}\BZ$, where
$z_{i,r}$ acts by $0$ on $\BZ$ for $r\not=0$.

Note that these constructions depend only on 
$\{n_i\}_{i\in I}$, not on $\lambda$.

\subsubsection{Fock spaces}
\label{se:Fock2}
We explain how to construct a $2$-representation of affine Lie algebras
of type $A$, following \cite{ChRou}, and we explain the relation
with $\CB(\Lambda_0)$. We consider the setting of \S \ref{se:Fock}.

Similarly to the construction of $X$, we construct an endomorphism $T$ of 
$F^2=\bigoplus_{n\ge 0}\Ind_{\GS_n}^{\GS_{n+2}}$ by left multiplication
by $(n+1,n+2)$ on the $(\BF_p\GS_{n+2},\BF_p\GS_n)$-bimodule 
$\BF_p\GS_{n+2}$.

We have a morphism of algebras
$$\bar{H}_n\to\End(F^n),\
X_i\mapsto F^{n-i}XF^{i-1},\ s_i\mapsto F^{n-i-1}TF^{i-1}.$$

Let $\Gamma$ be the quiver with vertex set $I=\BF_p$ and
with arrows $i\to i+1$.
Theorem \ref{th:equivdegenaffineHeckeandquiver} shows how to deduce a
morphism of algebras
$H_n(\Gamma)\to \End(\bigoplus_{\nu\in I^n}F_{\nu_1}\cdots F_{\nu_n})$.

\begin{thm}
\label{se:thmFock}
The constructions above endow $\CV$ with a $2$-representation of
$\FA(\hat{\Gsl}_p)$. We have an equivalence of $2$-representations
$$\bigl(\BF_p\otimes_{Z_{\Lambda_0}}\CL(\Lambda_0)\bigr)^i
\simeq\bigoplus_{n\ge 0}\BF_p\GS_n \mproj$$
\end{thm}

\begin{proof}
We need to show that
the maps $\rho_{i,\lambda}$ and
are isomorphisms. The corresponding
relations hold at the level of the Grothendieck groups. It follows from
Theorem \ref{th:enoughK0} that the required maps are isomorphisms.
The equivalence follows from the fact that $\CV$ is a highest weight
$2$-representation, with highest weight $\Lambda_0$.
\end{proof}

Since the left side of the equivalence is graded, this provides
us with gradings of group algebras of symmetric groups
over $\BF_p$. This can be made explicit. Indeed, Theorem
\ref{th:equivdegenaffineHeckeandquiver} induces an isomorphism of
algebras $\BF_p\otimes_\BZ\bar{H}_n(\Gamma)\iso \BF_p\GS_n$, and
the right hand side has homogeneous generators.

\smallskip
The constructions above extend to arbitrary $p\ge 2$, with the group
algebra of the symmetric group replaced by its Hecke algebra, as
explained in \S \ref{se:Fock}.

\smallskip
The isomorphism and the gradings above 
have been constructed and studied independently by Brundan and
Kleshchev \cite{BrKl1,BrKl2}. They have built a new approach to the
representation
theory of symmetric groups and their Hecke algebras using these gradings.

Such gradings had been shown to exist earlier (using
derived equivalences and good blocks) for blocks with abelian defect
\cite[Remark 3.11]{Rou2}.
Leonard Scott had raised the question in the mid-nineties to construct
gradings for group algebras of symmetric groups and more generally
Hecke algebras of finite Coxeter groups.

\subsubsection{Simple $2$-representations}
We explain here how cyclotomic quiver Hecke algebras provide a vast
generalization of the constructions of \S \ref{se:Fock} and \S \ref{se:Fock2},
as conjectured by Khovanov-Lauda and ourselves.

\smallskip
The left action of $\CB$ on itself induces an action of $\CB$ on 
$\CB(\lambda)^i=\bigoplus_n H_n(\lambda)\mproj$.

\begin{thm}[Kang-Kashiwara, Webster]
Given $s\in I$, the functor $F_s:\CB(\lambda)^i\to\CB(\lambda)^i$ has a right
adjoint. This provides an action of $\FA$ on $\CB(\lambda)^i$, with
highest weight
$(\CB(\lambda)^i)_{\lambda}=k_\lambda\mproj$.

There is an isomorphism of $\Gg$-modules $\BC\otimes_\BZ K_0(\CB(\lambda)^i)
\iso L(\lambda)$.
\end{thm}

Kang-Kashiwara's result \cite[Theorem 4.6]{KanKas} is given in the
case of symmetrizable Cartan matrices and in a graded setting, but it extends
with no change to our setting. Webster's result \cite{We1}
is also in a graded setting.

\smallskip
Note that the algebras $H_n(\lambda)$ are finitely generated
projective $k_\lambda$-modules \cite[Remark 4.20(ii)]{KanKas}.

\smallskip
There is a canonical morphism of $k$-algebras $Z_\lambda\to k_\lambda$
and an equivalence of additive $\FA$-categories (cf \S \ref{se:isotypic})
$$\Psi:\bigl(\CL(\lambda)\otimes_{Z_\lambda}k_\lambda\bigr)^i\iso
\CB(\lambda)^i$$

\begin{thm}
\label{th:LequivB}
The canonical map gives an isomorphism $Z_\lambda\iso k_\lambda$.
In particular,
there is an equivalence of categories $\CL(\lambda)^i\iso \CB(\lambda)^i$
compatible with the action of $\FA$.
\end{thm}

\begin{proof}
The proof is similar to that of \cite[Proposition 5.15]{Rou3}.
Let $i\in I$. We have $F_i^{n_i}(v_\lambda)\not=0$, while
$F_i^{n_i+1}(v_\lambda)=0$. It follows that the canonical
map $F_i^{(n_i)}E_i^{(n_i)}(v_\lambda)\to v_\lambda$ is an isomorphism
\cite[Lemma 4.12]{Rou3}. The action of $\BZ[x_{1,i},\ldots,x_{n_i,i}]^{\GS_{n_i}}$
on $E^{(n_i)}$ gives then an action on $v_\lambda$.
We let $z_{i,r}$ act on $v_\lambda$ as $(-1)^re_r(x_{1,i},\ldots,x_{n_i,i})$
for $1\le r\le n_i$. This provides a morphism of 
$k$-algebras $k_\lambda\to \End(v_\lambda)=Z_\lambda$.

We have a canonical functor $\CB\otimes_k k_\lambda\to\CL(\lambda)$ given by 
$M\mapsto M(v_\lambda)$. It induces a functor
$\Phi:\CB(\lambda)\to\CL(\lambda)$ compatible with the $\CB$-action: there
are canonical isomorphisms $\Phi F_i\iso F_i\Phi$ compatible with the
$x_i$'s and $\tau_{ij}$'s. Note that the functor
$\CL(\lambda)\xrightarrow{\can}\CL(\lambda)\otimes_{Z_\lambda}k_\lambda
\xrightarrow{\Psi}\CB(\lambda)$ is a left inverse to $\Psi$.

Let us show that $\Phi$ can be endowed with a compatibility for the
$\FA$-action. We need to show that the composition
$$\Phi E_i\xrightarrow{\eta_i\Phi E_i} E_iF_i\Phi E_i
\xrightarrow{E_i (\can^{-1}) E_i} E_i \Phi F_iE_i\xrightarrow{E_i\Phi\eps_i}
E_i\Phi$$
is an isomorphism for all $i\in I$ (cf \cite[\S 5.1.2]{ChRou}).
Let us show that the composition above is an
isomorphism when applied to $L\in\CB(\lambda)$.
Consider $M,\gamma,\alpha$ as in
Lemma \ref{le:moveEi} below. Since $\Phi\Psi(\alpha)$ is an isomorphism,
we deduce that the composition
$$\Phi(M)\xrightarrow{\Phi(\eta_i\bullet)}
\Phi(E_iF_iM)\xrightarrow{\Phi(E_i\gamma)}\Phi(E_iL)$$
is an isomorphism. To clarify the exposition, we identify
$\Phi F_i$ with $F_i\Phi$.
There is a commutative diagram
$$\xymatrix{
\Phi(E_iL)\ar[r]^-{\eta_i\bullet} &
E_iF_i\Phi(E_iL) \ar[r]^-{E_i\Phi(\eps_i\bullet)} &
E_i\Phi(L) & \\
\Phi(E_iF_iM)\ar[u]_{\Phi(E_i\gamma)}\ar[r]^-{\eta_i\bullet} &
E_iF_i\Phi(E_iF_iM)\ar[r]^-{E_i\Phi(\eps_i\bullet)}
\ar[u]_{\bullet\Phi(E_i\gamma)} &
E_i\Phi(F_iM)\ar[u]_{E_i\Phi(\gamma)} & \\
\Phi(M)\ar `d[dr] `[rrr]_-\sim `[uur] [uurr] \ar@/^3pc/[uu]^\sim
 \ar[u]_{\Phi(\eta_i\bullet)} \ar[r]^-{\eta_i\bullet} &
E_iF_i\Phi(M)\ar[ur]_-{\id}\ar[u]_{\bullet\Phi(\eta_iM)} & & \\
&&&
}$$
This completes the proof that $\Phi$ is compatible with the action of
$\FA$. Since $\Phi(k_\lambda)=v_\lambda$,
we obtain a morphism of $Z_\lambda$-algebras $k_\lambda\to Z_\lambda$ that
is a left inverse to the canonical morphism $Z_\lambda\to k_\lambda$. 
Consequently, these morphisms are isomorphisms.
\end{proof}

\begin{lemma}
\label{le:moveEi}
Given $i$ and $L\in\CB(\lambda)$, there are 
$M\in\CB(\lambda)$ and
$\gamma:F_iM\to L$
such that the composition
$$\alpha:\Phi(M)\xrightarrow{\eta_i\bullet}E_iF_i\Phi(M)=
E_i\Phi(F_iM)\xrightarrow{E_i\Phi(\gamma)}
E_i\Phi(L)$$
is an isomorphism.
\end{lemma}

\begin{proof}
It is enough to prove the lemma for $L=F_{i_r}\cdots F_{i_1}(k_\lambda)$
for any $i_1,\ldots,i_r\in I$.
We prove this by induction on $r$.
Assume the lemma holds for $r$. Consider $i_1,\ldots,i_r\in I$.
Let $M$, $\alpha$ and $\gamma$ be provided by the lemma.

Consider now $j=i_{r+1}\in I$. Let $M'=F_jM$. Let
$$\gamma'=(F_iF_jM\xrightarrow{\tau_{ij}\bullet}F_jF_iM
\xrightarrow{F_j\gamma} F_jL).$$
There is a commutative diagram
$$\xymatrix{
\Phi(F_jM)\ar[r]^-{\eta_i\bullet}\ar[d]_{F_j\eta_i\bullet}\ar@/_4pc/[dd]_\sim &
 E_iF_i\Phi(F_jM)\ar[r]^-{E_i\tau_{ij}\bullet}\ar[d]_{\bullet\eta_i\Phi(M)}
& E_i\Phi(F_jF_iM)\ar[r]^-{E_i\Phi(F_j\gamma)}\ar[d]_{\bullet\eta_i\Phi(M)}
\ar@/^5pc/[ddd]^\sim & E_i\Phi(F_jL)
 \ar  `d[ddddl]^-{\id} `[ddddlll] [dddlll]\\
F_jE_iF_i\Phi(M)\ar[d]_{\bullet\Phi(\gamma)}\ar[r]^-{\eta_i\bullet} & 
 E_iF_iF_jE_iF_i\Phi(M)\ar[r]^-{E_i\tau_{ij}\bullet}\ar[d]_{\bullet\Phi(\gamma)} &
E_iF_jF_iE_iF_i\Phi(M)\ar[ddl]^{\bullet\Phi(\gamma)} 
\ar[dd]_{E_iF_j\eps_i\bullet} &\\
F_jE_i\Phi(L) \ar[r]^-{\eta_i\bullet}\ar[d]_{\sigma_{ji}\bullet} &
 E_iF_iF_jE_i\Phi(L)\ar[d]_{E_i\tau_{ij}\bullet}  &&& \\
E_iF_j\Phi(L) & E_iF_jF_iE_i\Phi(L)\ar[l]_-{\bullet\eps_i\Phi(L)} &
E_iF_jF_i\Phi(M)\ar@/^1pc/[ll]^{\bullet\Phi(\gamma)} &\\
&&&
}$$

If $j\not=i$, then $\sigma_{ij}$ is an isomorphism, hence
$(M',\gamma')$ satisfies the requirements.
We assume now $i=j$. Let 
$n=\langle \lambda+\alpha_{i_1}+\cdots+\alpha_{i_r},\alpha_i^\vee\rangle$.

\smallskip
Assume $n\ge 0$.
Let $M''=M'\oplus L^{\oplus n}$ and
$\gamma''=\gamma'+\sum_{a=0}^{n-1}(x_i^a\bullet):F_iM''\to F_iL$.
Then, $(M'',\gamma'')$ satisfies the required properties.

\smallskip
Assume finally $n\le 0$. Consider $g=\sum_{l=0}^{1-n}\gamma\circ(x_i^lM):
F_iM\to L^{\oplus -n}$. The map $\Phi(g)$ is equal to the composition
$$\xymatrix{
F_i\Phi(M)\ar[r]^-{F_i\eta_i\bullet} \ar@/_1pc/[rrr]_{\sim} & 
F_iE_iF_i\Phi(M) \ar[rr]^-{F_iE_i\Phi(\gamma)} &&
F_iE_i\Phi(L) \ar[rr]^-{\sum_l x_i^l\bullet} &&
(F_iE_i\Phi(L))^{\oplus -n} \ar[r]^-{\eps_i\bullet} & \Phi(L)^{\oplus -n}
}$$
which is a split surjection. As a consequence, 
$g=\Psi\Phi(g)$ is a split surjection. Let $M''$ be
its kernel and $\gamma''=\gamma'_{|F_iM'}$. The composition
$$\Phi(M'')\hookrightarrow \Phi(F_iM)\xrightarrow{F_i\eta_i\bullet}
F_iE_iF_i\Phi(M)\xrightarrow{\bullet\Phi(\gamma)}F_iE_i\Phi(L)
\xrightarrow{\sigma_{ji}\bullet}E_iF_i\Phi(L)$$
is an isomorphism and we deduce that $(M'',\gamma'')$ satisfies the
requirements.
\end{proof}

Note that Lauda and Vazirani had shown earlier that $\CB(\lambda)$
gives rise to the crystal graph of $L(\lambda)$, in the symmetrizable case
\cite{LauVa}.

\subsubsection{Cyclotomic Hecke algebras for $\Gsl_2$}
\label{se:Heckesl2}
Let $n\in\BZ_{\ge 0}$. We have $H_n(Q)={^0H}_n$. One can deduce from
Theorem \ref{th:LequivB} that the $2$-representations $\CL(n)$ of
\S \ref{se:repsl2} and \S \ref{se:simplerep} are equivalent. One can also
show this directly without using $\CB(n)$,
by using the same method as in the proof of Theorem \ref{th:LequivB}. Let
us prove the more concrete fact that $\CB(n)$ coincides with the
category $\CL(n)$ of \S \ref{se:repsl2}.

\begin{lemma}
Given $i\le n$, then there is an isomorphism of rings
$$\phi:
H_i(n)\iso H_{i,n},\ T_j\mapsto T_j,\ X_{j'}\mapsto X_{j'}
\text{ and }z_l\mapsto (-1)^l e_l(x_1,\ldots,x_n)$$
for $1\le j\le i-1$, $1\le j'\le i$ and $1\le l\le n$.
If $i>n$, then $H_i(n)=0$.
\end{lemma}

\begin{proof}
Assume $i\le n$.
In order to prove that the map $\phi$ of the lemma is well defined,
it is enough to consider
the case $i=n$. We have
$$x_1^n-e_1(x_1,\ldots,x_n)+\cdots+(-1)^n e_n(x_1,\ldots,x_n)=0,$$
hence the map is well defined. It is clear that the map is surjective.

\smallskip
Let $A_i=\BZ[z_1,\ldots,z_n]\otimes_{\BZ} {^0H}_i$ and
$V_i=\BZ[z_1,\ldots,z_n]\otimes_\BZ P_i$, a faithful 
$A_i$-module.
Let $M_i$ be the $A_i$-submodule of $V_i$ generated by 
$X_1^n+X_1^{n-1}z_1+\cdots+X_1z_{n-1}+z_n$ and 
let $L_i=\sum_{a_j\le n-j} \BZ[z_1,\ldots,z_n]
x_1^{a_1}\cdots x_i^{a_i}$.
Let $V'_i=L_i+M_i$.
Let us show by induction on $i$ that $V_i=V'_i$.
This is true for $n=1$. Assume $V'_i=V_i$.
Note that $V'_{i+1}$ is stable under the action of $A_i$ and the
action of $T_i$.
It follows that $\partial_i(X_i^{n-i+1})\in V'_{i+1}$,
hence $X_{i+1}^{n-i}\in V'_{i+1}$. This shows that $V'_{i+1}$ is stable
under multiplication by $X_{i+1}$, hence under the action of $A_{i+1}$.
Since $V_{i+1}$ is generated by $1\in V'_{i+1}$ under the action of
$A_{i+1}$, we deduce that $V'_{i+1}=V_{i+1}$. 

We deduce that $H_i(n)$ has a faithful module of rank $\le \frac{n!}{(n-i)!}$
over $\BZ[z_1,\ldots,z_n]$. As a consequence, the image of
$P_i\otimes\BZ[z_1,\ldots,z_n]$ in $H_i(n)$ has rank $\le \frac{n!}{(n-i)!}$
over $\BZ[z_1,\ldots,z_n]$, hence $H_i(n)$ has rank $\le \frac{i!n!}{(n-i)!}$
over $\BZ[z_1,\ldots,z_n]$. That is the rank of $H_{i,n}$ over
$P_n^{\GS_n}$: it follows that $\phi$ is an isomorphism.
\end{proof}

\section{Geometry}

\subsection{Hall algebras}
We refer to \cite{Sch} for a general text on Hall algebras.
\subsubsection{Definition}
Let $\CA$ be an abelian category such that given $M,N\in\CA$, then
$\Hom_\CA(M,N)$ and $\Ext^1_\CA(M,N)$ are finite sets.
One can take for example
the category of finite dimensional representations of a quiver over a finite
field.

Given $M,N\in\CA$, let
$F_{M,N}^L$ be the number of submodules $N'$ of $L$ such that
$N'\simeq N$ and $L/N'\simeq M$.

Let $P_{M,N}^L$ denote the number of exact sequences
$0\to N\to L\to M\to 0$.
Then,
\begin{equation}
\label{relationPF}
F_{M,N}^L=\frac{P_{M,N}^L}{|\Aut(M)|\cdot |\Aut(N)|}.
\end{equation}

\smallskip
Let $H'_\CA$ be the free abelian group with basis the isomorphism classes of
objects of $\CA$
$$H'_\CA=\bigoplus_{L\in \CA/\sim} \BZ [L].$$

We define a product in $H'_\CA$ by
$$[M]\ast [N]=\sum_{L\in \CA/\sim} F_{M,N}^L [L].$$

The class $[0]$ is a unit for the product.
The algebra $H'_\CA$ is the {\em Hall algebra} of $\CA$.

One shows that the product is associative and more generally, that an iterated 
product counts filtrations.

Given $N_1,\ldots,N_n,L\in\CA$, let
$F_{N_1,\ldots,N_n}^L$ be the number of filtrations
$$L=L_0\supset\cdots\supset L_n=0$$
with $L_{i-1}/L_i\simeq N_i$.

\begin{prop}
We have
$[N_1]\ast\cdots\ast[N_n]={\displaystyle
\sum_{L\in\CA/\sim} F_{N_1,\ldots,N_n}^L [L]}$.
\end{prop}

\begin{rem}
When $\CA$ is semi-simple, then $H'_\CA$ is commutative. The ``next'' case
is the following.
Let $\CA$ be the category of finite abelian $p$-groups. The algebra
$H'_\CA$ has a basis parametrized by partitions and
$H'_\CA=\BZ[u_1,u_2,\ldots]$ is a polynomial ring in the countably many
variables $u_i=[(\BZ/p)^i]$ (Steinitz-Hall).
\end{rem}

\subsubsection{Hall algebra for an $A_2$ quiver}

Let us now describe the Hall algebra for $\CA$ the category of
finite dimensional representations of the quiver $\Gamma=1\to 2$ over
a finite field $k$ with $q$ elements. 

The indecomposable
representations of $\Gamma$ are $S(1)$, $S(2)$ and $M$
(cf Example \ref{ex:examplesquivers}).
Let $f_1=[S(1)]$, $f_2=[S(2)]$ and
$f_{12}=[M]$. We find
$f_1\ast f_2=f_{12}+[S(1)\oplus S(2)]$ and
$f_2\ast f_1=[S(1)\oplus S(2)]$. The algebra $H'_\CA$ is not commutative.
We have $[f_1,f_2]=f_{12}$.

We have $f_1\ast f_{12}=q[M\oplus S(1)]$ and
$f_{12}\ast f_1=[M\oplus S(1)]$.
So, $f_1\ast f_{12}=q f_{12}\ast f_1$.
If we view $q$ as an indeterminate and specialize it to $1$,
then the Lie subalgebra of $H'_\CA$ generated by $f_1$, $f_2$ and
$f_{12}$ is isomorphic to the Lie algebra of strictly upper triangular
$3\times 3$-matrices:
$$f_1\mapsto \left(\begin{matrix}0&1&0\\0&0&0\\0&0&0\end{matrix}\right),\ 
f_2\mapsto \left(\begin{matrix}0&0&0\\0&0&1\\0&0&0\end{matrix}\right),\ 
f_{12}\mapsto \left(\begin{matrix}0&0&1\\0&0&0\\0&0&0\end{matrix}\right).$$

\subsubsection{Quantum groups as Ringel-Hall algebras}
Let $\Gamma$ be a quiver with vertex set $I$ and assume $\Gamma$ has
no loops. Let $\CA$ be the category of finite dimensional representations
of $\Gamma$ over a finite field $k$ with $q$ elements.

The {\em Euler form} is defined by
$$\langle M,N\rangle=\dim \Hom(M,N)-\dim\Ext^1(M,N)$$
for $M,N\in\CA$.

We define the {\em Ringel-Hall} algebra $H_\CA$ as the $\BC$-vector
space $\BC\otimes_{\BZ}H'_\CA$ with 
the product
$$[M]\cdot [N]=q^{\langle M,N\rangle/2} [M]*[N].$$

\smallskip
The graph underlying $\Gamma$ encodes a symmetric Cartan matrix,
hence give rise to the nilpotent part $\Gn^-$
of a Kac-Moody algebra.

\smallskip
We can now state Ringel's Theorem.

\begin{thm}[Ringel]
\label{th:Ringel}
There is an injective morphism of $\BC$-algebras
$U_q(\Gn^-)\hookrightarrow H_\CA,\ f_i\mapsto [S(i)]$. If $\Gamma$ corresponds
to a Dynkin diagram, then this morphism is an isomorphism.
\end{thm}

In the next section we will explain, following Lusztig, how to
construct directly the non-quantum enveloping algebra $U(\Gn^-)$.

\subsection{Functions on moduli stacks of representations of quivers}

We refer to \cite{ChrGi} for a general introduction to geometric
representation theory.

\subsubsection{Moduli stack of representations of quivers}

Let $\Gamma$ be a quiver with vertex set $I$.
Let $\Rep=\Rep(\Gamma)$ be the moduli stack of representations of $\Gamma$
over $\BC$. This is a geometrical object whose points are isomorphism
classes of finite dimensional representations of $\Gamma$ over $\BC$.
The moduli stack also encodes the information of the group $\Aut(M)$,
given $M$ a representation of $\Gamma$. We have
$\Rep=\coprod_{d\in\BZ_{\ge 0}^I}\Rep_d$, where $\Rep_d$ corresponds
to representations $V$ with $\dim V_i=d_i$ for $i\in I$. We
refer to \S \ref{se:tensorLusztig} for a more precise description.

\smallskip
This stack can be described explicitly as a quotient.
Let $d\in \BZ_{\ge 0}^I$. The data of a representation of
$\Gamma$ with underlying vector spaces $\{\BC^{d_i}\}_i$ is the same
as the data of an element of $\CM_d=\bigoplus_{a:i\to j}
\Hom_{\BC}(\BC^{d_i},\BC^{d_j})$, where $a$ runs
over the arrows of $\Gamma$. This vector space has an action by
conjugation of the group $G_d=\prod_i \GL_{d_i}(\BC)$:
$$g\cdot (f_a)_a=(g_j f_a g_i^{-1})_a \text{ for }g=(g_i)_i.$$

Two representations are isomorphic is they correspond to elements of
$\CM_d$ in the same $G_d$-orbit. Given $f\in\CM_d$, we have
$\Stab_{G_d}(f)=\Aut(M)$, where $M$ is the representation of $\Gamma$
defined by $f$.

We have $\Rep_d=\CM_d/G_d$. 

\subsubsection{Convolution of functions}
\label{se:convolutionsets}
Let $X$ be a set and $\CF(X)$ the vector space of functions $X\to\BC$.

Consider a map $\phi:X\to Y$ between sets.

We define
$\phi^*:\CF(Y)\to\CF(X)$ by $\phi^*(f)(x)=f(\phi(x))$.

Assume $\phi^{-1}(y)$ is finite for all $y\in Y$. Define
$\phi_*:\CF(X)\to\CF(Y)$ by $\phi_*(f)(y)=\sum_{x\in \phi^{-1}(y)}f(x)$.

Now, given a diagram
$$\xymatrix{
& Z \ar[dl]_p \ar[d]_q \ar[dr]^r \\
X & X & X}$$
with the fibers of $r$ finite, we define a convolution of functions
$$\CF(X)\times\CF(X)\to\CF(X),\ (f,g)\mapsto f\circ g=
r_*(p^*(f)\cdot q^*(g)).$$

\subsubsection{Convolution of constructible functions}
We want now to extend the constructions of \S \ref{se:convolutionsets}
to the case of varieties, or rather stacks. The main problem is
to give a sense to $\phi_*$ when $\phi$ doesn't have finite fibers.

Let $X$ be a stack over $\BC$. We define $\CF_c(X)$, the space
of {\em constructible functions}, as the subspace
of $\CF(X)$ generated by the functions $1_V$, where $V$
runs over locally closed subspaces of $X$. Here, $1_V(x)=1$ if
$x\in V$ and $1_V(x)=0$ otherwise.

Given $X$ a stack, we denote by $\chi(X)=\sum_{i\ge 0}(-1)^i
\dim H^i(X)$ the {\em Euler characteristic} of $X$. There is
a unique extension of $\chi$ to disjoint unions of locally closed subsets of
stacks that satisfies $\chi(X)=\chi(V)+\chi(X-V)$.

\smallskip
Given $\phi:X\to Y$ a morphism of stacks, we define $\phi_*$ as follows.
Let $f=\sum_\alpha m_\alpha 1_{V_\alpha}$, where the $V_\alpha$ are
locally closed subsets of $X$ and $m_\alpha\in\BC$. We
put 
$$\phi_*(f)(y)=\sum_\alpha m_\alpha \chi(V_\alpha\cap \phi^{-1}(y)).$$

\subsubsection{Realization of $U(\Gn^-)$}
\label{se:lusztigfunctions}
Denote by $X$ the stack over $\mathrm{Rep}\times\mathrm{Rep}$ of pairs
$(V\subset V')$. We have three morphisms $p,q,r:X\to\Rep$
$$\xymatrix{
& X \ar[dl]_-p \ar[d]_-q \ar[dr]^-r \\
\Rep & \Rep & \Rep}$$
$$p(V\subset V')=V,\ q(V\subset V')=V'/V \text{ and } r(V\subset V')=V'.$$

Define a convolution
$$\CF_c(\Rep)\times\CF_c(\Rep)\to\CF_c(\Rep),\ (f,g)\mapsto f\circ g=
r_*(p^*(f)\cdot q^*(g)).$$

Let $a_i=1_{\Rep_i}$: we have $a_i(S(i))=1$ and $a_i(M)=0$ if
$M$ is a representation of $\Gamma$ not isomorphic to $S(i)$.

\begin{thm}[Lusztig]
There is an injective morphism of $\BC$-algebras
$U(\Gn^-)\to \CF_c(\Rep),\ f_i\mapsto a_i$. If $\Gamma$ corresponds to a
Dynkin diagram, this is an isomorphism.
\end{thm}

\subsection{Flag varieties}
We recall classical facts on affine Hecke algebra actions and flag varieties
in type $A$ (cf e.g. \cite{ChrGi}) and then flags of representations
of quivers \cite{Lu1}.

\subsubsection{Notations}
We fix a prime number $l$ and put $\Lambda=\bar{\BQ}_l$.
By scheme, we mean a separated scheme of finite type over $\BC$. 
Given $X$ a scheme or a stack, we denote by $D(X)$ the bounded derived
category of $l$-adic constructible sheaves on $X$ (cf \cite{LaOl1,LaOl2}).
All quotients will be taken in the category of stacks.

Given $X$ a smooth stack and $i:Z\to X$ a smooth closed substack, both of
pure dimension, we have a Gysin morphism
$i_* \Lambda_Z\to \Lambda_X[2(\dim X-\dim Z)]$. Let $D$ be the duality functor.
Via the canonical identifications
$D(\Lambda_Z)\iso \Lambda_Z[2\dim Z]$, $D(\Lambda_X)\iso \Lambda_X[2\dim X]$ and
$D\circ i_*\iso i_*\circ D$, the Gysin morphism is the dual of the canonical
map $\Lambda_X\to i_* \Lambda_Z$. Note finally that the automorphism of
$\Ext^*(\Lambda_X,\Lambda_X)$ induced by taking $\alpha$ to $D(\alpha)$ is the identity.

\smallskip
Let $\CC$ be a graded additive category and $M,N$ two objects
of $\CC$. We put $\Hom^\bullet_\CC(M,N)=\bigoplus_i \Hom(M,N[i])$.

Given a graded ring $A$, we define the graded dimension of a free
finitely generated graded $A$-module by $\mathrm{grdim}(A[i])=q^{-i/2}$ and
$\mathrm{grdim}(M)=\mathrm{grdim}(M_1)+\mathrm{grdim}(M_2)$ if
 $M\simeq M_1\oplus M_2$.

\subsubsection{Nil affine Hecke algebra and $\BP^1$-bundles}
\label{se:nilP1}
Let $X$ be a stack, $E$ a rank $2$ vector bundle on $X$ and
$\pi:Y=\BP(E)\to X$ the projectivized bundle.

Let $\alpha=c_1(\CO_\pi(-1))$ and
$\beta=c_1(\pi^*E/\CO_\pi(-1))$.
Let $x=\pi_*(\alpha)$ and $y=\pi_*(\beta)$, viewed in
$\Hom(\pi_*\Lambda_Y,\pi_*\Lambda_Y[2])$.

Let $T\in \Hom(\pi_*\Lambda_Y,\pi_*\Lambda_Y[-2])$ be the composition
$$T:\pi_*\Lambda_Y \xrightarrow{t} \Lambda_X[-2]
\xrightarrow{\can} \pi_* \Lambda_Y[-2]$$
where $t:\pi_*\Lambda_Y \xrightarrow{\can} \CH^2(\pi_*\Lambda_Y)[-2]
\xrightarrow[\sim]{\mathrm{tr}}\Lambda_X[-2]$ is the trace map.

\begin{prop}
\label{pr:HeckeP1}
We have $T^2=0$ $yT-Tx=1$ $xy=yx$ and $T(x+y)=(x+y)T$.
This defines a morphism of algebras 
$${^0H}_2\to \End^\bullet(\pi_*\Lambda_Y),\ X_1\mapsto x,\ X_2\mapsto y,\
T_1\mapsto T.$$
\end{prop}

\begin{proof}
We have $\alpha+\beta=c_1(\pi^*E)$ and $\alpha\beta=c_2(\pi^*E)$.
The composition 
$$\pi_*\pi^*\Lambda_X\xrightarrow[\sim]{\can}\pi_*\Lambda_Y
\xrightarrow{T}\pi_*\Lambda_Y[-2]\xrightarrow[\sim]{\can}
\pi_*\pi^*\Lambda_X[-2]$$
comes from natural transformations of functors, hence it commutes
with $\End^\bullet(\Lambda_X)$. It follows that $T$ commutes with
$c_1(E)=x+y$ and $c_2(E)=xy$.

\smallskip
Since $T^2$ factors through a map $\Lambda_X[-2]\to \Lambda_X[-4]$,
we have $T^2=0$.

\smallskip

The composition
$$\Lambda_X\xrightarrow{\can}\pi_* \Lambda_Y\xrightarrow{-x}
\pi_* \Lambda_Y[2]\xrightarrow{t}\Lambda_X$$
is the identity.
Indeed, after taking the fiber at a point $P\in X$, the composition is
$$\Lambda\xrightarrow{c_1(\CO(1))}H^2(\pi^{-1}(P),\Lambda)
\xrightarrow{t}\Lambda$$ 
and $c_1(\CO(1))$ is the class of a point.

So, we have an isomorphism
$$(\can,x\circ\can):
\Lambda_X\oplus \Lambda_X[-2]\iso \pi_*\Lambda_Y.$$

Note that the composition
$$\Lambda_X\xrightarrow{\can}\pi_*\Lambda_Y\xrightarrow{y}
\pi_*\Lambda_Y[2]\xrightarrow{t}\Lambda_X$$
is also the identity.

\smallskip
The composition
$$\Lambda_X\xrightarrow{\can}\pi_*\Lambda_Y\xrightarrow{T}
\pi_*\Lambda_Y[-2]$$
vanishes since it factors through a map $\Lambda_X\to \Lambda_X[-2]$.
It follows that
$$\Lambda_X\xrightarrow{\can}\pi_*\Lambda_Y\xrightarrow{yT-Tx}
\pi_*\Lambda_Y$$
is equal to the canonical map.

\smallskip

We have a commutative diagram
$$\xymatrix{
\Lambda_X\ar[r]^-{\can}\ar[dr]_-{c_2(E)} & \pi_*\Lambda_Y\ar[r]^-{xy} &
\pi_*\Lambda_Y[4]\ar[r]^-{T} & \pi_*\Lambda_Y[2] \\
& \Lambda_X[4] \ar[ur]_-{\can}}.$$
So, the composition
$\Lambda_X\to \pi_*\Lambda_Y[2]$
vanishes since it factors through a map $\Lambda_X[4]\to \Lambda_X[2]$.

It follows that the composition
$$\Lambda_X\xrightarrow{\can}\pi_*\Lambda_Y\xrightarrow{y}
\pi_*\Lambda_Y[2] \xrightarrow{yT-Tx} \pi_*\Lambda_Y[2]$$
is equal to the composition
$$\Lambda_X\xrightarrow{\can}\pi_*\Lambda_Y\xrightarrow{y}
\pi_*\Lambda_Y[2].$$
So, $yT-Tx=1$.
\end{proof}

\subsubsection{Nil affine Hecke algebras and flag varieties}
\label{se:flag}
We recall now the construction of an action of ${^0H}_n$ on
$H^*(\Gr_n/\GL_n)$ (cf \cite{Ku}), where $\Gr_n$ is the variety of
complete flags in $\BC^n$.

Let $\psi:\Gr_n/\GL_n\to\mathrm{pt}/\GL_n$ be the canonical map.

Consider the first Chern class of the line bundle defined by
$V_d/V_{d-1}$ over a complete flag
$(0=V_0\subset V_1\subset\cdots\subset V_n=\BC^n)$ and
let $X_d$ be the corresponding element $X_d:\psi_*\Lambda\to\psi_*\Lambda[2]$.

Let $\Gr_n(d)$ the the variety of flags
$(0=V_0\subset V_1\subset\cdots\subset V_{n-1}=\BC^n)$ such that
$\dim V_r/V_{r-1}=1$ for $r\not=d$ and
$\dim V_d/V_{d-1}=2$. The canonical map
\begin{align*}
\Gr_n&\to \Gr_n(d)\\
(0=V_0\subset V_1\subset\cdots\subset V_n=\BC^n)&\mapsto
(0=V_0\subset V_1\subset\cdots\subset V_{d-1}\subset V_{d+1}\subset\cdots
\subset V_n=\BC^n)
\end{align*}
is the projectivization
of the $2$-dimensional vector bundle $V_d/V_{d-1}$ over $\Gr_n(d)$.
It induces a map $p_d:\Gr_n/\GL_n\to \Gr_n(d)/\GL_n$.

Let $T\in \Hom(p_{d*}\Lambda,p_{d*}\Lambda[-2])$ be the composition
$$T:p_{d*}\Lambda \xrightarrow{t} \Lambda[-2]
\xrightarrow{\can} p_{d*} \Lambda[-2]$$
where $t:p_{d*}\Lambda \xrightarrow{\can} \CH^2(p_{d*}\Lambda)[-2]
\xrightarrow[\sim]{\mathrm{tr}}\Lambda[-2]$ is the trace map.
We denote by $T_d:\psi_*\Lambda\to\psi_*\Lambda[-2]$ the induced map.

\medskip
Let $X$ be a stack. The data of a rank $n$ vector bundle $\CL$ on $X$ is
equivalent to the data of a morphism of stacks $X\to \mathrm{pt}/\GL_n$.
Let $Y$ be the stack of full flags in a rank $n$ vector bundle
$\CL$ on $X$ and $\phi:Y\to X$ be the
associated map: this is the pullback of $\psi$ via $l$.

The following Theorem follows from Proposition \ref{pr:HeckeP1}, together with
a verification of the braid relations between $T_d$'s.

\begin{thm}
The construction above provides by base change a
morphism ${^0H}_n\to \End^\bullet(\phi_*\Lambda)$.
\end{thm}

Let $j:X'\to X$ be a closed immersion. Assume $X$ and $X'$ are smooth
of pure dimension. Then, the canonical morphism
$\phi_*\Lambda\to\phi'_*\Lambda$ and the morphism induced by the Gysin map
$\phi'_*\Lambda\to\phi_*\Lambda[2(\dim X-\dim X')]$ commute
with the action of ${^0H}_n$.

\subsubsection{Sheaves on moduli stacks of quivers}
\label{se:tensorLusztig}
We follow Lusztig \cite[\S 9]{Lu1}. Instead of working with equivariant
derived categories of varieties, we work with derived categories of the
corresponding quotient stacks. Given $X$ a variety acted on by $G$,
our perverse sheaves on $X/G$ correspond to shifts by $\dim G$
of the $G$-equivariant perverse sheaves on $X$ considered by Lusztig. The
duality is similarly shifted.

\smallskip
Given $X$ a scheme, we have an abelian category 
$\CO_X[\Gamma]\mMOD$ of representations
of $\Gamma$ over $X$, \ie, sheaves of $(\BZ \Gamma\otimes_{\BZ}\CO_X)$-modules.
Its objects
can be viewed as pairs $V=(\CV,\rho)$ where $\CV$ is an $\CO_X$-module
and $\rho:\BZ \Gamma\to \End(\CV)$ is a morphism of rings.

Given $J$ a subset of $I$, we put
$\CV_J=\sum_{i\in J}\rho(i)\CV$.

\smallskip
We denote by
$\Rep=\Rep(\Gamma)$ the algebraic stack of representations of $\Gamma$.
It is defined by assigning to a scheme $X$ the subcategory
of $\CO_X[\Gamma]\mMOD$ defined as follows:
\begin{itemize}
\item objects are pairs $(\CV,\rho)$ such that
$\CV$ is a vector bundle (of finite rank) over $X$
\item maps are isomorphisms.
\end{itemize}

Given a morphism of schemes $f:X\to Y$, we have a functor
$f^*:\Rep(\Gamma)(Y)\to\Rep(\Gamma)(X)$ given by base change.

\smallskip

We define the rank vector of $V=(\CV,\rho)$ as
$\rk V=\sum_{i\in I}(\rk \CV_i)i\in\BN[I]$.
We have a decomposition into connected components
$$\Rep=\coprod_{\alpha\in\BN[I]}\Rep_\alpha$$
where $\Rep_{\alpha}$ is the substack of representations with rank
vector $\alpha$. We denote by $j_\alpha$ the embedding of
the component $\Rep_\alpha$. We have
$\dim\Rep_{\alpha}=-\langle \alpha,\alpha\rangle$. Note that
$\Rep_0$ is a point.
We denote by $\CW_\alpha$ the tautological vector
bundle $\{V\}$ on $\Rep_\alpha$.

\smallskip

Given $i\in I$, we put $L_i=j_{i*} \Lambda[-1]$, a perverse sheaf on $\Rep$.

\medskip
Given $M\in D(\Rep_\mu)$ and $N\in D(\Rep_{\mu'})$, we put
$M\circ N=r_!(p^*M\otimes q^*N)[-\langle\mu',\mu\rangle]$.
This endows
$D(\Rep)$ with a structure of monoidal category (it is the reverse of the
tensor structure defined by Lusztig).

\subsubsection{Flags of representations}
\label{se:flags}
Given $\nu=(\nu^1,\ldots,\nu^n)\in\BN[I]^n$, we consider
the stack 
$\Rep_\nu$
of flags
$(0=V^0\subset V^1\subset \cdots\subset V^n)$ 
of representations of $\Gamma$
such that $\rk V^r/V^{r-1}=\nu^r$. We have a proper morphism
$$\pi_{\nu}:\Rep_\nu\to \Rep_{\sum_r \nu^r},\ 
(0=V^0\subset V^1\subset \cdots\subset V^n)
\mapsto V^n$$
and we put $\rho_{\nu}=j_{\sum_r \nu^r}\circ \pi_\nu:\Rep_\nu\to \Rep$.

\medskip
Given $\eps_1,\ldots,\eps_n\in\{\mathrm{empty},\mathrm{ss}\}$, we define
$\Rep_{(\nu^1_{\eps_1},\ldots,\nu^n_{\eps_n})}$ as the closed
substack of $\Rep_{(\nu^1,\ldots,\nu^n)}$ of flags
$(0=V^0\subset V^1\subset\cdots\subset V^n)$ such that for any $r$ such
that $\eps_r=\mathrm{ss}$, then
$V^r/V^{r-1}\simeq \bigoplus_i (V^r/V^{r-1})_i$ as a representation of
$\Gamma$.

\subsubsection{Flags and quotients}
Let $\bar{\Gamma}$ be the discrete quiver with vertex set $I$ and
let $\overline{\Rep}=\Rep(\bar{\Gamma})$.
The restriction map defines a morphism
$\Rep_\nu\to \overline{\Rep}_\nu$. This is a vector bundle of rank
$$\sum_{i\not=j} d_{ij}\sum_r \nu^r_i(\nu^1_j+\cdots+\nu^r_j).$$

\medskip
Let $\alpha=\sum_i \alpha_i i\in\BN[I]$. We denote by
$\widetilde{\Rep}_\alpha$ the variety
$\prod_{h:i\to j} \Hom_{\BC}(\BC^{\alpha_i},\BC^{\alpha_j})$, where
$h$ runs over the set of arrows of $\Gamma$.
There is an action of $G_\alpha=\prod_i \GL_{\alpha_i}$ on
$\widetilde{\Rep}_\alpha$ given by
$g\cdot f=(g_jf_hg_i^{-1})_{h:i\to j}$, where $g=(g_i)_{i\in I}$ and
$f=(f_h)_h$. A point of $\widetilde{\Rep}_\alpha$ defines a representation
of $\Gamma$ of dimension vector $\alpha$: this provides an isomorphism
$\widetilde{\Rep}_\alpha/G_\alpha\iso \Rep_\alpha$.
In particular, we obtain 
$\mathrm{B} G_\alpha\iso \overline{\Rep}_\alpha$.

\smallskip
Given $d_1,\ldots,d_r\ge 0$, we denote by $\Gr_{d_1,\ldots,d_r}$ the 
variety of flags $(0=V_0\subset V_1\subset\cdots\subset V_r=\BC^{\sum d_l})$
such that $\dim V_l/V_{l-1}=d_l$.
Let $\nu=(\nu^1,\ldots,\nu^n)\in\BN[I]^n$.
Let $\alpha=\sum_r \nu^r$ and
$n_i=\sum_{r=1}^n \nu^r_i$. We denote by
$\widetilde{\Rep}_\nu$ the subvariety of
$\prod_i \Gr_{\nu^1_i,\ldots,\nu^n_i} \times \widetilde{\Rep}_\alpha$ given by
families $((0=V_{i,0}\subset\cdots\subset V_{i,n}=\BC^{n_i})_i,(f_h)_h)$
such that
$f_h(V_{i,r})\subset V_{j,r}$ for all $h:i\to j$ and all $r$. The diagonal
action of $G=G_\alpha$ restricts to an action on $\widetilde{\Rep}_\nu$.
Sending a point to the associated filtered representation of $\Gamma$
defines an isomorphism
$\widetilde{\Rep}_\nu/G\iso \Rep_\nu$. Let $P_i$ be the parabolic
subgroup of $\GL_{n_i}$ stabilizing the standard flag
$F_i=(V_{i,0}=0\subset V_{i,1}=\BC^{\nu_i^1}\oplus 0\subset\cdots\subset
V_{i,n}=\BC^{\nu_i^1}\oplus\cdots \oplus\BC^{\nu_i^n})$ and let
 $P=\prod_i P_i$. We have a canonical isomorphism
$G/P\iso \prod_i \Gr_{\nu^1_i,\ldots,\nu^n_i}$ inducing an isomorphism
$G\setminus\! G/P\iso \overline{\Rep}_\nu$.

\smallskip
Let $\nu'=(\nu^{\prime 1},\ldots,\nu^{\prime n'})\in\BN[I]^{n'}$. 
We assume $\alpha=\sum_r \nu^{\prime r}$. This
defines as above a parabolic subgroup $P'$ of $G$. We denote
by $W$, $W_P$ and $W_{P'}$ the Weyl groups of $G$, $P$ and $P'$.
We have an isomorphism
$$\left(\widetilde{\Rep}_\nu\times_{\widetilde{\Rep}_\alpha}
\widetilde{\Rep}_{\nu'}\right)/G\iso 
\Rep_\nu\times_{\Rep} \Rep_{\nu'}.$$
The isomorphisms above induce an isomorphism
$$P'\setminus\! G/P\iso \overline{\Rep}_{\nu}\times_{\overline{\Rep}}
 \overline{\Rep}_{\nu'}.$$
Its closed points are in bijection with $W_{P'}\!\setminus\! W/W_P$ and
each such point $w$ defines a locally closed closed substack 
$X_w$.
This corresponds to the  decomposition
$$\prod_i \left(\Gr_{\nu^1_i,\ldots,\nu^n_i}\times 
\Gr_{\nu^{\prime 1}_i,\ldots,\nu^{\prime n'}_i}\right)=
\coprod_w \CO_w$$
into orbits under the action of $G$, \ie, $\CO_w/G \iso X_w$.

The restriction map $V\to \{V_i\}_{i\in I}$ induces a map
$\kappa:\Rep_{\nu}\times_{\Rep} \Rep_{\nu'}\to 
\overline{\Rep}_{\nu}\times_{\overline{\Rep}} \overline{\Rep}_{\nu'}$. 
The restriction of $\kappa$ over each $X_w$ is a vector bundle.
Note that $H^*_c(\Rep_{\nu}\times_{\Rep} \Rep_{\nu'})$ is a free
graded $H^*(BG)$-module of graded rank equal to the graded rank of
$H^*_c(\widetilde{\Rep}_\nu\times_{\widetilde{\Rep}_\alpha}
\widetilde{\Rep}_{\nu'})$ as a $\Lambda$-module.
The pullback of $\kappa$ is the projection map
$$\tilde{\kappa}:\widetilde{\Rep}_\nu\times_{\widetilde{\Rep}_\alpha}
\widetilde{\Rep}_{\nu'}\to 
\prod_i \left(\Gr_{\nu^1_i,\ldots,\nu^n_i}\times 
\Gr_{\nu^{\prime 1}_i,\ldots,\nu^{\prime n'}_i}\right).$$

\smallskip
Assume now $\nu^r\in I$ for all $r$. We have $W_P=W_{P'}=1$.
Define $\gamma_i,\gamma'_i:\{1,\ldots,n_i\}\to\{1,\ldots,n\}$ to be the
increasing maps such that $\nu^{\gamma_i(r)}=\nu^{\prime\gamma'_i(r)}=i$ 
for all $r$.

We identify $W$ with $\prod_i \GS_{n_i}$.
Let $w=(w_i)_i\in W$.
The fiber of $\tilde{\kappa}$ over
$((F_i,w_i(F'_i)))_i\in \CO(w)$ has dimension
$$\sum_{s\not= t} d_{st}\cdot
\#\{a,b|\gamma_t(b)<\gamma_s(a)\text{ and }\gamma'_t(w_t^{-1}(b))<
\gamma'_s(w_s^{-1}(a))\}.$$
We deduce that
$$\mathrm{grdim}H^*_c(\widetilde{\Rep}_\nu\times_{\widetilde{\Rep}_\alpha}
\widetilde{\Rep}_{\nu'})=
P(W,q)\sum_{w\in W}
q^{(l(w)+\sum_{s\not= t} d_{st}\cdot
\#\{a,b|\gamma_t(b)<\gamma_s(a)\text{ and }\gamma'_t(w_t^{-1}(b))<
\gamma'_s(w_s^{-1}(a))\})}$$
where $P(W,q)=\prod_i \prod_{r=1}^{n_i} \frac{q^r-1}{q-1}$ is
the Poincar\'e polynomial of $W$.

\subsection{Quiver Hecke algebras and geometry}
\label{se:quivergeo}
\subsubsection{Monoidal category of semi-simple perverse sheaves}
We denote by $\CP$ the
smallest full additive monoidal subcategory of $D(\Rep)$ closed under
translations and containing the objects $L_i$ for $i\in I$.

\smallskip
The following theorem gives a presentation of $\CP$ by generators and
relations. It has been proven independently by Varagnolo and Vasserot
\cite{VarVas}.

\begin{thm}
\label{th:catB}
There is an equivalence of graded monoidal categories
$R:(\bar{\BQ}_l\otimes_\BZ\CB(\Gamma))^i\,\mgr\iso \CP$.
\end{thm}

The category $\CB(\Gamma)$
 is defined by generators and relations and in 
\S \ref{se:generators} we define the images of the generating objects
and arrows. The verification of the relations and the proof that the induced
functor is an equivalence start in \S \ref{se:poly}.

\smallskip
Theorem \ref{th:catB} shows that quiver Hecke algebras $H_n(\Gamma)$
are $\Ext$-algebras of certain sums of shifted simple perverse sheaves
on quiver varieties, as all objects of $\CP$ are of that form.

\subsubsection{Canonical basis}
There is an isomorphism of $\BZ[q^{\pm 1/2}]$-algebras
\cite[\S 14]{Lu1}
\begin{equation}
\label{eq:georeal}
U_{\BZ[q^{\pm 1/2}]}(\Gn^-)\iso K_0(\CP),\ f_s\mapsto [L_s].
\end{equation}

Let $B$ be the set of isomorphism classes of
simple perverse sheaves on $\Rep$ that are contained in $\CP$. Every object of $\CP$ is isomorphic to a direct sum
of shifts of objects of $B$. The canonical
basis $C$ of $U_{\BZ[q^{\pm 1/2}]}(\Gn^-)$ corresponds, via the isomorphism
(\ref{eq:georeal}), to $\{[L]\}_{L\in B}$.

\smallskip
Recall that there is a duality $\Delta$ \cite[\S 4.2.1]{Rou3}
on $\CB(\Gamma)$, \ie, a graded equivalence of monoidal
categories $\CB(\Gamma)^\opp
\iso \CB(\Gamma)$ with $\Delta^2=\Id$ given by
$$F_s[n]\mapsto F_s[-n],\ x_s\mapsto x_s \text{ and }\tau_{st}\mapsto \tau_{ts}.$$

Let $C'$ be the set of classes in $K_0$ of indecomposable objects $M$
of $(\bar{\BQ}_l\otimes_\BZ\CB(\Gamma))^i\mgr$ such that $\Delta(M)\simeq M$.

\begin{cor}
\label{co:canonicalpositive}
We have an isomorphism
$U_{\BZ[q^{\pm 1/2}]}(\Gn^-)\iso
K_0((\bar{\BQ}_l\otimes_\BZ\CB(\Gamma))^i\,\mgr)$. It induces a bijection
$C\iso C'$.
\end{cor}

\subsubsection{Hecke generators}
\label{se:generators}
We set $R(F_s)=L_s$. Let us now define the value of $R$ on the generating
arrows of $\CB(\Gamma)$.

\smallskip
$\bullet\ $Let $s\in I$. We denote by $x_s\in\Hom(L_s,L_s[2])$ the image of
$c_1(\CW_s)\in H^2(\Rep_s,\Lambda)$.

\medskip
$\bullet\ $The forgetful morphism 
$\pi_{(s,s)}:\Rep_{(s,s)}\to \Rep_{2s}$
is the $\BP^1$-fibration associated to the rank $2$ bundle
$\CW_{2s}$. We denote by
$\tau_{ss}\in \Hom(L_s\circ L_s,L_s\circ L_s[-2])$ the image of
the composition $\pi_{(s,s)*}\Lambda\xrightarrow{\mathrm{trace}}\Lambda[-2]
\xrightarrow{\can} \pi_{(s,s)*}\Lambda[-2]$.

\medskip
$\bullet\ $
Let $s\not=t\in I$.
Consider the morphism 
$$f_{st}:\Rep_{(s,t)}\to \Rep_s\times\Rep_t,\ (V\subset V')
\mapsto (V,V'/V).$$
Let $\CM_{st}=f_{st*}\CO$: this is the vector bundle $\Ext^1(V',V)$ over
$\Rep_s\times\Rep_t=\{(V,V')\}$.
We have $\CM_{st}\simeq (\CW_s\boxtimes \CW_t^{-1})^{\oplus d_{ts}}$.

The vector bundle $f_{st}^*\CM_{st}$ has a section given
by assigning to $(V\subset V')$ the class of the extension
$0\to V\to V'\to V'/V\to 0$. The zero substack of
that section is $Z_{st}=\Rep_{(s+t)_{\mathrm{ss}}}$, a closed
substack of codimension $d_{ts}$ in $\Rep_{(s,t)}$.

We denote by
$\tau_{st}\in\Hom(L_s\circ L_t,L_t\circ L_s[m_{st}])$
the image of the composition
$$\rho_{(t,s)*}(\Lambda_{Z_{st}}\xrightarrow{\mathrm{Gysin}}
\Lambda_{\Rep_{(t,s)}}[2d_{st}])\circ
\rho_{(s,t)*}(\Lambda_{\Rep_{(s,t)}}\xrightarrow{\can}\Lambda_{Z_{st}}).$$

\subsubsection{Polynomial actions}
\label{se:poly}
Let us first study $\Hom$-spaces in the category $\CP$ under the
action of polynomial rings. 

Let $\nu\in I^n$ and $\nu'\in I^{n'}$. Given $i\in I$,
let $n_i=\#\{r|\nu_r=i\}$ and
$n'_i=\#\{r|\nu'_r=i\}$.
By \cite[\S 8.6]{ChrGi}, there is an isomorphism of
$(\Lambda[x_{\nu_1},\ldots,x_{\nu_n}],\Lambda[x_{\nu'_1},\ldots,x_{\nu'_{n'}}])$-bimodules
$$\Ext^*(L_{\nu_1}\circ\cdots \circ L_{\nu_n},
L_{\nu'_1}\circ\cdots\circ  L_{\nu'_{n'}})\iso
H^{\sigma-*}_c(\Rep_{\nu}\times_{\Rep} \Rep_{\nu'})$$
where $\sigma=\dim\Rep_\nu +\dim\Rep_{\nu'}$.
If $\Ext^*(L_{\nu_1}\circ\cdots\circ  L_{\nu_n},
L_{\nu'_1}\circ\cdots\circ  L_{\nu'_{n'}})\not=0$, then the stack
$\Rep_{\nu}\times_{\Rep} \Rep_{\nu'}$ is non-empty, so
$n_i=n'_i$ for all $i$. Assume this holds. 
 It follows from
\S \ref{se:flags} that $\Ext^*(L_{\nu_1}\circ\cdots \circ L_{\nu_n},
L_{\nu'_1}\circ\cdots\circ  L_{\nu'_{n}})$ is a free graded 
$H^*(BG)$-module of graded rank
$$N=v^{\sigma}P(W,q^{-1})\sum_{w\in W}
q^{-(l(w)+\sum_{s\not= t} d_{st}\cdot
\#\{a,b|\gamma_t(b)<\gamma_s(a)\text{ and }\gamma'_t(w_t^{-1}(b))<
\gamma'_s(w_s^{-1}(a))\})}.$$
We have 
$$\sigma=2\sum_s n_s(n_s-1)+\sum_{s\not= t} d_{st}
\left(\#\{a,b|\gamma_t(b)<\gamma_s(a)\}+
\#\{a,b|\gamma_t'(b)<\gamma_s'(a)\}\right).$$
On the other hand,
$$l(w)=\sum_s \#\{a,b|\gamma_s(b)<\gamma_s(a)\text{ and }
\gamma'_s(w_s^{-1}(b))> \gamma'_s(w_s^{-1}(a))\}.$$
It follows that 
$$N=P(W,q)
\sum_{w\in W} q^{\frac{1}{2}\sum_{s,t\in I}
m_{st}\cdot \#\{a,b| \gamma_s(a)<\gamma_t(b)\text{ and }
\gamma_s'(w_s(a))> \gamma_t'(w_t(b))\}}$$
and we deduce from Lemma \ref{le:grdim} that the graded dimensions of
the free
$\left(\bigotimes_i \Lambda[X_{i,1},\ldots,X_{i,n_i}]^{\GS_{n_i}}\right)$-modules
$\Hom_{\bar{\BQ}_l\otimes_\BZ\CB(\Gamma)}(F_{\nu_1}\cdots F_{\nu_n}, F_{\nu'_1}\cdots F_{\nu'_n})$
and $\Ext^*(L_{\nu_1}\circ\cdots \circ L_{\nu_n},
L_{\nu'_1}\circ\cdots\circ  L_{\nu'_{n}})$
coincide.

\subsubsection{Relations $\tau^2$}
\label{se:tau2}
Let $s\not=t\in I$. The self-intersection formula shows that
$$\Lambda_{Z_{st}}\xrightarrow{\mathrm{Gysin}}\Lambda_{\Rep_{(t,s)}}[2d_{st}]
\xrightarrow{\can}\Lambda_{Z_{st}}[2d_{st}]$$
is equal to
$$c_{d_{st}}(f_{ts}^*\CM_{ts})=
(c_1((\CW_t)_{|Z_{st}})-c_1((\CW_s)_{|Z_{st}}))^{d_{st}}.$$
On the other hand, the composition
$$\Lambda_{\Rep_{(s,t)}}\xrightarrow{\can}\Lambda_{Z_{st}}\xrightarrow{\mathrm{Gysin}}
\Lambda_{\Rep_{(s,t)}}[2d_{ts}]$$
is equal to 
$$[Z_{st}]=c_{d_{ts}}(f_{st}^*\CM_{st})=
(c_1((\CW_s)_{|Z_{st}})-c_1((\CW_t)_{|Z_{st}}))^{d_{ts}}.$$

We have shown that
$$\tau_{ts}\circ\tau_{st}=(-1)^{d_{st}}(x_sL_t-L_s x_t)^{d_{st}+
d_{ts}}.$$

It follows from \S \ref{se:flag} that $\tau_{ss}^2=0$.

\subsubsection{Relations $\tau^3$}
\label{se:tau3}
Consider now $s,t,u\in I$.

$\bullet\ $Assume first $s$, $t$ and $u$ are distinct.
The intersection of the closed substacks
$\Rep_{((s+t)_{\mathrm{ss}},u)}$ and
$\Rep_{(t,(s+u)_{\mathrm{ss}})}$ of
$\Rep_{(t,s,u)}$ is transverse, since the intersection of
$0\times \BC^{d_{us}}$ and $\BC^{d_{st}}\times 0$ in
$\BC^{d_{st}}\times \BC^{d_{us}}$ is transverse.

It follows that the composition
$(L_t\tau_{su})\circ (\tau_{st}L_u)$ is equal to the image of the composition
$$\rho_{(t,u,s)*}(\BC_Z\xrightarrow{\mathrm{Gysin}} \BC_{\Rep_{(t,u,s)}}[2(
d_{su}+d_{st})])\circ
\rho_{(s,t,u)*}(\BC_{\Rep_{(s,t,u)}}\xrightarrow{\can}\BC_Z)$$
where $Z$ is the substack of $\Rep_{(t,u,s)}$ (resp. of
$\Rep_{(s,t,u)}$) of triples
$(L\subset L'\subset L'')$ such that $(L'')_s$ is a direct summand of $L''$.

Similarly, the intersection of $Z$ and $\Rep_{((t+u)_{\mathrm{ss}},s)}$
in $\Rep_{(t,u,s)}$ is transverse and we deduce that
$(\tau_{tu}L_s)\circ(L_t\tau_{su})\circ (\tau_{st}L_u)$ is
equal to the image of
$$\rho_{(u,t,s)*}(\BC_{\Rep_{(s+t+u)_{\mathrm{ss}}}}
\xrightarrow{\mathrm{Gysin}} \BC_{\Rep_{(u,t,s)}}[2(d_{st}+d_{su}+d_{tu})])\circ
\rho_{(s,t,u)*}(\BC_{\Rep_{(s,t,u)}}\xrightarrow{\can}
\BC_{\Rep_{(s+t+u)_{\mathrm{ss}}}}).$$
A similar calculation provides the same description of
$(L_u\tau_{st})\circ(\tau_{su}L_t)\circ (L_s\tau_{tu})$, so we have
$$(\tau_{tu}L_s)\circ(L_t\tau_{su})\circ (\tau_{st}L_u)=
(L_u\tau_{st})\circ(\tau_{su}L_t)\circ (L_s\tau_{tu}).$$

\medskip
$\bullet\ $
We consider now the case where $s=t\not=u$.
As above, we obtain that
the composition $(\tau_{su}L_s)\circ (L_s\tau_{su})$ is equal to
the image of the composition
$$\rho_{(u,s,s)*}(\BC_{\Rep_{(s,s,u)_{\mathrm{ss}}}}
\xrightarrow{\mathrm{Gysin}} \BC_{\Rep_{(u,s,s)}}[4d_{su}])\circ
\rho_{(s,s,u)*}(\BC_{\Rep_{(s,s,u)}}\xrightarrow{\can}
\BC_{\Rep_{(s,s,u)_{\mathrm{ss}}}}).$$
The commutation of the action of the nil affine Hecke algebra in
\S \ref{se:flag} shows that
$$(\tau_{su}L_s)\circ(L_s\tau_{su})\circ (\tau_{ss}L_u)=
(L_u\tau_{ss})\circ(\tau_{su}L_s)\circ (L_s\tau_{su}).$$

\medskip
$\bullet\ $The case $s=t=u$ follows from the results of \S \ref{se:flag}.

\subsubsection{Conclusion}
The relations (3) and (4) are clear when $s\not=t$ and follow from
\S \ref{se:flag} when $s=t$.
The results of \S \ref{se:tau2} and \ref{se:tau3} complete the verification
of the defining relations for the category $\CB(\Gamma)$.
Thanks to
\S \ref{se:poly}, we obtain a monoidal $\bar{\BQ}_l$-linear graded functor
$R:(\bar{\BQ}_l\otimes\CB(\Gamma))^i\,\mgr\to \CP$. That functor is essentially
surjective.
It follows from Proposition \ref{pr:ideals} and from \S \ref{se:poly}
that $R$ is faithful.

Let $\nu\in I^n$ and $\nu'\in I^{n'}$.
By Nakayama's Lemma, it follows from \S \ref{se:poly}
that $R$ induces an
isomorphism
$$\Hom_{\CB(\Gamma)}(F_{\nu_1}\cdots F_{\nu_n}, F_{\nu'_1}\cdots F_{\nu'_n})\iso
\Ext^*(L_{\nu_1}\circ\cdots \circ L_{\nu_n},
L_{\nu'_1}\circ\cdots\circ  L_{\nu'_{n'}}).$$
This completes the proof that $R$ is an equivalence.

\subsection{$2$-Representations}
\label{se:2repgeo}
\subsubsection{Framed quivers and construction of representations}
Nakajima introduced new quiver varieties in order to construct 
irreducible representations $L(\lambda)$ of Kac-Moody algebras. We
present a modification due to Hao Zheng \cite{Zh}.

Let $\hat{\Gamma}$ be the quiver obtained from $\Gamma$ by adding
vertices $\hat{i}$ for $i\in I$ and arrows $i\to \hat{i}$.
We have $\Rep(\hat{\Gamma})=\coprod_{\mu,\nu\in\BZ_{\ge 0}^I}
\Rep_{\hat{\nu}+\mu}(\hat{\Gamma})$.

\smallskip
Assume $i$ is a source of $\Gamma$. Let $U_i$ be the substack
of $\Rep$ of representations $V$ such that the canonical map
$V_i\to\bigoplus_{a:i\to j}V_j$ is injective, where $a$ runs over
arrows of $\hat{\Gamma}$ starting at $i$. Let $\CN_i$ be the
thick subcategory of $D(\Rep(\hat{\Gamma}))$ of complexes of sheaves
with $0$ restriction to $\CN_i$.

If $i$ is not a source, consider a quiver $\Gamma'_i$ corresponding
to a different orientation of $\Gamma$ and such that $i$ is a source of
$\Gamma'_i$. Define $\CN'_i\subset D(\Rep(\hat{\Gamma}_i')$ as above. Now,
there is an equivalence $D(\Rep(\hat{\Gamma}'_i)\iso D(\Rep(\hat{\Gamma}))$
given by Fourier transform and we define $\CN_i$ to be the image
of $\CN'_i$.

Finally, let $\CN$ be the thick subcategory of $D(\Rep(\hat{\Gamma}))$
generated by $\CN_i$ for $i\in I$. This is independent of the choice
of the quivers $\Gamma'_i$. Let
$D=D(\Rep(\hat{\Gamma}))/\CN$.

\medskip

Consider now a root datum $(X,Y,\langle-,-\rangle,\{\alpha_i\}_{i\in I},
\{\alpha_i^\vee\}_{i\in I})$ with Cartan matrix that afforded by
$\Gamma$.

Let $\lambda\in X^+$. Let $\nu_i=\langle\lambda,\alpha_i^\vee\rangle$
and $\nu=(\nu_i)_i\in\BZ_{\ge 0}^I$. We put
$\Rep(\lambda)=\coprod_{\mu\in\BZ_{\ge 0}^I}
\Rep_{\hat{\nu}+\mu}(\hat{\Gamma})$ and we denote by
$D(\lambda)$ the image of $D(\Rep(\lambda))$ in $D$.

The convolution functor $L_i\circ -$ stabilizes $\CN$ and
induces an endofunctor $E_i$ of $D(\lambda)$.

It has a right adjoint
$F_i$. Let $\CP(\lambda)$ be the smallest full subcategory of
$D(\lambda)$ containing $\BC_{\Rep_{\hat{\nu}}}$, stable under
$E_i$ for $i\in I$, and stable under direct summands and direct sums.

\begin{thm}[Zheng]
\label{th:Zheng}
The functors $E_i$ and $F_i$ satisfy Serre relations, and
abstract versions of the $\rho_{i,\lambda}$ and
$\sigma_{ij}$ isomorphisms. In particular, they induce an action
of $U_q(\Gg)$ on
$\BC\otimes_{\BZ}K_0(\CP(\lambda))$ and the resulting module is
isomorphic to the simple module of highest weight $\lambda$.
\end{thm}

\subsubsection{$2$-representations}
\label{se:geo2rep}
The action by convolution of $\CB(\Gamma)$
on $D(\Rep(\hat{\Gamma}))$ (cf Theorem \ref{th:catB})
induces a graded action on $\CP(\lambda)$.

\begin{thm}
\label{th:LequivP}
The graded action of $\CB(\Gamma)$ on $\CP(\lambda)$ extends to
a graded action of $\FA(\Gamma)$.
There is an equivalence of graded $2$-representations of $\FA(\Gamma)$
$$\bigl(\CL(\lambda)\otimes_k k^\Gamma\otimes_\BZ\bar{\BQ}_l\bigr)^i\,\mgr
\iso \CP(\lambda).$$
\end{thm}

\begin{proof}
We have $\End^\bullet(\Lambda_{\Rep_\nu})\simeq H^*(BG_\nu)$. 
Given $M,N\in\CP(\lambda)$, the space $\Hom^\bullet(M,N)$ is
a finitely generated $H^*(BG_\nu)$-module \cite[proof of Proposition 3.2.5]{Zh}.
By Theorem \ref{th:Zheng}, the functors $E_i$ and $F_i$ induce
an action of $\Gsl_2$ on 
$K_0(\CP(\lambda))\otimes_{\BZ[q^{\pm 1/2}]}\BC[q^{\pm 1/2}]/(q^{1/2}-1)$.
 We deduce
from Corollary \ref{cor:additiveK0} that we have a $2$-representation
of $\FA(\Gamma)$ (the grading can be forgotten to check that the maps
$\rho_{s,\lambda}$ are isomorphisms).

Let $i\in I$ and $0\le r\le \nu_i$.
We have $\Rep_{\hat{\nu}+r\alpha_i}=\prod_{j\not=i}
B\GL_{\nu_j}\times BP_r$, where $P_r$ is the maximal parabolic subgroup
of $\GL_{\nu_i}$ with Levi $\GL_r\times\GL_{\nu_i-r}$. The graded action
of $H^*(B\GL_{\nu_i})$ on $E_i^{(\nu_i)}(\Lambda_{\Rep_\nu})$ corresponds
to the action of $\bar{\BQ}_l[X_1,\ldots,X_{\nu_i}]^{\GS_{\nu_i}}$.
We deduce that the canonical map $P_{\lambda}\otimes_\BZ\bar{\BQ}_l\to 
H^*(BG_\nu)$ is an isomorphism. This proves the last part of the theorem.
\end{proof}

\begin{rem}
Zheng provides more generally a construction of tensor products of
simple representations, and the first part of
Theorem \ref{th:LequivP}, and its proof, generalize
immediately to that case: this provides graded $2$-representations
with Grothendieck group that tensor product of simple representations.
\end{rem}

Putting Theorems \ref{th:LequivB} and \ref{th:LequivP} together, we obtain

\begin{cor}
There is an equivalence compatible with the graded action
of $\FA(\Gamma)$
$$(\CB(\lambda)\otimes_k k^\Gamma\otimes_\BZ\bar{\BQ}_l)^i\mgr\iso\CP(\lambda)
.$$
\end{cor}

As a consequence, the indecomposable projective
modules for cyclotomic quiver Hecke algebras over 
$\bar{\BQ}_l\otimes_\BZ k^\Gamma$ correspond to the canonical
basis elements of $L(\lambda)$. When $\Gamma$ has type $A_n$ or
$\tilde{A}_n$, this is Ariki's Theorem (formerly, the Lascoux-Leclerc-Thibon
conjecture). Here, we used the geometry
of quiver varieties, which carry the same singularities
as flag varieties, in type $A$.

\begin{rem}
It would be interesting to extend Theorems \ref{th:catB} and
\ref{th:LequivP} to the case of coefficients $\BZ$ or $\BF_p$.
\end{rem}

Lauda has given an independent proof of the results of \S \ref{se:geo2rep}
for $\Gsl_2$: in this case, the geometry is that of flag varieties of
type $A$ \cite{Lau2}.
An earlier geometrical approach has been given by Cautis, Kamnitzer and
Licata for $\Gsl_2$, based on coherent sheaves on
cotangent bundles of flag varieties and compactifications of those
\cite{CauKaLi1}, later generalized to arbitrary $\Gamma$ \cite{CauKaLi2}.

Webster has given a presentation of our results and constructions
in \S \ref{se:geo2rep} and has used this to develop a categorification of the
Reshetikhin-Turaev invariants \cite{We1,We2}. He has also constructed
a counterpart of the categories $\CB(\lambda)$ for tensor products.

\end{document}